\documentclass[12pt]{amsart}

\usepackage[margin = 1in]{geometry}
\usepackage{amsfonts, amsmath,amscd, amssymb, cite, color, latexsym, mathrsfs, 
slashed, stmaryrd, verbatim, wasysym , tensor }
\usepackage[all]{xy}
\usepackage{graphicx}
\usepackage{hyperref}
\usepackage[applemac]{inputenc}

 \usepackage{tikz}

\newtheorem*{acknowledgements}{Acknowledgements}
\newtheorem{theorem}{Theorem}[section]

\newtheorem{lemma}[theorem]{Lemma}
\newtheorem{proposition}[theorem]{Proposition}
\newtheorem{remark}[theorem]{Remark}

\numberwithin{equation}{section}

\let\oldsqrt\sqrt
\def\sqrt{\mathpalette\DHLhksqrt}
\def\DHLhksqrt#1#2{%
\setbox0=\hbox{$#1\oldsqrt{#2\,}$}\dimen0=\ht0
\advance\dimen0-0.2\ht0
\setbox2=\hbox{\vrule height\ht0 depth -\dimen0}%
{\box0\lower0.4pt\box2}}

\newcommand{\bhs}[1]{\mathfrak B_{#1}}
\renewcommand{\tilde}{\widetilde}
\renewcommand{\bar}{\overline}
\renewcommand{\Re}{\operatorname{Re}}
\renewcommand{\hat}[1]{\widehat{#1}}

\newcommand\pa{\partial}

\newcommand{\lrp}[1]{\left( {#1} \right)}

\newcommand\cf{cf\@. }
\newcommand\ie{i\@.e\@. }

\newcommand\eps\varepsilon
\renewcommand\epsilon\varepsilon

\newcommand\bn{\mathbf{n}}

\newcommand\fc{\operatorname{fc}}

\font\k=cmr7

\newcommand {\di}{\mbox{\k dis}}

\newcommand\hD{\widehat{D}}
\newcommand\bW{\overline{W}}

\newcommand\Rep{\operatorname{Rep}}
\renewcommand\det{\operatorname{det}}

\newcommand\even{\operatorname{even}}

\DeclareMathOperator*{\FP}{\operatorname{FP}}

\newcommand\Hom{\operatorname{Hom}}
\newcommand\Gal{\operatorname{Gal}}
\newcommand\Id{\operatorname{Id}}
\renewcommand\Im{\operatorname{Im}}

\newcommand\phg{\operatorname{phg}}

\newcommand\bsl{\backslash}
\newcommand\rank{\operatorname{rank}}
\newcommand\level{\operatorname{level}}
\newcommand\Restr{\operatorname{Res}}

\renewcommand\Re{\operatorname{Re}}
\DeclareMathOperator*\Res{\operatorname{Res}}

\newcommand\Spec{\operatorname{Spec}}

\newcommand\Tr{\operatorname{Tr}}

\newcommand\tr{\operatorname{tr}}

\newcommand\pr{\operatorname{pr}}

\newcommand\GL{\operatorname{GL}}

\newcommand\spane{\operatorname{span}}

\newcommand\hE{\widehat{E}}
\newcommand\bw{\overline{w}}

\newcommand\paperintro%
        {%
         }
\newcommand\paperbody%
        {%
         }

\newcommand\bbA{\mathbb{A}}

\newcommand\bbC{\mathbb{C}}

\newcommand\bbF{\mathbb{F}}
\newcommand\bbG{\mathbb{G}}
\newcommand\bbH{\mathbb{H}}

\newcommand\bbK{F}

\newcommand\bbN{\mathbb{N}}

\newcommand\bbP{\mathbb{P}}
\newcommand\bbQ{\mathbb{Q}}
\newcommand\bbR{\mathbb{R}}

\newcommand\bbT{\mathbb{T}}

\newcommand\bbZ{\mathbb{Z}}

\renewcommand\bn{\overline{n}}

\newcommand\cA{\mathcal{A}}

\newcommand\cC{\mathcal{C}}
\newcommand\cD{\mathcal{D}}
\newcommand\cE{\mathcal{E}}

\newcommand\cH{\mathcal{H}}

\newcommand\cO{\mathcal{O}}

\newcommand\barf{\overline{f}}

\newcommand\fn{\mathfrak{n}}

\newcommand\tu{\widetilde{u}}
\newcommand\tv{\widetilde{v}}

\newcommand\td{\widetilde{d}}

\newcommand\SL{\operatorname{SL}}
\newcommand\SO{\operatorname{SO}}
\newcommand\Sym{\operatorname{Sym}}

\newcommand\bX{\overline{X}}

\newcommand\SU{\operatorname{SU}}
\newcommand\bz{\overline{z}}
\newcommand\vol{\operatorname{vol}}
\newcommand\free{\operatorname{free}}
\newcommand\tor{\operatorname{tor}}

\newcommand\tX{\widetilde{X}}

\DeclareMathAlphabet{\mathpzc}{OT1}{pzc}{m}{it}

\begin{document}

\title[Torsion in the cohomology of arithmetic groups]{On the growth
of torsion in the cohomology of some arithmetic groups of $\bbQ$-rank one}

\author{Werner M\"uller}
\address{Universit\"at Bonn, Mathematisches Institut, Endnicher Allee 60, D-53115 Bonn, Germany}
\email{mueller@math.uni-bonn.de}
\author{Fr\'ed\'eric Rochon}
\address{D\'epartement de Math\'ematiques, UQ\`AM}
\email{rochon.frederic@uqam.ca }

\maketitle

\begin{abstract}
Given a number field $\bbK$ with ring of integers $\cO_{\bbK}$, one can associate to any torsion free subgroup of $\SL(2,\cO_{\bbK})$ of finite index a complete Riemannian manifold of finite volume with fibered cusp ends.  For natural choices of flat vector bundles on such a manifold, we show that analytic torsion is identified with the Reidemeister torsion of the Borel-Serre compactification.  This is used to obtain exponential growth of torsion in the cohomology for sequences of congruence subgroups.

\end{abstract}

\tableofcontents

\section{Introduction}
Let $G$ be a  connected semi-simple
algebraic group over $\bbQ$ \cite{Milne2018}, \cite{Borel-Tits} and let
$G_\infty:=G(\bbR)$ be the group of real points of $G$. Let
$\Gamma\subset G(\bbQ)$ be an arithmetic subgroup. Then $\Gamma$ is a lattice
in $G_\infty$. Let $\rho\colon G\to \GL(V)$ be a rational
finite-dimensional complex representation. Suppose that there is a
$\Gamma$-invariant lattice $L\subset V$ with $L\otimes_\bbZ\bbC=V$. The
cohomology groups $H^j(\Gamma;L)$ are finitely generated abelian groups. Let
$H^j(\Gamma;L)_{\free}$ and $H^j(\Gamma;L)_{\tor}$ be the free and the torsion
subgroups, respectively. We have
$H^j(\Gamma;L)_{\free}\otimes_\bbZ\bbC\cong H^j(\Gamma;V)$. If the underlying
locally symmetric space has an algebro-geometric structure, the Langlands
conjectures predict that there are deep connections between the cohomology
groups $H^j(\Gamma;V)$ and the theory of automorphic forms and number theory. 
Recent results show that torsion classes
are also connected to number theory. Ash \cite{As} conjectured that for a
congruence subgroup $\Gamma\subset\SL(n,\bbZ)$, any Hecke eigenclass
$\xi\in H^\ast(\Gamma;\bbF_p)$ is attached to a continuous semisimple
Galois representation $\rho\colon \Gal(\bar\bbQ/\bbQ)\to \GL(n,\bbF_p)$ such
that Frobenius and Hecke eigenvalues match up. A more general version of this
conjecture has been proved by Scholze \cite{Sch}.
However, much less is known about the
structure of the torsion group, for example, its size. 
If $G$ is $\bbQ$-anisotropic, i.e., $\Gamma\backslash G_\infty$ is
compact, Bergeron and Venkatesh \cite{BV} obtained first results
concerning the growth of torsion, if $\Gamma$ varies in a
tower of lattices, and  they also formulated a  conjecture predicting the
growth of torsion in general.

To state the results and the conjecture we need to introduce some notation.
Let $K\subset G_\infty$ be a maximal compact subgroup. Let $\widetilde X
=G_\infty/K$ be the associated global Riemannian symmetric space. Let
${\mathfrak g}$ and ${\mathfrak k}$ be the Lie algebras of $G_\infty$ and
$K$, respectively. The fundamental rank $\delta(G)$ of $G$ is defined by
$\delta(G):=\rank({\mathfrak g}_\bbC)-\rank({\mathfrak k}_\bbC)$. As explained
by Bergeron and Venkatesh \cite{BV}, when $\delta(G)=1$, one expects for
arithmetic reasons that $H^\ast(\Gamma;L)$ should have a lot of torsion and
a small free part. This conjecture is supported by the following result proved
in \cite{BV} for anisotropic $G$. Assume that $\cdots\subset \Gamma_j\subset
\Gamma_{j-1}\subset\cdots\subset\Gamma$ is a decreasing sequence of congruence
subgroups such that $\cap_j\Gamma_j=\{1\}$. $L$ is called {\it strongly acyclic}
for the family $\{\Gamma_j\}$ if the Laplacians on $V$-valued $i$-forms on
$\Gamma_k\backslash\widetilde X$ are uniformly bounded away from 0 for all degrees
$i$ and all $\Gamma_k$. In this case, $H^\ast(\Gamma_k\setminus\widetilde{X};L)$ is a pure torsion group.
If $\delta(G)=1$, then by \cite[Theorem 1.4]{BV} one has
\begin{equation}\label{inequal1}
\liminf_{j\to\infty}\sum_q\frac{\log|H^q(\Gamma_j;L)_{\tor}|}{[\Gamma\colon\Gamma_j]}\ge c_{G,L}\vol(\Gamma\backslash\widetilde X)>0,
\end{equation}
where the sum is over the integers $q$ such that
$q+\frac{\dim (\widetilde X)+1}{2}$ is odd and $c_{G,L}>0$ is a constant that
depends only on $G$ and $L$. To establish the lower bound, Bergeron and
Venkatesh prove the following result
\begin{equation}\label{equl1}
\lim_{j\to\infty}\sum_q(-1)^{q+\frac{\dim (\widetilde X)+1}{2}}\frac{\log|H^q(\Gamma_j;L)_{\tor}|}
{[\Gamma\colon\Gamma_j]}=c_{G,L}\vol(\Gamma\backslash\widetilde X).
\end{equation}
The proof of \eqref{equl1} uses the equality of analytic torsion and
Reidemeister torsion. It follows from \eqref{inequal1}  that for some $q$,
$|H^q(\Gamma_j;L)_{\tor}|$ grows exponentially as $j\to\infty$.
Based on \eqref{equl1}, 
Bergeron and Venkatesh made a conjecture with a precise prediction of
the growth of torsion \cite[Conjecture 1.3]{BV} without any assumption on $L$.
The conjecture states that for each $q$ 
\begin{equation}\label{conj1}
\lim_{j\to\infty}\frac{\log|H^q(\Gamma_j;L)_{\tor}|}{[\Gamma\colon\Gamma_j]}
\end{equation}
exists and equals zero unless $\delta(G)=1$ and
$q=\frac{\dim(\widetilde X)+1}{2}$. In this case it equals
$c_{G,L}\vol(\Gamma\backslash\widetilde X)$ with $c_{G,L}>0$. All this is under
the assumption that the $\bbQ$-rank of $G$ is 0.

Since many important arithmetic groups are not co-compact, it is desirable to
extend these results to groups $G$ with $\bbQ$-rank $>0$.
In \cite{AGMY}
the authors made comprehensive computations of torsion subgroups in
$H^j(\Gamma,\bbZ)$, where $\Gamma\subset G(\bbQ)$ is an arithmetic subgroup
for $G=\GL_n/\bbQ$, $n=3,4,5$, or $G=\GL_2$ over specific number fields for
which $\delta(G)=1$ or $2$. They use their computations to extend the
conjecture \eqref{conj1} of Bergeron und Venkatesh by removing the restriction
to the cocompact case and allowing any growth of level, not just in a tower.
These are Conjectures 7.1 and 7.2 in \cite{AGMY}. In particular, Conjecture 7.2
predicts that for a family $\{\Gamma_j\}_{j\in\bbN}$ of congruence subgroups in a
fixed arithmetic groups $\Gamma$ with $\level(\Gamma_j)\to\infty$,
\begin{equation}
\liminf_{j\to \infty}\frac{\log|H^q(\Gamma_j;L)_{\tor}|}{[\Gamma\colon\Gamma_j]}
\end{equation}
exists. If $G$ has $\bbQ$-rank $>0$, then the lim-inf equals zero unless
$\delta(G)=1$ and $q$ is the top degree of the cuspidal range.

The first step
beyond $\bbQ$-rank 0 is the case of hyperbolic manifolds of finite volume which
has been treated in \cite{Pfaff2014b}, \cite{MR2}. This is the $\bbR$-rank 1
case. The method of \cite{BV} is based on the equality of analytic torsion and
Reidemeister torsion \cite{Cheeger1979}, \cite{Muller1978}, \cite{Muller1993}.
This equality is not available in the non-compact case.
However, for hyperbolic manifolds of finite volume there is a formula with an
explicit defect term \cite{MR1}, whose asymptotic behavior can be controlled for
a family of congruence subgroups $\{\Gamma_j\}_{j\in\bbN}$ of a fixed arithmetic
group $\Gamma$. There is another obstacle if one wants to apply the result
to deduce a formula similar to \eqref{equl1}. Even if $L$
strongly acyclic, the cohomology $H^\ast(\Gamma;L\otimes\bbC)$ does not
vanish. Only the interior cohomology vanishes. In general, there is cohomology
coming from the boundary of the Borel-Serre compactification. This is the
Eisenstein cohomology which gives rise to a non-trivial regulator in the
expression of the Reidemeister torsion in terms of the order of the torsion
subgroup in the cohomology $H^\ast(\Gamma;L)$ \cite[$\S$ 2]{BV}. We were able
to cope with these problems and established a lower bound similar to
\eqref{inequal1} \cite{MR2}.

In this paper we consider $\bbQ$-rank 1 cases with $\bbR$-rank $>1$.
The corresponding locally symmetric space is a manifold with fibered cusps.
The semi-simple group $G$ is defined as follows.
Let $\bbK$ be a number field of degree $d_{\bbK}$ over $\bbQ$.
Let $\cO_{\bbK}$ denote the ring of algebraic integers of $\bbK$. We consider
$\SL(2)$ as an algebraic group over $\bbK$. Let $G_0=\SL(2)/\bbK$ and let
\begin{equation}\label{restr-scal}
G=\Restr_{\bbK/\bbQ}(G_0)
\end{equation}
be the algebraic group, which is obtained from $G_0$ by restriction
of scalars \cite{We}. Then $G$ is a semi-simple
algebraic group over $\bbQ$. Moreover, we have
\[
G_\infty=G(\bbR)=\prod_{v|\infty} \SL(2,\bbK_v),\quad G(\bbQ)=\SL(2,\bbK).
\]
Let $\sigma_1,\ldots, \sigma_{r_1}$ be the embeddings of $\bbK$ in $\bbR$ and let $\tau_1,\overline{\tau}_1,\ldots,\tau_{r_2},\overline{\tau}_{r_2}$ denote the remaining embeddings of $\bbK$ in $\bbC$, so that $d_{\bbK}=r_1+2r_2$. Then
$$
G_{\infty}= \SL(2,\bbK\otimes_{\bbQ}\bbR)= \SL(2,\bbR)^{r_1}\times \SL(2,\bbC)^{r_2}
$$
and $K_\infty=\SO(2)^{r_1}\times \SU(2)^{r_2}$ is a maximal compact subgroup.  The
corresponding symmetric space equals
$$
   \tX:= G_{\infty}/K_\infty= (\bbH^2)^{r_1}\times (\bbH^3)^{r_2}.
$$
Let $\Gamma\subset\SL(2,\cO_{\bbK})$ be a torsion free subgroup of finite index.  Then $\Gamma$ is a discrete subgroup of $G_{\infty}$ via the embedding $\iota: \Gamma\to \SL(2,\bbR)^{r_1}\times \SL(2,\bbC)^{r_2}$ defined by
\begin{multline}
     \left( \begin{array}{cc} \alpha & \beta \\ \gamma & \delta \end{array} \right)\mapsto \\
     \left(\left( \begin{array}{cc} \sigma_1(\alpha) & \sigma_1(\beta) \\ \sigma_1(\gamma) & \sigma_1(\delta) \end{array} \right),\ldots, \left( \begin{array}{cc} \sigma_{r_1}(\alpha) & \sigma_{r_1}(\beta) \\ \sigma_{r_1}(\gamma) & \sigma_{r_1}(\delta) \end{array} \right), \left( \begin{array}{cc} \tau_1(\alpha) & \tau_1(\beta) \\ \tau_1(\gamma) & \tau_1(\delta) \end{array} \right), \ldots, \left( \begin{array}{cc} \tau_{r_2}(\alpha) & \tau_{r_2}(\beta) \\ \tau_{r_2}(\gamma) & \tau_{r_2}(\delta) \end{array} \right)\right)
\label{emb.1}\end{multline}
and the quotient
$$
   X= \Gamma\setminus \widetilde{X}
$$
is a manifold.  As described in \cite{Borel1974}, it comes with a natural
metric.  First, the invariant metric one should consider on
$\tX=(\bbH^2)^{r_1}\times (\bbH^3)^{r_2}$ is 
\begin{equation}
    \widetilde{g}=  \sum_{i=1}^{r_1}\frac{dx_i^2+ dy_i^2}{y_i^2}+ 2 \sum_{j=1}^{r_2} \frac{|dz_j|^2+ dt_j^2}{t_j^2}
\label{mc.3}\end{equation}
using the upper half-plane and the upper half-space models for $\bbH^2$ and $\bbH^3$.
Since $\Gamma$ acts by isometries, $\widetilde{g}$ descends to a metric $g$ on $X$.  The manifold $X$ has the homotopy type of a compact manifold with boundary with each boundary component corresponding to a fibered cusp end. 
 If $\bbP^1(\bbK)$ is the projective line of the number field $\bbK$, then the fibered cusp ends are in bijection with $\Gamma\setminus \bbP^1(\bbK)$.  Let in fact $\mathfrak{P}_{\Gamma}\subset \bbP^1(\bbK)$ be a set of representatives for the classes in $\Gamma\setminus \bbP^1(\bbK)$, so that $\mathfrak{P}_{\Gamma}$ naturally corresponds to the set of fibered cusp ends of $X$.  Without loss of generality, we will assume that $[1:0]\in \mathfrak{P}_{\Gamma}$.    As described in \cite[no.28, p.69]{Shimizu}, when $\Gamma= \SL(2,\cO_{\bbK})$, the fibered cusp ends are also naturally identified with the ideal classes of $\bbK$.
 
 Thus, $X$ has a natural compactification as a manifold with boundary $\bX$ with boundary components $Y_{\eta}$ labeled by $\eta\in \mathfrak{P}_{\Gamma}$.  Each boundary component comes with a natural fiber bundle
 \begin{equation}
    \phi_{\eta}: Y_{\eta}\to S_{\eta}
 \label{int.1}\end{equation}
with base $S_{\eta}$ and fibers $\phi_{\eta}^{-1}(y)$ diffeomorphic to tori,
$$
        S_{\eta}\cong \bbT^{r_1+r_2-1}, \quad \phi_{\eta}^{-1}(y)\cong \bbT^{r_1+2r_2} \quad \forall y\in S_{\eta}.
$$
In the fibered cusp end associated to $\eta$, the metric $g$, up to scaling, takes the form
\begin{equation}
    g= \frac{dr^2}{r^2}+ \phi_{\eta}^*g_{S_{\eta}}+ r^{-2}\kappa, \quad r\in (R_{\eta},\infty),
\label{int.2}\end{equation}
for some $R_{\eta}>0$, where $g_{S_{\eta}}$ is a flat metric on $S_{\eta}$ and $\kappa$ is a $2$-tensor inducing a flat metric on each fiber of \eqref{int.1} in such a way that $\phi^*_{\eta}g_{S_{\eta}}+\kappa$ is a flat metric on $Y_{\eta}$ making \eqref{int.1} a Riemannian submersion with respect to the metric $g_{S_{\eta}}$ on the base.  Moreover, the natural connection of \eqref{int.1} induced by the metric $\phi_{\eta}^*g_{S_{\eta}}+\kappa$ has trivial curvature in the sense of \cite[$\S$ 10.1]{BGV}.  However, the second fundamental form of this Riemannian submersion does not vanish in general, which by \cite[Lemma~10.3]{BGV} corresponds to the fact that the family of fiberwise metrics is not necessarily parallel with respect to the natural connection.    From \eqref{int.2}, it can also be inferred that $g$ is a complete metric of finite volume.

There are natural flat vector bundles associated to the Riemannian manifold $(X,g)$.  To describe them, let $V$ be the standard representation of $\SL(2,\bbR)$ and let $W$ be the standard representation of $\SL(2,\bbC)$.  Denote also by $\bW$ the complex conjugate of $W$, that is, the dual representation.  For $q\in \bbN$, let $V_q$ be the $q$th symmetric power of $V$, $W_q$ be the $q$th symmetric power of $W$ and $\bW_q$ be the $q$th symmetric power of $\bW$.  Then for $m=(m_1,\ldots,m_{r_1})\in\bbN^{r_1}_0$ and $n=(n_1, \bn_1,\ldots,n_{r_2},\bn_{r_2})\in\bbN^{2r_2}_0$, the tensor product representation $(V_{m_1}\otimes\cdots \otimes V_{m_{r_1}})\otimes (W_{n_1}\otimes \bW_{\bn_1}\otimes\cdots\otimes W_{n_{r_2}}\otimes \bW_{\bn_{r_2}})$ with map
\begin{equation}\label{stand-repr}
\begin{split}
\varrho_{m,n}: \SL(2,\bbR)^{r_1}\times &\SL(2,\bbC)^{r_2}\to \\
&\GL((V_{m_1}\otimes\cdots \otimes V_{m_{r_1}})\otimes (W_{n_1}\otimes\bW_{\bn_1}\otimes\cdots\otimes W_{n_{r_2}}\otimes \bW_{\bn_{r_2}}))
\end{split}
\end{equation}
is an irreducible representation of $G_{\infty}$.  There is a natural flat vector bundle $E_{m,n}\to X$, associated to $\varrho_{m,n}|_\Gamma$
$$
     E_{m,n}= \Gamma\setminus (\tX\times (V_{m_1}\otimes\cdots \otimes V_{m_{r_1}})\otimes (W_{n_1}\otimes \bW_{\bn_1}\otimes\cdots\otimes W_{n_{r_2}}\otimes \overline{W}_{\bn_{r_2}})).
$$
By \cite[Sect. 3]{MM1963} this bundle can be equipped with a canonical
bundle metric $h$, which is defined by an admissible inner product in the
representation space \cite[Lemma 3.1]{MM1963}.
The flat connection is not unitary with respect to this metric, but it is at
least unimodular.

One central goal of the present paper is to study the analytic torsion $T(X,E_{m,n},g,h)$ of $(X,E_{m,n},g,h)$ as defined in \cite{ARS1} and to relate it with the Reidemeister torsion of $(\bX, E_{m,n})$.  Recall that on closed manifolds, such a relation was conjectured by Ray and Singer \cite{Ray-Singer} and subsequently established independently by Cheeger \cite{Cheeger1979} and the first author \cite{Muller1978} when the flat connection is unitary.  This was extended to unimodular flat connections by the first author in \cite{Muller1993}, while the general case was treated by Bismut and Zhang \cite{Bismut-Zhang}.

For non-compact manifolds, some relation between analytic torsion and Reidemeister torsion has been obtained on manifolds with cylindrical ends by Hassell \cite{Hassell} using the surgery pseudodifferential calculus of Mazzeo and Melrose \cite{mame1,hmm}.  Developping instead a surgery pseudodifferential calculus adapted to fibered cusp ends, a corresponding result was obtained in \cite{ARS1} when the Riemannian manifold has fibered cusp ends with a sharper result in \cite{ARS2} when there are only cusp ends, that is, fibered cusp ends whose bases are points. Following \cite{Hassell}, the strategy of \cite{ARS1,ARS2} consists in considering the double 
$$
   M= \bX \bigcup_{\pa\bX} \bX
$$
of $\bX$ obtained by gluing two copies of $\bX$ along their boundaries and to consider a family of smooth metrics $g_{\epsilon}$ on $M$ degenerating to the fibered cusp metric of interest on each copy of $X$ in $M$ as $\epsilon\searrow 0$.  Through a uniform construction of the resolvent and of the heat kernel as $\epsilon\searrow 0$, it was then possible to describe the asymptotic behavior of analytic torsion on $M$ as $\epsilon\searrow 0$ and identify one of the limiting terms as analytic torsion on each copy of $X$ in $M$.  On the other hand, through the formula of Milnor \cite{Milnor1966} for Reidemeister torsions appearing in a short exact sequence of complexes, one can relate the Reidemeister torsion of $M$ with the one of $X$ via a suitable Mayer-Vietoris long exact sequence in cohomology.  Combining with the result of \cite{Muller1993} on $M$, one can then obtain a relation between analytic torsion and Reidemeister torsion.  

In general, the limiting behavior of analytic torsion as $\epsilon\searrow 0$ involves many terms,  some of which possibly not very explicit.  However, assuming that the base of each fibered cusp ends is even dimensional with the vector bundle being acyclic in each fiber, many of these terms vanish, yielding the simple formula of \cite[Theorem~1.3]{ARS1}.  In the cusp case, it was possible in \cite{ARS1} to replace the acyclicity condition by a much weaker Witt condition.  Moreover, in this latter case, instead of  the Reidemeister torsion of $\bX$, what appears in the formula is the intersection $R$-torsion of Dar \cite{Dar} associated to the stratified space obtained from $\bX$ by collapsing each of its boundary component onto a point.   

In both \cite{ARS1} and \cite{ARS2}, one important restriction is that the bundle metric $h$ of the flat vector bundle is required to be smooth on the compactification $\bX$, excluding in particular the natural bundle metric of Matsushima and Murakami \cite{MM1963} on locally homogeneous spaces.  This is the starting point of \cite{MR1}, where it was shown that the strategy of \cite{ARS1,ARS2}, suitably adapted, but still using the same analytical tools, works to obtain a relation between analytic torsion and Reidemeister torsion on finite volume hyperbolic manifolds when the flat vector bundle comes from representation theory and is equipped with the bundle metric of \cite{MM1963}.  

The present paper expand further in this direction by obtaining the following result for the Riemannian manifold with fibered cusp ends $(X,g)$ described above.
\begin{theorem}
Let $\bbK$ be a number field such that $r_2$ is odd (\ie $\dim X$ is odd) and $r_1+r_2>2$.  If $r_1=0$ suppose also that $\bn_1=\cdots=\bn_{r_2}=0$ and $n\ne0$.  In this case, 
$$
  T(X,E_{m,n},g,h)= \tau(\bX,E_{m,n},\mu_X),
$$
where $\tau(\bX,E_{m,n},\mu_X)$ is the Reidemeister torsion of $(\bX,E_{m,n})$ associated to $\mu_X$, an explicit choice of basis of $H^*(\bX;E_{m,n})$ described in \eqref{se.8} below.
\label{int.3}\end{theorem}
\begin{remark}
When $r_1>0$, notice that our result applies to the trivial line bundle $E_{0,0}$.  
\end{remark}

Compared to \cite{ARS1}, notice that we no longer require that $\dim S_{\eta}$ be even, but only that $\dim S_{\eta}>1$.  This improvement is relying on the fact that the metric $g$ is exactly given by the model \eqref{int.2} in each fibered cusp end and that \eqref{int.1} has vanishing curvature as a Riemannian submersion.  The case where $r_1=0$ and $r_2=1$, so in particular with $\dim S_{\eta}=0$, is not covered by Theorem~\ref{int.3}, but there is a corresponding result in this case, namely \cite[Theorem~7.1]{MR1}, this time however with an explicit defect term depending on $E_{m,n}$.  On the other hand, if $r_1=r_2=1$, that is, when $\dim S_{\eta}=1$, there seems to be a defect term as well, but hard to determine or estimate with the current techniques. 

Our next goal is to apply this result to study the growth of torsion in
the cohomology of the arithmetic groups $\Gamma$ as described above. To this
end, let

\begin{equation}\label{repres1}
\varrho\colon G\to \GL(V)
\end{equation}
be a $\bbQ$-rational representation of $G$ on a finite dimensional 
$\bbQ$-vector space $V$.  Since $\varrho$ is a $\bbQ$-rational representation of $G$ on $V$,
we know by \cite[Chapter VII, Prop 5.1, p.400]{Milne2011} that there exists a lattice $\Lambda\subset V$, i.e., $V=\Lambda\otimes_\bbZ\bbQ$,
which is invariant under $\Gamma$.  Let $\varrho_\infty$
be the representation of $G_\infty$ on $V_\bbC=V\otimes_\bbQ\bbC$ obtained by restriction
of the representation of $G(\bbC)$ to $G_\infty$.  Let
$E\to\Gamma\bsl\widetilde X$ be the flat vector bundle defined by
$\varrho_\infty|_{\Gamma}$.   In analogy with the compact case, we call
$\Lambda$ a  $L^2$-acyclic $\Gamma$-module if $E$ has trivial $L^2$-cohomology, namely $H^\ast_{(2)}(\Gamma\bsl\widetilde X; E)=0$.
If in fact
$H^\ast(\Gamma\bsl\widetilde X; E)=0$, we say that $\Lambda$ is an acyclic $\Gamma$-module.

Now, if $\mathfrak{n}\subset \cO_{\bbK}$ is an ideal, we can take $\Gamma$ to be the principal congruence subgroup of $\SL(2,\cO_{\bbK})$ of level $\mathfrak{n}$ defined by
\begin{equation}
   \Gamma(\mathfrak{n}):= \left\{  \lrp{ \begin{array}{cc} a & b \\ c & d \end{array}  }\in \SL(2,\cO_{\bbK}) \; | \; a-1,d-1,b,c\in \mathfrak{n}  \right\}.
\label{eg.6}\end{equation}

This can be seen as a subgroup of $G_{\infty}$ via the embedding \eqref{emb.1}.
Let $\{\fn_i\}_{i\in\bbN}$ be a sequence of ideals in $\cO_{\bbK}$ satisfying
\begin{equation}\label{ideals}
\fn_i\subset\fn_1,\; i\in\bbN,\quad \text{and}\quad N(\fn_i)\to \infty
\;\text{as}\;i\to\infty,
\end{equation}
where $N(\fn_i):=[\cO_{\bbK}:\fn_i]$ is the absolute norm of $\fn_i$. 
The associated sequence of principal congruence subgroups
$\Gamma(\fn_i)\subset\SL(2,\cO_{\bbK})$, $i\in\bbN$, satisfies
\[
\Gamma(\fn_i)\subset\Gamma(\fn_1),\; i\in\bbN,\quad\text{and}\quad
[\Gamma(\fn_1)\colon\Gamma(\fn_i)]\to\infty\;\text{as}\;i\to\infty.
\]
Let $X_i:= \Gamma(\mathfrak{n}_i)\bsl \widetilde{X}$. 
Assuming that $\Gamma(\mathfrak{n}_1)$ is torsion free will ensure that $X_i$ is a smooth manifold for each $i$. Let
\[
L_i=\Gamma(\mathfrak{n}_i)\backslash(\widetilde{X}\times\Lambda)
\]
be the local system of free $\bbZ$-modules on $X_i$ associated to $\Lambda$.
Let $E_i\to X_i$ be the flat vector bundle, which is defined by
$\varrho_\infty|_{\Gamma(\mathfrak{n}_i)}$.

Combining Theorem~\ref{int.3} with \cite{Matz-Muller} allows us to conclude the following result about the size of $H^*(X_i;L_i)_{\tor}$; see Theorems~\ref{eg.9} and \ref{eg.24} below for further details.
\begin{theorem}\label{main-thm}
Let $\bbK$ be a number field with $r_2=1$ and $r_1>1$. Let $\varrho$ be a
$\bbQ$-rational representation of $G$ on $V$.  Let $\Lambda\subset V$
be an  arithmetic $\Gamma(\fn_1)$-module and let $L_i$ be
the local system over $X_i$, associated to $\Lambda$.  Suppose that $\varrho_{\infty}$ decomposes into a sum of irreducible representations $\tau_j$ such that $\tau_j\ncong \tau_j\circ \vartheta$, where $\vartheta$ is the standard Cartan involution of $G_{\infty}$ with respect to $K_{\infty}$.
If $\Lambda$ is an acyclic $\Gamma(\mathfrak{n}_i)$-module for each $i$, then for the sequence of principal congruence subgroups
$\{\Gamma(\fn_i)\}_{i\in\bbN}$ we have
$$
\liminf_{i\to \infty} \sum_{q+r_1 \;\; \operatorname{even}} \frac{\log |H^q(\Gamma(\mathfrak{n}_i);\Lambda)|}{[\Gamma(\mathfrak{n}_1):\Gamma(\mathfrak{n}_i)]}\ge 2(-1)^{r_1+1}t^{(2)}_{\widetilde{X}}(\varrho_{\infty})\vol(X_1)>0,
$$
where $t^{(2)}_{\widetilde{X}}(\varrho_{\infty})$ is the $L^2$-torsion \cite{Lott1992,Mathai1992} associated to $\widetilde{X}$ and $\varrho_{\infty}$.
If we drop the  assumption  that $\Lambda$ is acyclic, but assume that the natural isomorphism $V^*\cong V$ induces an isomorphism $\Lambda^*\cong \Lambda$, then 
$$
\liminf_{i\to \infty} \sum_{q+r_1 \;\; \operatorname{even}}
\frac{\log |H^q(\Gamma(\mathfrak{n}_i);\Lambda)_{\tor}|}{[\Gamma(\mathfrak{n}_1):\Gamma(\mathfrak{n}_i)]}\ge
(-1)^{r_1+1}t^{(2)}_{\widetilde{X}}(\varrho_{\infty})\vol(X_1)>0.
$$ 
\label{int.4}\end{theorem}
\begin{remark}
As explained in Proposition~\ref{prop-repr1}, the approach of Bergeron and Venkatesh \cite[8.1]{BV} yields many examples of acyclic $\Gamma(\mathfrak{n}_i)$-modules to which Theorem~\ref{int.4} applies. 
\end{remark}
\begin{remark}
If $\Lambda$ is  $L^2$-acyclic but not self-dual, we can apply Theorem~\ref{main-thm} to $\Lambda\oplus\Lambda^*\subset V\oplus V^*$ to obtain exponential growth of torsion in cohomology.
\label{int.4b}\end{remark}

\begin{remark}
The condition $r_2=1$ is important in the theorem, since when $r_2>1$, the fundamental rank of $G_{\infty}$ is not equal to $1$, so
$t^{(2)}_{\widetilde{X}}(\varrho)=0$ by \cite[Proposition~5.2]{BV}.
\label{int.5}\end{remark}
\begin{remark}
The case $r_2=1$ with $r_1=0$ is not covered by this result, but in this case the exponential growth of torsion was obtained in \cite{Pfaff2014b}, see also  \cite[Corollary~1.6]{MR2}.
\label{int.6}\end{remark}

The paper is organized as follows.  In \S~\ref{gfce.0}, we give a detailed geometric description of the metric $g$ and the bundle metric $h$ in the fibered cusp ends.  This is used in \S~\ref{hdR.0} to study the Hodge-deRham operator and its asymptotic behavior under degeneration to fibered cusp metrics, so that in \S~\ref{cdat.0}, we can obtain the corresponding asymptotic behavior of analytic torsion.  This is combined in \S~\ref{sev.0} with the fine understanding of the asymptotic behavior of small eigenvalues of the Hodge-deRham operators under a cusp degeneration to prove Theorem~\ref{int.3}.  In \S~\ref{sect-acyclic}, we explain how to construct acyclic $\Gamma$-modules. This provides examples in \S~\ref{eg.0} to which we can apply our result to deduce Theorem~\ref{int.4} about the exponential growth of torsion in cohomology. 

\begin{acknowledgements}
The authors are grateful to an anonymous referee and to Hugo Chapdelaine for helpful comments.
The second author acknowledges support from NSERC.
\end{acknowledgements}

\section{Geometry of the fibered cusp ends} \label{gfce.0}

 Recall from the introduction that the fibered cusp ends are identified with $\Gamma\setminus\bbP^1(\bbK)$ and $\mathfrak{P}_{\Gamma}\subset \bbP^1(\bbK)$ is a fixed subset of representatives that includes $[1:0]$. Let us first describe the cusp end corresponding to $[1:0]\in \mathfrak{P}_{\Gamma}$.  Thus, let $B\subset \SL(2)$ be the standard Borel subgroup.  Set 
$$
   B_{\infty}= B(\bbK\otimes_{\bbR}\bbC)= B(\bbR)^{r_1}\times B(\bbC)^{r_2}    
$$
with 
$$
 B(\bbR)= \left\{ \left(\begin{array}{cc} \lambda & x \\ 0 & \lambda^{-1}  \end{array} \right) \; | \; \lambda \in \bbR^*, x\in\bbR\right\}
$$
and 
$$
B(\bbC)= \left\{ \left(\begin{array}{cc} \mu & z \\ 0 & \mu^{-1}  \end{array} \right) \; | \; \mu \in \bbC^*, z\in\bbC\right\}.
$$
Let $\nu: B_{\infty}\to (\bbR^+)^*$ be defined by 
\begin{equation}
  \nu(b)= \left(\prod_{i=1}^{r_1}|\lambda_i|\right)\left( \prod_{j=1}^{r_2}|\mu_j|^2\right)
\label{bdf.1}\end{equation}
for 
$$
b= \left(  \left(\begin{array}{cc} \lambda_1 & x_1 \\ 0 & \lambda_1^{-1}  \end{array} \right),\ldots, \left(\begin{array}{cc} \lambda_{r_1} & x_{r_1} \\ 0 & \lambda_{r_1}^{-1}  \end{array} \right), \left(\begin{array}{cc} \mu_1 & z_1 \\ 0 & \mu_1^{-1}  \end{array} \right),\ldots, \left(\begin{array}{cc} \mu_{r_2} & z_{r_2} \\ 0 & \mu_{r_2}^{-1}  \end{array} \right)\right)\in B_{\infty}.
$$
If 
$$
   B_{\infty}(1)= \left\{ b\in B_{\infty} \; | \; \nu(b)=1  \right\},
$$
then 
$$
    B_{\infty}(1)\cap K= B_{\infty}\cap K, \quad B_{\infty}(1)\cap \Gamma= B_{\infty}\cap \Gamma,
$$
and 
$$
     Y:= B_{\infty}\cap \Gamma \setminus B_{\infty}(1)/ B_{\infty}\cap K
$$
is the cross-section of the cusp end associated to $[1:0]\in\mathfrak{P}_{\Gamma}$, where $K=\SO(2)^{r_1}\times \SU(2)^{r_2}$ is the standard maximal compact subgroup of $G_{\infty}$.  In fact, by definition,
$$
  B_{\infty}\cap \Gamma\subset B_{\infty}\cap \SL(2,\cO_{\bbK})= \left\{  \iota\left(  \left(\begin{array}{cc} \lambda & x_1 \\ 0 & \lambda^{-1}  \end{array} \right)\right) \; | \; \lambda, \lambda^{-1}, x_1\in \cO_{\bbK} \right\},
$$
so $\lambda\in \cO^*_{\bbK}$.  If $N_{\bbK/\bbQ}: \bbK\to \bbQ$ is the norm defined by
$$
  N_{\bbK/\bbQ}(k)= \left(\prod_{i=1}^{r_1} \sigma_i(k)\right) \left(\prod_{j=1}^{r_2} |\tau_j(k)|^2\right), \quad k\in \bbK,
$$
then since $\lambda$ is a unit, one has that
$$
     N_{\bbK/\bbQ}(\lambda)=\pm 1,
$$
so that 
$$
    \nu\circ\iota \left(  \left(\begin{array}{cc} \lambda & x_1 \\ 0 & \lambda^{-1}  \end{array} \right)\right)= |N_{\bbK/\bbQ}(\lambda)|=1,
$$
confirming that $B_\infty\cap \Gamma= B_{\infty}(1)\cap \Gamma$.

The other cusp ends admit a similar description.  First, if $\Gamma=\SL(2,\cO_{\bbK})$, the cusp ends are identified with the ideal classes $c_1,\ldots, c_h$ of $\bbK$.  If $\mathfrak{a}$ is a representative in the ideal class $c_{\nu}$, pick $a,b\in \cO_{\bbK}$ such that $\mathfrak{a}$ is the ideal in $\cO_{\bbK}$ generated by $a$ and $b$.  Without loss of generality, we can assume that $a\ne 0$.  Then the cusp end associated to $c_{\nu}$ corresponds to the point 
$$
    \eta:= [a:b] \in \bbP^1(\bbK).
$$
Since the element
$$
   \lrp{\begin{matrix} a & 0 \\ b & a^{-1}  \end{matrix} } \in \SL(2,\bbK)
$$
sends $[1:0]$ onto $[a:b]$, the parabolic subgroup $P_{\eta}$ associated to $\eta$ is given by 
$$
    P_{\eta}=  \lrp{\begin{matrix} a & 0 \\ b & a^{-1}  \end{matrix} }^{-1} P_0  \lrp{\begin{matrix} a & 0 \\ b & a^{-1}  \end{matrix} },
$$
where 
$$
   P_0= B(\bbK)= \left\{  \lrp{\begin{matrix} \lambda & z \\ 0 & \lambda^{-1}  \end{matrix} } \; | \; \lambda\in \bbK^*, \; z\in \bbK \right\}
$$
is the parabolic subgroup of $[1:0]\in \bbP^{1}(\bbK)$.  In particular, a computation shows that 
$$
 \lrp{\begin{matrix} a & 0 \\ b & a^{-1}  \end{matrix} } \Gamma\cap P_{\eta} \lrp{\begin{matrix} a^{-1} & 0 \\ -b & a  \end{matrix} } = \left\{   \lrp{\begin{matrix} \lambda& w\\ 0 & \lambda^{-1}  \end{matrix} } \; | \; \lambda\in \cO_{\bbK}^*, \; w\in \mathfrak{a}^{-2}  \right\},
$$
where
$$
   \mathfrak{a}^{-2}= \left\{ w\in \bbK \; | \; uw\in \cO_{\bbK} \; \forall u \in \mathfrak{a}^2 \right\}.
$$
Thus, the cross-section of the cusp end associated to $\eta\in \bbP^1(\bbK)$ is
\begin{equation}
  Y_{\eta}:= \Gamma_{\mathfrak{a}}\setminus B_{\infty}(1)/B_{\infty}\cap K
\label{mc.1}\end{equation}
with
$$
    \Gamma_{\mathfrak{a}}:= \left\{ \iota\lrp{\lrp{ \lrp{\begin{matrix} \lambda & w \\ 0 & \lambda^{-1}  \end{matrix} }}}  \; | \; \lambda\in \cO_{\bbK}^*, \; w\in \mathfrak{a}^{-2} \right\} \subset B_{\infty}(1).  
$$
More generally, if $\Gamma$ is a subgroup of $\SL(2,\cO_{\bbK})$, the cusp ends are of the form 
\begin{equation}
Y_{\eta}= \Gamma_{\eta} \setminus B_{\infty}(1) / B_{\infty}\cap K
\label{mc.2}\end{equation}
with $\Gamma_{\eta}$ a finite index subgroup of $\Gamma_{\mathfrak{a}}$ for some ideal $\mathfrak{a}$ representing an ideal class of $\bbK$.

Denote the unipotent radical of $B$ by $N$.  Then $T=B/N$ is a split torus of dimension $r_1+r_2-1$.  Set 
$$
   T_{\infty} := T(\bbK\otimes_{\bbQ}\bbR)= T(\bbR)^{r_1}\times T(\bbC)^{r_2} 
$$
with
$$
  T(\bbR):= \left\{   \left(\begin{array}{cc} \lambda & 0 \\ 0 & \lambda^{-1}  \end{array} \right) \; | \; \lambda\in\bbR^* \right\} \quad\mbox{and} \quad T(\bbC):=\left\{ \left(\begin{array}{cc} \mu & 0 \\ 0 & \mu^{-1}  \end{array} \right) \; | \;  \mu\in\bbC^* \right\}.
$$
Let $K_T$ (respectively $\Gamma_{\eta,T}$) be the image of $K\cap B_{\infty}$ (respectively $\Gamma_{\eta}$) under the projection $p_{\infty}: B_{\infty}\to T_{\infty}$.  Then  
\begin{equation}
   K_T= (T(\bbR)\cap \SO(2))^{r_1}\times (T(\bbC)\cap\SU(2))^{r_2}
\label{nf.1}\end{equation}
with
$$
T(\bbR)\cap\SO(2)= \{\pm \Id\} \quad \mbox{and}  \quad  T(\bbC)\cap\SU(2)= \left\{ \left( \begin{matrix} e^{i\theta} & 0 \\ 0 & e^{-i\theta}  \end{matrix} \right) \; |\; \theta\in\bbR \right\},
$$
while
$$
  \Gamma_{\eta,T}\subset \left\{   \iota\left(\left(\begin{array}{cc} u & 0 \\ 0 & u^{-1}  \end{array} \right)\right) \; | \; u\in \cO_{\bbK}^* \right\}  .
$$
In particular, \eqref{nf.1} shows that
$$
   T_{\infty}/K_T\cong (((\bbR^+)^*)^{r_1+r_2}.
$$
If we set 
$$
   T_{\infty}(1):= \{ t\in T_{\infty} \; | \; \nu(t)=1\},
$$
this means $T_{\infty}(1)/K_T$ is identified with the kernel of the group homomorphism induced by $\nu$,
$$
    \begin{array}{llcl} \nu: & (((\bbR^+)^*)^{r_1+r_2} & \to & (\bbR^+)^* \\
      & (\lambda_1,\ldots,\lambda_{r_1},\mu_1,\ldots,\mu_{r_2}) & \mapsto & \left( \prod_{i=1}^{r_1} \lambda_i\right) \left( \prod_{j=1}^{r_2} \mu_j^2\right),
      \end{array}
$$
namely
$$
    T_{\infty}(1)/K_T\cong (((\bbR^+)^*)^{r_1+r_2-1}.
$$
By the Dirichlet's Unit Theorem (see for instance \cite[Theorem~38]{Marcus}), 
$$
      \Gamma_{\eta,T}\setminus T_{\infty}(1) /K_T \cong \bbT^{r_1+r_2-1}.
$$
is a real torus of dimension $r_1+r_2-1$.  
The projection map $p_{\infty}: B_{\infty}\to T_{\infty}$ induces a projection
\begin{equation}
   \phi_{\eta}: Y_{\eta}\to \Gamma_{\eta,T}\setminus T_{\infty}(1)/ K_T \cong \bbT^{r_1+r_2-1}.
\label{fb.1}\end{equation}
This map is a locally trivial fibration with fiber  
$$
    \Gamma_{\eta}\cap N_{\infty}\setminus N_{\infty}\cong \bbT^{r_1+2r_2}=\bbT^{d_{\bbK}},
$$
where 
$$
N_{\infty}= N(\bbK\otimes_{\bbQ}\bbR)= N(\bbR)^{r_1}\times N(\bbC)^{r_2}
$$
with 
$$
    N(\bbR)= \left\{ \left(\begin{array}{cc}  1 & x \\ 0 &  1  \end{array}  \right)\; | \; x\in\bbR \right\} \quad \mbox{and} \quad N(\bbC)=\left\{ \left(\begin{array}{cc} 1 & z \\ 0 & 1 \end{array}  \right)\; | \; z\in\bbC \right\}.
$$

To describe the metric $g$ in \eqref{mc.3} in the cusp end corresponding to $\eta\in \mathfrak{P}_{\Gamma}$, let us make the following change of variables with respect to the coordinates used in \eqref{mc.3}, 
\begin{equation}\label{new-coord}
\log r= \frac{1}{d_{\bbK}}\left(\sum_{i=1}^{r_1}\log y_i+ 2\sum_{j=1}^{r_2}\log t_j\right) , \quad  u_i= \log y_i-\log y_1, \quad v_j= \log t_{j}-\log y_1,
\end{equation}
if the number field $\bbK$ admits at least one real embedding.  If the number field $\bbK$ admits no real embedding, that is, if $r_1=0$, then we define $\log r$ as before and set instead
$$
     v_j= \log t_j-\log t_1 \quad \forall j\in \{2,\ldots, r_2\}.
$$
Using the conventions that $u_1=0$ if $r_1>0$ and $v_1=0$ if $r_1=0$, we set 
$$
   \mu:= \frac{1}{d_{\bbK}}\left( \sum_{i=1}^{r_1} u_i + 2 \sum_{j=1}^{r_2} v_j  \right),
$$
so that
$$
    \log y_i=\log r +u_i-\mu   \quad \mbox{and}  \quad \log t_j= \log r+v_j -\mu.
$$
Setting $\widetilde{u}_i=u_i-\mu$ and $\widetilde{v}_j=v_j-\mu$, the metric $\widetilde{g}$ becomes
$$
\widetilde{g}= d_{\bbK}\frac{dr^2}{r^2}+ d_{\bbK}g_{S_{\eta}}+ \frac{1}{r^2}\left( \sum_{i=1}^{r_1}e^{-2\tu_i}dx_i^2+ 2 \sum_{j=1}^{r_2} e^{-2\tv_j}|dz_j|^2 \right)
$$
when $r_1>0$, where $g_{S_{\eta}}$ is a flat metric on the base $S_{\eta}\cong\bbT^{r_1+r_2-1}$ of the fibered bundle \eqref{fb.1}  and the cusp is when $r\to \infty$.  Thus, the metric $g_{S_{\eta}}$ can be seen as a Euclidean metric in $u_2,\ldots, u_{r_1},v_1,\ldots, v_{r_2}$  (just in $v_2,\ldots, v_{r_2}$ if $r_1=0$), though not necessarily the canonical one.

To ease the comparison with \cite{MR1}, we will divide this metric by $d_{\bbK}$ and let 
\begin{equation}
g_{\fc}= \frac{dr^2}{r^2}+ g_{S_{\eta}}+ \frac{1}{d_{\bbK}r^2}\left( \sum_{i=1}^{r_1}e^{-2\tu_i}dx_i^2+ 2 \sum_{j=1}^{r_2} e^{-2\tv_j}|dz_j|^2 \right)
\label{mc.4}\end{equation}
be the fibered cusp metric we will consider on $X$.  For this metric, a local basis of orthonormal forms is given by
$$
    \frac{dr}{r}, \nu_1,\ldots, \nu_{r_1+r_2-1}, \frac{e^{-\tu_1} dx_1}{\sqrt{d_{\bbK}} r},\ldots, \frac{e^{-\tu_{r_1}} dx_{r_1}}{\sqrt{d_{\bbK}} r},
     \frac{e^{-\tv_1}dz_1}{\sqrt{d_{\bbK}}r}, \frac{e^{-\tv_1}d\overline{z}_1}{\sqrt{d_{\bbK}}r},\ldots, \frac{e^{-\tv_{r_2}}dz_{r_2}}{\sqrt{d_{\bbK}}r}, \frac{e^{-\tv_{r_2}}d\overline{z}_{r_2}}{\sqrt{d_{\bbK}}r},
$$
where $\nu_1,\ldots, \nu_{r_1+r_2-1}$ is a basis of orthonormal parallel forms for $g_{S_{\eta}}$.

On the other hand, in terms of these coordinates and the bundle metric of \cite{MM1963},   a local basis of orthonormal sections of the flat vector bundle $E_{m,n}$ when $|m|:=m_1+\cdots+m_{r_1}=1$ with $m_i=1$ for some fixed $i$ and $n=0$ is given  by 
$$
    \begin{array}{l}e_{i,1}= \left( \begin{array}{cc} \lambda & t \\ 0 & \lambda^{-1}\end{array} \right)\left(\begin{array}{c} 1 \\ 0  \end{array}  \right)= r^{\frac12}e^{\frac{\tu_i}2 } \left( \begin{array}{c} 1 \\ 0 \end{array} \right), \\
     e_{i,2}= \left( \begin{array}{cc} \lambda & t \\ 0 & \lambda^{-1}\end{array} \right)\left(\begin{array}{c} 0 \\ 1  \end{array}  \right)= r^{-\frac12}e^{-\frac{\tu_i}2} \left( \begin{array}{c} x_i \\ 1 \end{array} \right), 
\end{array}     \quad \lambda= \sqrt{y_i}=r^{\frac12}e^{\frac{\tu_i}2}, t= \frac{x_i}{\sqrt{y_i}}.
$$ 
If instead $m=0$ and $|n|:= n_1+\bn_1+\cdots+ n_{r_2}+\bn_{r_2}=1$ with $n_j=1$ for some fixed $j$, then a local basis of sections of $E_{m,n}$ is given by
$$
    \begin{array}{l}f_{j,1}= \left( \begin{array}{cc} \lambda & \zeta \\ 0 & \lambda^{-1}\end{array} \right)\left(\begin{array}{c} 1 \\ 0  \end{array}  \right)= r^{\frac12}e^{\frac{\tv_j}2} \left( \begin{array}{c} 1 \\ 0 \end{array} \right), \\
     f_{j,2}= \left( \begin{array}{cc} \lambda & \zeta \\ 0 & \lambda^{-1}\end{array} \right)\left(\begin{array}{c} 0 \\ 1  \end{array}  \right)= r^{-\frac12}e^{-\frac{\tv_j}2} \left( \begin{array}{c} z_j \\ 1 \end{array} \right), 
\end{array}     \quad \lambda= \sqrt{t_j}=r^{\frac12}e^{\frac{\tv_j}{2}}, \zeta= \frac{z_j}{\sqrt{t_j}}.
$$ 
Finally, if $m=0$ and $|n|:= n_1+\bn_1+\cdots+ n_{r_2}+\bn_{r_2}=1$ with $\bn_j=1$ for some fixed $j$, then a local basis of sections of $E_{m,n}$ is given by
$$
 \begin{array}{l}\barf_{j,1}= \left( \begin{array}{cc} \lambda & \overline{\zeta} \\ 0 & \lambda^{-1}\end{array} \right)\left(\begin{array}{c} 1 \\ 0  \end{array}  \right)= r^{\frac12}e^{\frac{\tv_j}2} \left( \begin{array}{c} 1 \\ 0 \end{array} \right), \\
     \barf_{j,2}= \left( \begin{array}{cc} \lambda & \overline{\zeta} \\ 0 & \lambda^{-1}\end{array} \right)\left(\begin{array}{c} 0 \\ 1  \end{array}  \right)= r^{-\frac12}e^{-\frac{\tv_j}2} \left( \begin{array}{c} \bz_j \\ 1 \end{array} \right), 
\end{array}     \quad \lambda= \sqrt{t_j}=r^{\frac12}e^{\frac{\tv_j}{2}}, \zeta= \frac{z_j}{\sqrt{t_j}}.
$$
Notice in particular that $\barf_{j,1}$ and $\barf_{j,2}$ are precisely the complex conjugates of $f_{j,1}$ and $f_{j,2}$ respectively.
Hence, more generally, 
$$
    w_{k,l}:= \left(\bigotimes_{i=1}^{r_1}(e_{i,1}^{k_i}\otimes e_{i,2}^{m_i-k_i})\right) \otimes \left( \bigotimes_{j=1}^{r_2}(f_{j,1}^{l_j}\otimes \barf_{j,1}^{\overline{l_j}}\otimes f_{j,2}^{n_j-l_j}\otimes \barf_{j,2}^{\bn_j-\overline{l}_j})\right),
$$
for $k\le m\in \bbN^{r_1}_0$ and  $l\le n\in\bbN^{2r_2}_0$,  
form a local basis of orthonormal sections of $E_{m,n}$.  One computes that 
$$
  de_{i,1}= \frac12\left( \frac{dr}{r}+ d\tu_i \right)e_{i,1}, \quad de_{i,2}= \left( e^{-\tu_i}\frac{dx_i}{r} \right)e_{i,1}- \frac{1}2 \left(  \frac{dr}{r}+d\tu_i \right)e_{i,2} \quad 
$$
and
$$
df_{j,1}= \frac{1}2\left( \frac{dr}{r}+d\tv_j \right)f_{j,1}, \quad df_{j,2}= \left(e^{-\tv_j} \frac{dz_j}{r}\right)f_{j,1} -\frac{1}2 \left( \frac{dr}r+ d\tv_j \right)f_{j,2},
$$
so that
\begin{multline}
 dw_{k,l}= \left[ \frac{(2|k|+2|l|-|m|-|n|)}2\frac{dr}r  + \sum_{i=1}^{r_1}\frac{2k_i-m_i}2 d\tu_i + \sum_{j=1}^{r_2} \frac{2(l_j+\overline{l}_j)-n_j-\bn_j}2 d\tv_j\right]w_{k,l} \\
 + \sum_{i=1}^{r_1}(m_i-k_i)\left(e^{-\tu_i}\frac{dx_i}{r}\right)w_{k+1_i,l}  \\
 + \sum_{j=1}^{r_2}\lrp{
   (n_j-l_j)\left( e^{-\tv_j}\frac{dz_j}{r} \right)w_{k,l+1_j}  + (\bn_j-\overline{l}_j)\left( e^{-\tv_j}\frac{d\bz_j}{r} \right)w_{k,l+\overline{1}_j} },
\label{dif.1}\end{multline}
where $1_i$ is the $r_1$-tuple with $i$th entry equal to one and the other entries equal to zero, $1_j$ is the $2r_2$-tuple with $(2j-1)$th entry equal to 1 and other entries equal to zero and $\overline{1}_j$ is the $2r_2$-tuple with $(2j)$th entry equal to 1 and other entries equal to zero.  

Since the representation $\rho_{m,n}$ is self-dual, notice that the dual flat vector bundle $E_{m,n}^*$ is naturally isomorphic to $E_{m,n}$ as a flat vector bundle, as well as a hermitian vector bundle.  This can be seen directly in terms of the local sections $w_{i,j}$.   The orthonormal basis of sections dual to $\{e_{i,1},e_{i,2}\}$ is given by
$$
    e^{1}_i= r^{-\frac12}e^{-\frac{\tu_i}2} \left( \begin{array}{cc} 1 & -x_i  \end{array} \right) \quad \mbox{and} \quad e^2_i= r^{\frac12}e^{\frac{\tu_i}2}\left(\begin{array}{cc} 0 & 1  \end{array}  \right),
$$
while the local orthonormal basis of sections dual to $\{f_{j,1},f_{j,2}\}$  is given by 
$$
    f^{1}_j= r^{-\frac12}e^{-\frac{\tv_j}2} \left( \begin{array}{cc} 1 & -z_j  \end{array} \right) \quad \mbox{and} \quad f^2_j= r^{\frac12}e^{\frac{\tv_j}{2}}\left(\begin{array}{cc} 0 & 1  \end{array}  \right)
$$
with their complex conjugates $ \barf^{1}_j$ and $\barf^{2}_j$ giving the local orthonormal basis of sections dual to $\{\barf_{j,1},\barf_{j,2}\}$.
For these sections, one computes that 
$$
\begin{aligned}
d e^1_i&= -\frac12\left( \frac{dr}r+d\tu_i  \right)e^1_i- \left( \frac{e^{-\tu_i}dx_i}{r}\right)e^2_i, \\
de^2_i &= \frac12\lrp{\frac{dr}r+ d\tu_i} e^2_i, \\
df^1_j &= -\frac12 \lrp{\frac{dr}r+ d\tv_j }f^1_j -\lrp{\frac{e^{-\tv_j}dz_j}{r}}f^2_j, \\
df^2_j &= \frac12 \lrp{\frac{dr}r+d\tv_j}f^2_j,
\end{aligned}
$$
so that the natural isomorphisms of flat vector bundles $E_{1_i,0}\to E^*_{1_i,0}$, $E_{0,1_j}\to (E_{0,1_j})^*$ and $E_{0,\overline{1}_j}\to (E_{0,\overline{1}_j})^*$are induced by
$$
     e_{i,1}\mapsto e^2_i, \; e_{i,2}\mapsto -e^1_i, \quad f_{j,1}\mapsto f^2_j, \; f_{j,2}\mapsto -f^1_j \quad \mbox{and} \quad \quad \barf_{j,1}\mapsto \barf^2_j, \; \barf_{j,2}\mapsto -\barf^1_j.
$$

\section{Cusp degeneration and the Hodge-deRham operator} \label{hdR.0}

As in \cite{MR1}, $X$ can be compactified to a manifold with boundary $\bX$ by adding a copy of $Y_{\eta}$ at infinity for each cusp end
$\eta\in \mathfrak{P}_{\Gamma}$,
$$
\bX= X \cup \lrp{ \sqcup_{\eta\in\mathfrak{P}_{\Gamma}} Y_{\eta} }, \quad \pa \bX= \sqcup_{\eta\in\mathfrak{P}_{\Gamma}} Y_{\eta}.
$$
On $\bX$, we can choose a boundary defining function $x$ such that for each $\eta\in \mathfrak{P}_{\Gamma}$, 
$x=\frac{1}{r}$ in the cusp end \eqref{mc.4} corresponding to $\eta$.   
If 
$$
  M = \bX\cup_{\pa \bX}\bX
$$ 
is the double of $\bX$ along $\pa X$, then on $M$, one can consider a family of metrics $g_{\epsilon}$ parametrized by $\epsilon>0$ which in a tubular neighborhood $Y_{\eta}\times (-\delta,\delta)_x$ of $Y_{\eta}$ in $M$ takes the form
\begin{equation}
 g_{\fc,\epsilon}=\frac{dx^2}{\rho^2}+ g_{S_{\eta}}+ \frac{\rho^2}{d_{\bbK}}\left( \sum_{i=1}^{r_1} e^{-2\tu_i}dx_i^2+ 2\sum_{j=1}^{r_2} e^{-2\tv_j}|dz_j|^2 \right),  \quad \rho:= \sqrt{x^2+\epsilon^2} 
\label{model.1}\end{equation}
and which on $M\setminus \pa\bX= X\sqcup X$ converges to $g_{\fc}$ on each copy of $X$ as $\epsilon\searrow 0$.  
\begin{figure}[h]
\begin{tikzpicture}
\draw(1,0) arc[radius=1, start angle=0, end angle=180];
\draw (-3,0)--(-1,0);
\draw[->] (1,0)--(3,0);
\draw[->] (0,1)--(0,4);
\node at (0,0.5) {$\bhs{sb}$};
\node at (-2,-0.3) {$\bhs{sm}$};
\node at (3.2,0.2) {$x$};
\node at (0.2,4.2) {$\epsilon$};
\end{tikzpicture}
\caption{The single surgery space $X_s$}
\label{fig.1}\end{figure}

For such a family of metrics, we can consider as in \cite{ARS1} the single surgery space of Mazzeo-Melrose \cite{mame1}
\begin{equation}
   X_s= [M\times [0,1]_{\epsilon};\pa\bX\times \{0\}]
\label{sss.1}\end{equation}
obtained by blowing up $\pa\bX\times \{0\}$ inside $M\times [0,1]_{\epsilon}$ in the sense of Melrose.  It is a manifold with corners with natural blow-down map $\beta_s:X_s\to M\times [0,1]_{\epsilon}$.  We denote by $\bhs{sm}:=\overline{\beta_s^{-1}(M\setminus \pa\bX)\times\{0\}}$ the lift of the old boundary hypersurface at $\epsilon=0$ and by $\bhs{sb}:=\beta_s^{-1}(\pa\bX\times \{0\})$ the new boundary hypersurface created by the blow-up.

There is a corresponding flat vector bundle $\hE_{m,n}$  corresponding to $E_{m,n}$ on each copy of $X$.  We can equip $\hE_{m,n}$ with a bundle metric $h_{\epsilon}$ depending on $\epsilon>0$ and such that $h_{\epsilon}\to h$ on each copy of $X$ as $\epsilon\searrow 0$.  To describe this metric near $Y_{\eta}$, it suffices to give a local basis of orthonormal sections, which we take to be
$$
    \hat{w}_{k,l}:= \left(\bigotimes_{i=1}^{r_1}(\hat{e}_{i,1}^{k_i}\otimes \hat{e}_{i,2}^{m_i-k_i})\right) \otimes \left( \bigotimes_{j=1}^{r_2}(\hat{f}_{j,1}^{l_j}\otimes \hat{\barf}_{j,1}^{\overline{l}_j}\otimes \hat{f}_{j,2}^{n_j-l_j}\otimes  \hat{\barf}_{j,2}^{\bn_j-\overline{l}_j})\right)
$$
for $k\le m\in \bbN^{r_1}_0$ and $ l\le n\in\bbN^{2r_2}_0$   with
$$
    \hat{e}_{i,1}:= \rho^{-\frac12}e^{\frac{\tu_i}2} \left(\begin{array}{c} 1\\ 0  \end{array}\right), \; \hat{e}_{i,2}:= \rho^{\frac12}e^{-\frac{\tu_i}2} \lrp{\begin{array}{c} x_i\\1  \end{array}}, \; \hat{f}_{j,1}:= \rho^{-\frac12}e^{\frac{\tv_j}2} \lrp{ \begin{array}{c} 1\\0  \end{array}}, \; \hat{f}_{j,2}:= \rho^{\frac12}e^{-\frac{\tv_j}2} \lrp{ \begin{array}{c} z_j\\1  \end{array}}
$$
and $\hat{\barf}_{j,1}, \hat{\barf}_{j,2}$ the complex conjugates of $\hat{f}_{j,1}$ and $\hat{f}_{j,2}$.
In terms of these sections, notice that the following analog of \eqref{dif.1} holds,
\begin{multline}
 d\hat{w}_{k,l}= \left[ \frac{(|m|+|n|-2|k|-2|l|)}2\frac{x}{\rho}\frac{dx}{\rho}  + \sum_{i=1}^{r_1}\frac{2k_i-m_i}2 d\tu_i + \sum_{j=1}^{r_2} \frac{2(l_j+\overline{l}_j)-n_j-\bn_j}2 d\tv_j\right]\hat{w}_{k,l} \\
 + \sum_{i=1}^{r_1}(m_i-k_i)\left(e^{-\tu_i}\rho dx_i\right)\hat{w}_{k+1_i,l}  \\
 +\sum_{j=1}^{r_2}
   \lrp{(n_j-l_j)\left( e^{-\tv_j}\rho dz_j \right)\hat{w}_{k,l+1_j} +(\bn_j-\overline{l}_j)\left( e^{-\tv_j}\rho d\bz_j \right)\hat{w}_{k,l+\overline{1}_j}    }.
\label{dif.1b}\end{multline}

Let $\eth_{\fc,\epsilon}=d+d^*$ be the Hodge-deRham operator associated to $(M,\hE_{m,n}, g_{\fc,\epsilon},h_{\epsilon})$, where $d$ is the exterior differential and $d^*$ is its formal adjoint with respect to $g_{\fc,\epsilon}$ and $h_{\epsilon}$.  In a tubular neighborhood $Y_{\eta}\times (-\delta,\delta)_x$ of $Y_{\eta}$ in $M$, we can describe $\eth_{\fc,\epsilon}$ in terms of the decomposition of forms
\begin{equation}
    \omega= \omega_0 + \frac{dx}{\rho}\wedge \omega_1
\label{dec.1}\end{equation}
with $\omega_0$ and $\omega_1$ forms not involving $\frac{dx}{\rho}$.  To do so, we need to introduce a weight operator $W$ defined on forms by the number operator in the fibers of the fiber bundle $Y_{\eta}\to S_{\eta}$,
\begin{equation}
\begin{aligned}
   &  W\left(\frac{dx}{\rho}\right)= W(\nu_q)=0, \quad q\in \{1,\ldots, r_1+r_2-1\},  \\
    &  W(\omega)=\omega \quad \mbox{for} \quad \omega\in\left\{ e^{-\tu_i}\rho dx_i, e^{-\tv_j}\rho dz_j, e^{-\tv_j}\rho d\overline{z}_j \right\}, \quad i\le r_1, \; j\le r_2, 
\end{aligned}      
\label{w.1}\end{equation}
 and on sections of $\hE_{m,n}$ by
 \begin{equation}
  W(\hat{w}_{k,l})= \lrp{\frac{|m|+|n|}2-|k|-|l|}\hat{w}_{k,l}.  
 \label{w.2}\end{equation}
 In terms of this weight and the decomposition \eqref{dec.1}, we have that
 \begin{equation}
   \eth_{\fc,\epsilon}= \lrp{\begin{array}{cc} \rho^{-1}\eth_{Y_{\eta}/S_{\eta}}+ \eth_{S_{\eta}} & \frac{x}{\rho}(W-d_{\bbK})-\rho\frac{\pa}{\pa x} \\
     \rho\frac{\pa}{\pa x}+ \frac{x}{\rho}W & -\rho^{-1}\eth_{Y_{\eta}/S_{\eta}}- \eth_{S_{\eta}} \end{array}}
 \label{dec.2}\end{equation}
For some vertical and horizontal operators $\eth_{Y_{\eta}/S_{\eta}}$ and $\eth_{S_{\eta}}$ with respect to the fiber bundle $\phi_{\eta}:Y_{\eta}\to S_{\eta}$ of \eqref{fb.1}.  Compared to \cite[\S~2.2]{ARS1}, notice that there is no curvature term, which is due to the fact that the Riemannian submersion $\phi_{\eta}:Y_{\eta} \to S_{\eta}$ has trivial curvature in the sense of \cite[\S~10.1]{BGV}.
 
Let $D_{v,\eta}$ be the vertical family of \cite{ARS1} associated to the family of Hodge-deRham operators $\eth_{\fc,\epsilon}$ at $Y_{\eta}$, that is,
$$
 D_{v,\eta}= \lrp{\begin{array}{cc} \eth_{Y_{\eta}/S_{\eta}}& 0 \\
      0 & -\eth_{Y_{\eta}/S_{\eta}}\end{array}}.
$$  

This is a family of operators acting fiberwise on the fibers of the fiber bundle 
$$
     \phi_{\eta}: Y_{\eta}\to S_{\eta}.
$$
As in \cite{MR1}, in each fiber, it corresponds to the Hodge-deRham operator with respect to the induced fiberwise metric and trivial flat vector bundle of rank $\rank E_{m,n}$, that is the trivial vector bundle $E_{m,n}$ with flat structure obtained by declaring the sections $\hat{w}_{i,j}$ to be flat.  Its kernel is naturally a vector bundle over $S_{\eta}$ with local sections given by
$$
    \Lambda^*\langle \frac{dx}{\rho}, \nu_q, e^{-\tu_i}\rho dx_i, e^{-\tv_j}\rho dz_j, e^{-\tv_j}\rho d\bz_j \; | \; q\le r_1+r_2-1, i\le r_1, j\le r_2 \rangle\otimes \langle \hat{w}_{k,l}\; | \; k\le m, \; l\le n\rangle.
$$ 
The corresponding horizontal operator $D_{b,\eta}$ of \cite{ARS1} is obtained by letting 
\begin{equation}
\rho^{\frac{d_{\bbK}}2}\eth_{\fc,\epsilon}\rho^{-\frac{d_\bbK}2}= \lrp{\begin{array}{cc} \rho^{-1}\eth_{Y_{\eta}/S_{\eta}}+ \eth_{S_{\eta}} & \frac{x}{\rho}\left(W-\frac{d_{\bbK}}{2}\right)-\rho\frac{\pa}{\pa x} \\
     \rho\frac{\pa}{\pa x}+ \frac{x}{\rho}\left(W-\frac{d_{\bbK}}2 \right) & -\rho^{-1}\eth_{Y_{\eta}/S_{\eta}}- \eth_{S_{\eta}} \end{array}}
\label{model.2}\end{equation}
 act on such sections extended smoothly off the boundary hypersurface $\bhs{sb}$ of $X_s$ (defined in \eqref{sss.1}) and then restricting back to $\bhs{sb}$.  A computation shows that in terms of the decomposition \eqref{dec.1}, 
\begin{equation}
 D_{b,\eta}=\lrp{ \begin{array}{cc} \eth_{S_{\eta}} & -D(\frac{d_{\bbK}}2-W) \\D(W-\frac{d_{\bbK}}2) & -\eth_{S_{\eta}}  \end{array}   },
\label{dec.3}\end{equation}
where 
$$
  D(a)=\langle X\rangle^{-a} \langle X\rangle \frac{\pa}{\pa X} \langle X\rangle^a= \langle X\rangle \frac{\pa}{\pa X}+ \frac{aX}{\langle X\rangle},  \quad a\in \bbR, \quad X= \frac{x}{\epsilon}, \; \langle X\rangle= \sqrt{1+X^2},   
$$ 
 and $\eth_{S_{\eta}}$ is the Hodge-deRham operator on $S_{\eta}$ (with metric induced by $g_{\fc}$) acting on the flat vector bundle generated by the space of sections
 \begin{equation}
 \cC^*=  \Lambda^*\langle e^{-\tu_i}\rho dx_i, e^{-\tv_j}\rho dz_j, e^{-\tv_j}\rho d\bz_j \; | \;  i\le r_1, j\le r_2 \rangle\otimes \langle \hat{w}_{k,l}\; | \; k\le m, \; l\le n\rangle.
 \label{dec.4}\end{equation}
 Let $\eth_{\cC}$ be the natural (Kostant type) Hodge-deRham operator associated with the finite dimensional complex \eqref{dec.4} with differential $d_{\cC}$ defined by 
 \begin{multline}
    d_\cC \hat{w}_{k,l}= \sum_{i=1}^{r_1}(m_i-k_i)\left(e^{-\tu_i}\rho dx_i\right)\hat{w}_{k+1_i,l}  \\+ \sum_{j=1}^{r_2}
   \lrp{  (n_j-l_j)\left( e^{-\tv_j}\rho dz_j \right)\hat{w}_{k,l+1_j} + (\bn_j-\overline{l}_j)\left( e^{-\tv_j}\rho d\bz_j \right)\hat{w}_{k,l+\overline{1}_j}   }
\label{dc.1}\end{multline}
with natural metric induced by $g_{\fc,\epsilon}$ and $h_{\epsilon}$.  By Hodge theory, the kernel of $\eth_{\cC}$ corresponds to the cohomology of the complex $\cC^*$.  It admits in fact the following explicit description.
\begin{lemma}
The kernel of $\eth_{\cC}$ is given by
\begin{multline}
\cH^*(\cC)=\left(\bigwedge_{i=1}^{r_{1}} \langle \hat{e}^{m_i}_{i,1}, \hat{e}_{i,2}^{m_i}e^{-\tu_i}\rho dx_i \rangle\right) \wedge \\
\left( \bigwedge_{j=1}^{r_2} \langle \hat{f}_{j,1}^{n_j}\otimes\hat{\barf}^{\bn_j}_{j,1}, \hat{f}^{n_j}_{j,1}\otimes\hat{\barf}^{\bn_j}_{j,2}e^{-\tv_j}\rho d\bz_j, \hat{f}_{j,2}^{n_j}\otimes\hat{\barf}^{\bn_j}_{j,1}e^{-\tv_j}\rho dz_j, \hat{f}^{n_j}_{j,2}\otimes\hat{\barf}^{\bn_j}_{j,2}e^{-2\tv_j}\rho^2dz_j\wedge d\bz_j \rangle \right).
\end{multline}
\label{dec.5c}\end{lemma}
\begin{proof}
Notice first that the complex $\cC^*$ admits the decomposition into subcomplexes
\begin{equation}
    \cC^*= \left( \bigwedge_{i=1}^{r_1} \cC_i^* \right)\wedge \left( \bigwedge_{j=1}^{r_2} \cD^*_j \right) \wedge \left( \bigwedge_{j=1}^{r_2} \overline{\cD}^*_j \right),
\label{dec.5d}\end{equation}
where 
$$
  \cC_i^*=\langle \hat{e}^{k_i}_{i,1}\otimes \hat{e}^{m_i-k_i}_{i,2}, \hat{e}^{k_i}_{i,1}\otimes \hat{e}^{m_i-k_i}_{i,2} e^{-\tu_i}\rho dx_i \; | \; k_i\in \{0,1,\ldots, m_i\} \rangle
$$
is the complex with differential given by
$$
\begin{aligned}
 & d_{\cC_i}(\hat{e}^{k_i}_{i,1}\otimes \hat{e}^{m_i-k_i}_{i,2} )=  (m_i-k_i)e^{k_i+1}_{i,1}\otimes e^{m_i-k_i-1}_{i,2} e^{-\tu_i}\rho dx_i, \\
 & d_{\cC_i}(\hat{e}^{k_i}_{i,1}\otimes \hat{e}^{m_i-k_i}_{i,2} e^{-\tu_i}\rho dx_i) = 0,
\end{aligned}
$$
while 
$$
 \cD_j^*=\langle \hat{f}^{l_j}_{j,1}\otimes \hat{f}^{n_j-l_j}_{j,2} \; | \; l_j\in \{0,1,\ldots, n_j\} \rangle \otimes \langle 1, dz_j \rangle
$$
is the complex with differential given by
$$
d_{\cD_j} (\hat{f}^{l_j}_{j,1}\otimes \hat{f}^{n_j-l_j}_{j,2}\wedge \omega)= (n_j-l_j) \hat{f}^{l_j+1}_{j,1}\otimes \hat{f}^{n_j-l_j-1}_{j,2} e^{-\tv_j}\rho dz_j\wedge \omega
$$
for $\omega\in  \langle 1, dz_j \rangle$ and $\overline{\cD}^{*}_j$ is its complex conjugate.  There are corresponding Hodge-deRham operators $\eth_{\cC_i}$ and $\eth_{\cD_j\wedge \overline{\cD}_j}$.  A direct computation shows that their kernels are respectively given by
$$
   \cH^*(\cC_i)= \langle \hat{e}^{m_i}_{i,1}, \hat{e}_{i,2}^{m_i}e^{-\tu_i}\rho dx_i \rangle,
$$
$$
 \cH^*(\cD_j\wedge \overline{\cD}_j)= \langle \hat{f}_{j,1}^{n_j}, \hat{f}_{j,2}^{n_j}e^{-\tv_j}\rho dz_j \rangle \wedge \langle \hat{\barf}_{j,1}^{n_j}, \hat{\barf}_{j,2}^{n_j}e^{-\tv_j}\rho d\bz_j \rangle.
$$
Since 
$$
  \eth_{\cC}= \sum_{i=1}^{r_1}\eth_{\cC_i}+ \sum_{j=1}^{r_2} \eth_{\cD_j\wedge \overline{\cD}_j} 
$$
and the various Hodge-deRham operators in this sum anti-commute, the result follows.  In terms of cohomology, this is just the K\"unneth formula for the decomposition \eqref{dec.5d}.  
\end{proof}

The differential $d_{S_{\eta}}$ of the complex
 \begin{equation}
     \Omega^*(S_{\eta};\ker\eth_{Y_{\eta}/S_{\eta}})
\label{dec.5}\end{equation}
is then of the form
\begin{equation}
    d_{S_{\eta}}= \widetilde{d}_{S_{\eta}}+ d_{\cC}
\label{dec.5b}\end{equation}
with $\widetilde{d}_{S_{\eta}}$ also a differential, namely
\begin{equation}
\begin{aligned}
  & \widetilde{d}_{S_{\eta}} \hat{w}_{k,l}= \left[  \sum_{i=1}^{r_1}\frac{2k_i-m_i}2 d\tu_i + \sum_{j=1}^{r_2} \frac{2(l_j+\overline{l}_j)-n_j-\bn_j}2 d\tv_j\right]\hat{w}_{k,l}, \\
  & \widetilde{d}_{S_{\eta}}  (e^{-\widetilde{u}_i}\rho dx_i) = -d\widetilde{u_i}\wedge (e^{-\widetilde{u}_i}\rho dx_i), \quad i\le r_1, \\
   &\widetilde{d}_{S_{\eta}} (e^{-\widetilde{v}_j}\rho dz_j) = -d\widetilde{v}_j\wedge (e^{-\widetilde{v}_j}\rho dz_j), \quad j\le r_2, \\
  &\widetilde{d}_{S_{\eta}} (e^{-\widetilde{v}_j}\rho d\overline{z}_j) = -d\widetilde{v}_j\wedge (e^{-\widetilde{v}_j}\rho d\overline{z}_j), \quad j\le r_2, 
\end{aligned}   
\label{dec.5f}\end{equation}
and 
$$
   \widetilde{d}_{S_{\eta}} ( \omega \wedge \nu)= (d\omega) \wedge \nu + (-1)^q \omega\wedge \widetilde{d}_{S_{\eta}}\nu \quad \mbox{for} \quad \omega\in\Omega^q(S_{\eta}), \; \nu\in \cC^*.
$$
Using the natural metric induced by $g_{\fc,\epsilon}$ and $h_{\epsilon}$, the differential $\widetilde{d}_{S_{\eta}}$ has a formal adjoint $d^*_{S_{\eta}}$.  In particular, there is a corresponding Hodge-deRham operator $\widetilde{\eth}_{S_{\eta}}$ and we see from \eqref{dec.5b} that
$$
  \eth_{S_{\eta}}= \widetilde{\eth}_{S_{\eta}}+ \eth_{\cC}.
$$  
Since, $d_{S_{\eta}}=\widetilde{d}_{S_{\eta}}+d_{\cC}$ is a differential, we see that the differentials $\widetilde{d}_{S_{\eta}}$ and $d_{\cC}$ anti-commute.  Correspondingly, their formal adjoints $\widetilde{d}_{S_{\eta}}^*$ and $d_{\cC}^*$ anti-commute.  In fact, writing the formal adjoints in terms of the corresponding Hodge star operators we see that $\widetilde{\eth}_{S_{\eta}}$ and $\eth_{\cC}$ anti-commutes.  Hence, we see that
$$
   \eth_{S_{\eta}}^2= \widetilde{\eth}_{S_{\eta}}^2+ \eth_{\cC}^2.
$$
This implies that an element in the kernel of $\eth_{S_{\eta}}$ will be in the kernels of $\widetilde{\eth}_{S_{\eta}}$ and $\eth_{\cC}$ and vice-versa.   Combined with Lemma~\ref{dec.5c}, this can be used to obtain the following description of $\ker \eth_{S_{\eta}}$.
\begin{lemma}
If $r_1\ge 1$, the kernel of $\eth_{S_{\eta}}$ is trivial unless $m_1=\cdots=m_{r_1}$ and $n_j\in\{2m_1-\bn_j,2m_1+2+\bn_j, \bn_j-2m_1-2\}$ for all $j\in\{1,\ldots,r_2\}$.  In this latter case, the kernel is  given by
\begin{multline}
   \langle w_{m,l}\left( \bigwedge_{j\in J} e^{-\tv_j}\rho dz_j\right)\wedge \left( \bigwedge_{j\in \overline{J}} e^{-\tv_j}\rho d\bz_j\right) ,  \\
   w_{0,n-l} \left(\bigwedge_{i=1}^{r_1} e^{-\tu_i}\rho dx_i\right)\wedge \left( \bigwedge_{j\notin J} e^{-\tv_j}\rho dz_j \right)\wedge \left( \bigwedge_{j\notin \overline{J}} e^{-\tv_j}\rho d\bz_j\right)\rangle\wedge \cH^*(S_{\eta}),
\end{multline}
where $J= \{ j\in \{1,\ldots, r_2\}\; | \; n_j= \bn_j-2m_1-2\}$, $\overline{J}= \{ j\in \{1,\ldots, r_2\}\; | \; n_j= 2m_1+2+\bn_j\}$,
$$
     l_j= \left\{ \begin{array}{ll} 0, & j\in J, \\ n_j, & j\notin J, \end{array}\right. \quad  \overline{l}_j= \left\{ \begin{array}{ll} 0, & j\in \overline{J}, \\ \overline{n}_j, & j\notin \overline{J}, \end{array}\right.
$$

 and $\cH^*(S_{\eta})$, which can be identified with the cohomology ring of $S_{\eta}\cong \bbT^{r_1+r_2-1}$, is the finite dimensional exterior algebra generated by $d\tu_2,\ldots, d\tu_{r_1}, d\tv_1,\ldots, d\tv_{r_2}$.
\label{dec.5e}\end{lemma}
\begin{proof}
First, in the special case where $r_1=1$ and $r_2=0$, notice that $S_{\eta}$ is a point and the result is a direct consequence of Lemma~\ref{dec.5c}.  Otherwise, since we assume $r_1\ge 1$, we must have $\dim S_{\eta}\ge 1$ and by the discussion above, $\ker \eth_{S_{\eta}}$ corresponds to the kernel of $\widetilde{\eth}_{S_\eta}$ acting on 
$\Omega^*(S_{\eta};\ker\eth_{\cC})$.  In terms of the differential $\td_{S_{\eta}}$ and the induced metric, we see from Lemma~\ref{dec.5c} that $\ker \eth_{\cC}$ splits orthogonally into a direct sum of flat line bundles on $S_{\eta}$.  It thus suffices to compute the kernel of $\widetilde{\eth}_{S_{\eta}}$ when acting on sections of each of these flat line bundles.  To describe the corresponding differentials, it is convenient to use, instead of $u_2,\ldots, u_{r_1}, v_1,\ldots, v_{r_2}$, the local coordinates $\tu_2,\ldots, \tu_{r_1}, \tv_1,\ldots, \tv_{r_2}$ on $S_{\eta}$.  This is possible, since as readily checked, the corresponding linear change of coordinates has non-vanishing Jacobian.  In terms of these coordinates, notice that
$$
\mu =\frac{1}{d_{\bbK}}\left(\sum_{i=2}^{r_1}u_i+ 2\sum_{j=1}^{r_2} v_j \right)= \sum_{i=2}^{r_1}\tu_i +2 \sum_{j=1}^{r_2} \tv_j,
$$
so the differential $\td_{S_{\eta}}$ is given by 
$$
\td_{S_{\eta}} \hat{w}_{k,l}=\left[ \sum_{i=2}^{r_1} \frac{2\widetilde{k_i}-\widetilde{m}_i}2d\tu_i + \sum_{j=1}^{r_2}\frac{2(\widetilde{l}_j+\widetilde{\overline{l}}_j)-\widetilde{n}_j-\widetilde{\bn}_j}2 d\widetilde{v}_j \right] \hat{w}_{k,l}
$$ 
with $\widetilde{k}_i=k_i-k_1$, $\widetilde{m}_i=m_i-m_1$, $\widetilde{l}_j=l_j-k_1$, $\widetilde{\overline{l}}_j=\overline{l}_j-k_1$, $\widetilde{n}_j=n_j-m_1$ and $\widetilde{\bn}_j=\bn_j-m_1$, while
$$
 \td_{S_{\eta}}(e^{-\tu_1}\rho dx_1)= d\mu\wedge e^{-\tu_1}\rho dx_1= \left( \sum_{i=2}^{r_1}d\tu_i+ 2\sum_{j=1}^{r_2} d\tv_j\right)\wedge e^{-\tu_1}\rho dx_1
$$
and otherwise is described as in \eqref{dec.5f}.  Now, by Lemma~\ref{dec.5c}, assuming that $m\in \bbN^{r_1}$ and $n\in\bbN^{2r_2}$ for the moment, the flat line bundles of the decomposition of $\ker\eth_{\cC}$ over $S_{\eta}$ are spanned by sections of the form
$$
   \sigma= \hat{w}_{k,l} \left( \bigwedge_{i, k_i=0} e^{-\tu_i}\rho dx_i\right)\wedge \left( \bigwedge_{j,l_j=0} e^{-\tv_i}\rho dz_j\right)\wedge \left( \bigwedge_{j,\overline{l}_j=0}  e^{-\tv_i}\rho d\bz_j\right)
$$
for $k$ and $l$ such that $k_i\in\{0,m_i\}$ for all $i$ and $l_j\in\{0,n_j\}$, $\overline{l}_j\in\{0,\bn_j\}$ for all $j$.   In particular, 
we compute that 
\begin{multline}
  \td_{S_{\eta}}\sigma= \left[ \sum_{i=2}^{r_1} \left( \frac{2\widetilde{k}_i-\widetilde{m}_i}2-\delta_{k_i0}+\delta_{k_10} \right)d\tu_i \right. \\  \left.+ \sum_{j=1}^{r_2}\left( \frac{2(\widetilde{l}_j+\widetilde{\overline{l}}_j)-\widetilde{n}_j-\widetilde{\bn}_j}2+ 2\delta_{k_10}-\delta_{l_j0}-\delta_{\overline{l}_j0} \right)d\tv_j \right]\wedge \sigma.
\label{dec.5g}\end{multline}
Using the K\"unneth formula in terms of the decomposition induced by the coordinates $\tu_2,\ldots,\tu_{r_1}, \tv_1,\ldots,\tv_{r_2}$, we see that the cohomology is trivial unless $d\sigma=0$, that is, unless the line bundle is trivial as a flat vector bundle, in which case the cohomology is isomorphic to $\cH^*(S_{\eta})$.  A careful inspection of \eqref{dec.5g} then shows that the only way $d\sigma=0$ is if 
$$
    \sigma= w_{m,l}\left( \bigwedge_{j\in J} e^{-\tv_j}\rho dz_j\right)\wedge \left( \bigwedge_{j\in \overline{J}} e^{-\tv_j}\rho d\bz_j\right)       
$$
or
$$
\sigma= w_{0,n-l} \left(\bigwedge_{i=1}^{r_1} e^{-\tu_i}\rho dx_i\right)\wedge \left( \bigwedge_{j\notin J} e^{-\tv_j}\rho dz_j \right)\wedge \left( \bigwedge_{j\notin \overline{J}} e^{-\tv_j}\rho d\bz_j\right)
$$
with $m,n$,$l$, $J$ and $\overline{J}$ as in the statement of the lemma, yielding the result.  When we allow $m\in \bbN^{r_1}_0$ and $n\in\bbN^{r_2}_0$, the formula \eqref{dec.5g} must be modified appropriately, but again the same conclusion can be reached.
\end{proof}

When $r_1=0$, the kernel of $\eth_{S_{\eta}}$ can be computed as follows.
\begin{lemma}
If $r_1=0$, suppose that $\bn_1=\cdots \bn_{r_2}=0$ and assume without loss of generality that $n_1\ge \cdots\ge n_{r_2}\ge 0$.  If $n_1>0$, then the kernel of $\eth_{S_{\eta}}$ is trivial unless $n_j\in\{n_1, n_1-2\}$ for all $j$, in which case it is given by
\begin{equation}
\langle w_{0,n}(e^{-\tv}\rho d\bz)^{\alpha}, w_{0,0}\left( \bigwedge_{j=1}^{r_2} e^{-\tv_j}\rho dz_j \right)\wedge (e^{-\tv}\rho d\bz)^{\beta} \rangle \wedge \cH^*(S_{\eta}),
\label{dec.5k}\end{equation}
where $\alpha,\beta\in \{0,1\}^{r_2}$ are such that
$$
\begin{aligned}
      n_j=n_1 \;&\Longrightarrow \; \alpha_j=\alpha_1, \; \beta_j= \beta_1, \\
      n_j= n_1-2 & \Longrightarrow \; \alpha_j=0, \alpha_1=1, \beta_j=1, \beta_1=0.
\end{aligned}      
$$
In particular, if $n_j=n_1-2$ for some $j$, then $\alpha$ and $\beta$ are uniquely determined by these conditions, while otherwise $\alpha,\beta\in \{ (0,\ldots,0), (1,\ldots,1)\}$.  If instead $n_1=\cdots=n_{r_2}=0$, then the kernel is of the form
\begin{equation}
   \langle (\rho dz)^\alpha\wedge (\rho d\bz)^\beta \; | \; \alpha_j+\beta_j= \alpha_1+\beta_1 \; \forall j \rangle \wedge \cH^*(S_{\eta}).
\label{dec.5l}\end{equation}
with 
\begin{equation}
(e^{-\tv}\rho d\bz)^\beta=(e^{-\tv_1}\rho d\bz_1)^{\beta_1}\wedge\cdots \wedge (e^{\tv_{r_2}}\rho d\bz)^{\beta_{r_2}}_{r_2}
\label{dec.5i}\end{equation}
and similarly for $(e^{-\tv}\rho dz)^\alpha$.

\label{dec.5h}\end{lemma}
 \begin{proof}
 If $r_2=1$, then $S_{\eta}$ is a point and the result is a direct consequence of Lemma~\ref{dec.5c}.  Otherwise, we can still follow the strategy of the proof of Lemma~\ref{dec.5e}.  We will only highlight the changes needed.  First, we can use the local coordinates $\tv_2,\ldots, \tv_{r_2}$ on $S_{\eta}$ instead of $v_2,\ldots, v_{r_2}$.  In terms of these coordinates,
 $$
    \mu= \frac{2}{d_{\bbK}}\sum_{j=2}^{r_2}v_j= \sum_{j=2}^{r_2}\tv_j
 $$
 and the differential $\td_{S_{\eta}}$ takes the form 
 $$
       \td_{S_{\eta}} \hat{w}_{0,l}= \left[ \sum_{j=2}^{r_2} \lrp{\frac{2\widetilde{l}_j-\widetilde{n}_j}2}d\tv_j  \right] \hat{w}_{0,l}
 $$
 with now $\widetilde{l}_j=l_j-l_1$ and $\widetilde{n}_j=n_j-n_1$, while 
 $$
 \begin{aligned}
 \td_{S_{\eta}} (e^{-\tv_1}\rho dz_1)&= d\mu\wedge e^{-\tv_1}\rho dz_1= \lrp{\sum_{j=2}^{r_2}d\tv_j}\wedge e^{-\tv_1}\rho dz_1, \\
 \td_{S_{\eta}} (e^{-\tv_1}\rho d\bz_1)&= d\mu\wedge e^{-\tv_1}\rho d\bz_1= \lrp{\sum_{j=2}^{r_2}d\tv_j}\wedge e^{-\tv_1}\rho d\bz_1,
 \end{aligned}
 $$
 and otherwise is as in \eqref{dec.5f}.  Now, $\ker\eth_{\cC}$ still splits into flat line bundles.   Assuming that $n\in \bbN^{r_2}$, it is spanned by sections of the form
 $$
    \sigma= \hat{w}_{0,l}\lrp{\bigwedge_{j, l_j=0}e^{-\tv_j}\rho dz_j}\wedge (e^{-\tv}\rho d\bz)^{\beta}
 $$
 for $l$ such that $l_j\in\{0,n_j\}$ for all $j$ and with $d\bz^{\beta}$ as in \eqref{dec.5i}.  We compute in this case that 
 \begin{equation}
 \td_{S_{\eta}} \sigma= \left[ \sum_{j=2}^{r_2} \lrp{\frac{2\widetilde{l}_j-\widetilde{n}_j}2 +\delta_{l_10} -\delta_{l_j0}+\delta_{\beta_11}-\delta_{\beta_j1}   }d\tv_j  \right] \wedge \sigma.
\label{dec.5j}\end{equation}
Again, to have non-trivial kernel, we must have that $\td_S\sigma=0$.  Looking at \eqref{dec.5j}, we see that the kernel must be of the claimed form \eqref{dec.5k}.  If more generally $n\in\bbN_0^{r_2}$, formula \eqref{dec.5j} must be suitably modified, but again the same conclusion can be reached, except in the case where $n_1=0$, that is, when $n=(0,\ldots,0)$ and  $E_{0,n}$ is a trivial flat line bundle, in which case the kernel is instead described by \eqref{dec.5l}.
 \end{proof}
 \begin{remark}
 The assumption that $\bn_1=\cdots=\bn_{r_2}=0$ in Lemma~\ref{dec.5h} can be removed, but then the result has a more complicated description.
 \label{oc.1}\end{remark}
 
 These results can be used to study the operator $D_{b,\eta}$.  To see this, we need the following lemma.
 \begin{lemma}
 The operator $\eth_{S_{\eta}}$ commutes with the weight operator $W$.  
 \label{dec.5m}\end{lemma}
 \begin{proof}
 From their definitions, the differential $d_{\cC}$ and $\td_{S_{\eta}}$ clearly commute with the weight operator $W$.  Since $W$ is self-adjoint, this means that $d^*_{\cC}$ and $\td^*_{S_{\eta}}$ also commute with $W$.  Hence, so do $\eth_{\cC}$, $\widetilde{\eth}_{S_{\eta}}$ and $\eth_{S_{\eta}}=\eth_{\cC}+\widetilde{\eth}_{S_{\eta}}$.  
 
 \end{proof}
 Using Lemma~\ref{dec.5m}, we see that the operator $D_{b,\eta}$ of \eqref{dec.3} is such that
 \begin{equation}
\begin{aligned}
D_{b,\eta}^2 &= \lrp{\begin{array}{cc} \eth_{S_{\eta}}^2- D(\frac{d_{\bbK}}2-W)D(W-\frac{d_\bbK}2) & 0 \\ 0 & \eth_{S_{\eta}}^2- D(W-\frac{d_{\bbK}}2)D(\frac{d_{\bbK}}2-W) \end{array}} \\
&=\lrp{\begin{array}{cc} \eth_{S_{\eta}}^2+ D(W-\frac{d_{\bbK}}2)^*D(W-\frac{d_{\bbK}}2) & 0 \\ 0 & \eth_{S_{\eta}}^2+D(\frac{d_{\bbK}}2-W)^*D(\frac{d_{\bbK}}2-W) \end{array}}.
\end{aligned}
\label{dec.8}\end{equation}
Hence, as in \cite{MR1}, $D_{b,\eta}$ can possibly be non-Fredholm or have a non-trivial $L^2$-kernel only if $\eth_{S_{\eta}}$ has a non-trivial kernel.  
\begin{proposition}
The operators $D_{b,\eta}$ and $D_{b,\eta}^2$ are Fredholm as $b$-operators for the $b$-density $\frac{dX}{\langle X\rangle}$ for $X\in\bbR$ if $r_1>0$ or if $r_1=0$, $n\ne 0$ and $\bn_1=\cdots=\bn_2=0$.  When $r_1> 0$ the $L^2$-kernel of $D_{b,\eta}$ (and $D_{b,\eta}^2$) is trivial unless $m_1=\cdots=m_{r_1}$ and $$n_j\in\{2m_1-\bn_j,2m_1+2+\bn_j, \bn_j-2m_1-2\}$$ for all $j\in \{1,\ldots,r_2\}$, in which case it is given by 
\begin{multline}
   \langle w_{m,l}\left( \bigwedge_{j\in J} e^{-\tv_j}\rho d z_j\right)\wedge \left( \bigwedge_{j\in \overline{J}} e^{-\tv_j}\rho d\bz_j\right)\wedge\langle X\rangle^{-\frac{d_{\bbK}}2-\frac{|m|-|n|}2-|l|+|J|+|\overline{J}|}\frac{dX}{\langle X\rangle}, \\ w_{0,n-l} \left(\bigwedge_{i=1}^{r_1} e^{-\tu_i}\rho dx_i\right)\wedge \left( \bigwedge_{j\notin J} e^{-\tv_j}\rho dz_j \right)\wedge \left( \bigwedge_{j\notin \overline{J}} e^{-\tv_j}\rho d\bz_j\right) \langle X\rangle^{-\frac{d_{\bbK}}2-\frac{|m|+|n|}2+|l|+|J|+|\overline{J}|} \rangle\wedge \cH^*(S_{\eta})
\end{multline}
with $J, \overline{J}$ and $l$ as in Lemma~\ref{dec.5e}.  If instead $r_1=0$, but $\bn_1=\cdots=\bn_{r_2}=0$, $n_1\ge 1$ and we assume without loss of generality that $n_1\ge\cdots\ge n_{r_2}$, then the $L^2$-kernel of $D_{b,\eta}$ is trivial unless $n_{j}\in\{n_1, n_1-2\}$ for all $j$, in which case it is given by 
\begin{equation}
\langle w_{0,n}(e^{-\tv}\rho d\bz)^{\alpha}\wedge \langle X\rangle^{-\frac{d_{\bbK}}2-\frac{|n|}2+|\alpha|}\frac{dX}{\langle X \rangle}, w_{0,0}\left( \bigwedge_{j=1}^{r_2} e^{-\tv_j}\rho dz_j \right)\wedge (e^{-\tv}\rho d\bz)^{\beta}\langle X\rangle^{-\frac{|n|}2-|\beta|} \rangle \wedge \cH^*(S_{\eta}),
\label{dec.8e}\end{equation}
with $\alpha,\beta\in \{0,1\}^{r_2}$ such that
$$
\begin{aligned}
      n_j=n_1 \;&\Longrightarrow \; \alpha_j=\alpha_1, \; \beta_j= \beta_1, \\
      n_j= n_1-2 & \Longrightarrow \; \alpha_j=0, \alpha_1=1, \beta_j=1, \beta_1=0.
\end{aligned}      
$$
\label{dec.8d}\end{proposition}
 \begin{proof}
 By \cite{MelroseAPS}, it suffices to check that the indicial family of $I(D^2_{b,\eta},\lambda)$ is invertible for all $\lambda \in \bbR$.  As shown in \cite[(2.12)]{ARS2}, 
 $$
    I(D(a),\lambda)= \pm(a-i\lambda)
 $$
 at $X=\pm \infty$, hence 
 $$
  I(D_{b,\eta}^2,\lambda)= \lrp{\begin{array}{cc} \eth_{S_{\eta}}^2 + (W-\frac{d_{\bbK}}2)^2+\lambda^2 & 0 \\ 0 & \eth_{S_{\eta}}^2 + (W-\frac{d_{\bbK}}2)^2+ \lambda^2 \end{array}}
 $$
 at both ends.  Clearly, this is invertible for $\lambda\ne 0$, while at $\lambda=0$, it is invertible provided $W-\frac{d_{\bbK}}2\ne 0$ when acting on $\ker\eth_{S_{\eta}}$.  By the explicit description of $\ker\eth_{S_{\eta}}$ given in Lemmas~\ref{dec.5e} and \ref{dec.5h} and the definition of $W$ in \eqref{w.1} and \eqref{w.2}, this is indeed the case unless $r_1=0$ and $n=0$.   To describe the $L^2$-kernel of $D_{b,\eta}$ and $D_{b,\eta}^2$, notice from \eqref{dec.8} that it corresponds to the kernel of 
 $$
 \lrp{ \begin{array}{cc} 0 & -D(\frac{d_{\bbK}}2-W) \\D(W-\frac{d_{\bbK}}2) & 0  \end{array}   }
 $$ 
 acting on sections of the trivial vector bundle $\ker \eth_{S_{\eta}}\oplus \ker \eth_{S_{\eta}}$ over $\bbR$.  Since the $L^2$-kernel of $D(a)$ is non-trivial and spanned by $\langle X \rangle^{-a}$ if and only if $a>0$, the description of the $L^2$-kernel of $D_{b,\eta}$ and $D_{b,\eta}^2$ therefore follows from Lemmas~\ref{dec.5e} and \ref{dec.5h} and the definition of $W$.
 
 \end{proof}
 
 When $D_{b,\eta}$ is Fredholm, we can apply the uniform construction of the resolvent of \cite[Theorem~4.5]{ARS1} to $\eth_{\fc,\epsilon}$ as $\epsilon\searrow 0$.
 \begin{theorem}
 Suppose that either $r_1\ne 0$ or $r_1=0$ with $\bn_1=\cdots=\bn_{r_2}=0$ and $n\ne 0$.  Then the family of operators $D_{\fc,\epsilon}:=\rho^{\frac{d_{\bbK}}2}\eth_{\fc,\epsilon}\rho^{-\frac{d_{\bbK}}2}$ (using $b$-densities) has finitely many small eigenvalues, that is, there are finitely many eigenvalues of $D_{\fc,\epsilon}$ tending to $0$ as $\epsilon\searrow 0$.  Furthermore, the projection $\Pi_{\operatorname{small}}$ on the eigenspace of small eigenvalues is a polyhomogeneous operator of order $-\infty$ in the surgery calculus of \cite{mame1} and 
 \begin{equation}
 \rank \Pi_{\operatorname{small}}= 2 \dim\ker_{L^2} \eth_{\fc}+ \sum_{\eta\in\mathfrak{P}_{\Gamma}}\dim\ker_{L^2_b} D_{b,\eta}.
 \label{ucr.1a}\end{equation}
 \label{ucr.1}\end{theorem}
 \begin{proof}
 As already noticed, $\ker D_{v,\eta}$ is a vector bundle over $S_{\eta}$, which ensures that Assumption~1 of \cite[Theorem~4.5]{ARS1} holds, while Assumption~2 of this theorem holds thanks to Proposition~\ref{dec.8d}.  Hence, the result follows from \cite[Theorem~4.5 and Corollary 5.2]{ARS1}, \cf the proof of \cite[Theorem~3.5]{MR1}.
 \end{proof}

 \section{Cusp degeneration of analytic torsion} \label{cdat.0}
 
Let $T(M,\hE_{m,n},g_{\fc,\epsilon},h_{\epsilon})$ be the analytic torsion of $(M,\hE_{m,n},g_{\fc,\epsilon},h_{\epsilon})$.  In this section, we will study the limiting behavior of $T(M;\hE_{m,n}, g_{\fc,\epsilon},h_{\epsilon})$ as $\epsilon \searrow 0$.  To do so, we need first to give a better description of the small eigenvalues occurring in Theorem~\ref{ucr.1}.   First, on the fibered cusp end $(T,\infty)_r\times Y_{\eta}$, we can compute $L^2$-cohomology as follows.  As in \cite{MR1}, we can use polyhomogeneous forms, so the $L^2$-cohomology can be computed using the complex
\begin{multline}
  L^2\cA'_{\phg}\Omega^q((T,\infty)_r\times Y_{\eta};E_{m,n})= \\ \{\nu \in L^2\cA_{\phg}\Omega^q((T,\infty)_r\times Y_{\eta};E_{m,n})\; | \; d\nu\in L^2\cA_{\phg}\Omega^{q+1}((T,\infty)_r\times Y_{\eta};E_{m,n}) \},
\label{se.1}\end{multline}
where $L^2\cA_{\phg}\Omega^q((T,\infty)_r\times Y_{\eta};E_{m,n})$ is the space of smooth $L^2$-forms admitting a polyhomogeneous expansion in $x=\frac1r$ in the sense of \cite{MelroseAPS}.
The differential of this complex decomposes as
$$
     d= d_r + d_{Y_{\eta}}
$$ 
with $d_r$ the differential in the $(T,\infty)_r$ factor and $d_{Y_{\eta}}$ the differential in the $Y_{\eta}$ factor.  Using that the natural connection of the fibered bundle $\phi_{\eta}: Y_{\eta}\to S_{\eta}$ is flat, we see that the differential $d_{Y_{\eta}}$ further decomposes as 
$$
     d_{Y_{\eta}}= d_{\cC} +\td_{S_{\eta}}+d_{Y_{\eta}/S_{\eta}}
$$  
where $d_{\cC}$ is the differential of the complex \eqref{dec.4} with $\rho=1$, $d_{Y_{\eta}/S_{\eta}}$ is corresponding to the vertical differential of the smooth coefficients in front of the basis generating the complex $\cC$ and $\td_{S_{\eta}}$ is the remaining horizontal differential.  Using this decomposition, we can in particular induce a double complex out of \eqref{se.1} with  differentials $d_A:=d_{Y_{\eta}/S_{\eta}}$ and $d_B:= d_r+\td_{S_{\eta}}+d_{\cC}$ with bi-degree given by $(W, N_r+N_{S_{\eta}}+N_{Y_{\eta}/S_{\eta}}-W)$, where $W$ is the weight operator defined earlier on and $N_r$, $N_{S_{\eta}}$ and $N_{Y_{\eta}/S_{\eta}}$ are the number operators in the factors $(T,\infty)_r$ the base an the fibers of the fiber bundle $Y_{\eta}\to S_{\eta}$.  The first page of the corresponding spectral sequence is 
$$
E_1= L^2\cA'_{\phg}\Omega^*((T,\infty)_r\times S_{\eta}; \ker D_{v,\eta})
$$
with $d_1=d_B=d_r+\td_{S_{\eta}}+d_\cC$.  In fact, this spectral sequence degenerates at the second page, which is just the cohomology of $(E_1,d_1)$.  We can again decompose this differential by
$$
    d_1= d_{1,A}+ d_{1,B} \quad \mbox{with} \; d_{1,A}= \td_{S_{\eta}}+d_{\cC} \quad \mbox{and} \quad d_{1,B}=d_r,
$$
inducing on $(E_1,d_1)$ a structure of double complex with bi-degree $(N_{S_{\eta}}+N_{Y_{\eta}/S_{\eta}},N_r)$.  The corresponding spectral sequence has first page 
$$
   E_1'= L^2\cA_{\phg}'\Omega((T,\infty)_r, \ker\eth_{S_{\eta}})
$$
with differential $d_1'=d_{1,B}=d_r$, where a basis for $\ker\eth_{S_{\eta}}$ is given by Lemmas~\ref{dec.5e} and \ref{dec.5h}.
This spectral sequence degenerates at the second page $E_1''$ so that the $L^2$-cohomology of the fibered cusp end is identified with $E_1''$, that is, with the cohomology of the complex
$$
      L^2\cA_{\phg}'\Omega^*((T,\infty)_r;\ker\eth_{S_{\eta}})
$$
with differential $d_1'=d_r$.  By \cite[p.501]{HHM}, we deduce the following.

\begin{proposition}
When $r_1>0$, the $L^2$-cohomology of $((T,\infty)_r\times Y_{\eta}; E_{m,n})$ is trivial unless $m_1=\cdots=m_{r_1}$ and $n_j\in\{2m_1-\bn_j,2m_1+2+\bn_j, \bn_j-2m_1-2\}$ for all $j\in \{1,\ldots, r_2\}$, in which case it is given by 
$$
  H^*_{(2)}((T,\infty)_r\times Y_{\eta};E_{m,n})\cong \langle \bw_{m,l}\lrp{\bigwedge_{j\in J}e^{-\tv_j}dz_j}\wedge \lrp{\bigwedge_{j\in \overline{J}}e^{-\tv_j}d\bz_j}\rangle\wedge \cH^*(S_{\eta})
$$
with $J, \overline{J}$ and $l$ as in Lemma~\ref{dec.5e} and $\bw_{k,l}$ is $w_{k,l}$ with $r=1$.  If instead $r_1=0$, but $\bn_1=\cdots=\bn_{r_2}=0$ and $n_1\ge 1$, then  assuming without loss of generality that $n_1\ge \cdots \ge n_{r_2}$, the $L^2$-cohomology of $((T,\infty)_r\times Y, E_{m,n})$ is trivial unless $n_j\in \{n_1, n_1-2\}$ for all $j$, in which case it is given by 
$$
    H^*_{(2)}((T,\infty)_r\times Y_{\eta};E_{m,n})\cong \langle \bw_{0,n} (e^{-\tv}d\bz)^{\alpha}\rangle \wedge \cH^*(S_{\eta})
$$
with $\alpha\in \{0,1\}^{r_2}$ such that 
$$
   n_j=n_1\; \Longrightarrow \; \alpha_j=\alpha_1, \quad n_j=n_1-2 \; \Longrightarrow \; \alpha_j=0, \; \alpha_1=1.
$$
\label{cdat.2}\end{proposition}
If we forget about the factor of $(T,\infty)_r$, notice that the above spectral sequence argument also shows that
\begin{equation}
       H^*(Y_{\eta};E_{m,n})\cong \ker\eth_{S_{\eta}},
\label{cdat.2b}\end{equation}
so Lemmas~\ref{dec.5e} and \ref{dec.5h} give an explicit description of $H^*(Y_{\eta};E_{m,n})$.  It also induces a natural orthogonal decomposition
\begin{equation}
 H^*(Y_{\eta};E_{m,n})= H_+^*(Y_{\eta};E_{m,n})\oplus H_-^*(Y_{\eta};E_{m,n}),
\label{deco.1}\end{equation}
when the cohomology is not trivial, where 
\begin{equation}
\begin{aligned}
 H_-^*(Y_{\eta};E_{m,n})&= \bw_{0,n-l} \left(\bigwedge_{i=1}^{r_1} e^{-\tu_i}\rho dx_i\right)\wedge \left( \bigwedge_{j\notin J} e^{-\tv_j}\rho dz_j \right)\wedge \left( \bigwedge_{j\notin \overline{J}} e^{-\tv_j}\rho d\bz_j\right)\wedge \cH^*(S_{\eta}),  \\
 H_+^*(Y_{\eta};E_{m,n})&=  \bw_{m,l}\left( \bigwedge_{j\in J} e^{-\tv_j}\rho dz_j\right)\wedge  \left( \bigwedge_{j\in \overline{J}} e^{-\tv_j}\rho d\bz_j\right)\wedge\cH^*(S_{\eta})
\end{aligned}
\label{deco.2}\end{equation}
in the setting of Lemma~\ref{dec.5e}, while in the setting of \eqref{dec.5k}, 
\begin{equation}
\begin{aligned}
H_-^*(Y_{\eta};E_{m,n})&= \bw_{0,0}\left( \bigwedge_{j=1}^{r_2} e^{-\tv_j}\rho dz_j \right)\wedge (e^{-\tv}\rho d\bz)^{\beta}  \wedge \cH^*(S_{\eta}), \\
H_+^*(Y_{\eta};E_{m,n})&=  \bw_{0,n}(e^{-\tv}\rho d\bz)^{\alpha} \wedge \cH^*(S_{\eta}).
\end{aligned}
\label{deco.3}\end{equation}  
In particular, notice that Proposition~\ref{dec.8d} implies that 
\begin{equation}
 \ker_{L^2_b}D_{b,\eta}\cong H^*_+(Y_{\eta};E_{m,n})\oplus H^*_-(Y_{\eta};E_{m,n}).
\label{deco.3b}\end{equation}
The decomposition \eqref{deco.1} also implies that
\begin{equation}
H^*(\pa \bX; E_{m,n})= H^*_+(\pa\bX;E_{m,n})\oplus H^*_-(\pa\bX;E_{m,n})
\label{deco.4}\end{equation}
with 
$$
   H^*_{\pm}(\pa\bX;E_{m,n})= \bigoplus_{\eta\in\mathfrak{P}_{\Gamma}}  H_{\pm}^*(Y_{\eta};E_{m,n}).
$$
Let $\pr_{\pm}: H^*(\pa\bX;E_{m,n})\to H_{\pm}^*(\pa\bX;E_{m,n})$ be the induced projections.  Lemmas~\ref{dec.5e} and \ref{dec.5h} yields the following description of these spaces.
\begin{lemma}\label{lem-bound-cohom}
When $r_1>0$, the cohomology of $(\pa\bX,E_{m,n})$ is trivial unless $m_1=\cdots=m_{r_1}$ and $n_j\in\{2m_1-\bn_j,2m_1+2+\bn_j, \bn_j-2m_1-2\}$ for all $j\in \{1,\ldots,r_2\}$, in which case it is given by 
$$
\begin{aligned}
 H_-^*(\pa \bX;E_{m,n}) &\cong \bigoplus_{\eta\in\mathfrak{P}_{\Gamma}} \langle \bw_{0,n-l}\lrp{\bigwedge_{i=1}^{r_1}e^{-\tu_i}dx_i}\wedge \lrp{\bigwedge_{j\notin J}e^{-\tv_j}dz_j}\wedge \lrp{\bigwedge_{j\notin \overline{J}}e^{-\tv_j}d\bz_j}\rangle\wedge \cH^*(S_{\eta}) \\
 H_+^*(\pa \bX;E_{m,n}) &\cong \bigoplus_{\eta\in\mathfrak{P}_{\Gamma}} \langle \bw_{m,l}\left( \bigwedge_{j\in J} e^{-\tv_j}\rho dz_j\right) \wedge \left( \bigwedge_{j\in \overline{J}} e^{-\tv_j}\rho d\bz_j\right) \rangle \wedge \cH^*(S_{\eta}),
 \end{aligned}
$$
for $J=\{ j\in\{1,\ldots, r_2\} \; | \; n_j=\bn_j+2m_1+2\}$ and $\overline{J}= \{ j\in \{1,\ldots, r_2\}\; | \; n_j= 2m_1+2+\bn_j\}$.  If instead $r_1=0$, but $\bn_1=\cdots=\bn_{r_2}=0$,  $n_1\ge 1$ and we assume without loss of generality that $n_1\ge \cdots n_{r_2}$, then the cohomology of $(\pa \bX, E_{m,n})$ is trivial unless $n_j\in\{n_1,n_1-2\}$ for all $j$, in which case it is given by
$$
\begin{aligned}
   H^*_-(\pa\bX;E_{m,n})&\cong \bigoplus_{\eta\in\mathfrak{P}_{\Gamma}} \langle \bw_{0,0}\lrp{\bigwedge_{j=1}^{r_2}e^{-\tv_j}dz_j} \wedge (e^{-\tv}d\bz)^{\beta} \rangle\wedge \cH^*(S_{\eta}), \\
   H_+^*(\pa \bX;E_{m,n})&=  \bigoplus_{\eta\in\mathfrak{P}_{\Gamma}}  \langle \bw_{0,n}(e^{-\tv}\rho d\bz)^{\alpha} \rangle \wedge \cH^*(S_{\eta}).
\end{aligned}   
$$
where $\alpha, \beta\in\{0,1\}^{r_2}$ is such that
$$
 \begin{aligned}
      n_j=n_1 \;&\Longrightarrow \; \alpha_j=\alpha_1, \quad \beta_j= \beta_1, \\
      n_j= n_1-2 & \Longrightarrow \;  \alpha_j=0, \alpha_1=1, \beta_j=1, \beta_1=0.
\end{aligned}   
$$
\label{se.4}\end{lemma}

We can use this information to compute the cohomology group $H^*(\bX;E_{m,n})$ using known results and the Mayer-Vietoris long exact sequence in $L^2$-cohomology
\begin{equation}
\xymatrix{
\cdots H^q_{(2)}(X;E_{m,n}) \ar[r] & H^q_{(2)}((T,\infty)_r\times \pa\bX;E_{m,n})\oplus H^q(\bX;E_{m,n}) \ar[r] &  H^q(\pa\bX;E_{m,n}) \cdots.
}
\label{se.3}\end{equation}
  This yields the following.
\begin{theorem}
Suppose that $r_1>0$ or that $r_1=0$ with $\bn_1=\cdots=\bn_{r_2}=0$ and $n\ne 0$.  Then there is an isomorphism 
\begin{equation}
   H^*(\bX; E_{m,n})\cong H^*_{(2)}(X;E_{m,n}) \oplus H^*_-(\pa\bX;E_{m,n}).
\label{harder.1}\end{equation}
Furthermore, if $n_j\ne \bn_j$ for some $j\in\{1,\ldots,r_2\}$, then
\begin{equation}\label{vanish-cohom}
H^*_{(2)}(X;E_{m,n})= \{0\}
\end{equation}
and this simplifies to
\begin{equation}
  H^*(\bX;E_{m,n})\cong H^*_-(\pa\bX;E_{m,n}).
\label{harder.2}\end{equation} 
\label{cdat.1}\end{theorem}
\begin{proof}
By a result of Harder \cite[Remark (3), p.159]{Harder}, we know that the image of the restriction map 
$$
    \iota_{\pa\bX}^*: H^*(\bX;E_{m,n})\to H^*(\pa \bX;E_{m,n})
$$
is identified with $H^*_-(\pa\bX; E_{m,n})$ via the projection $\pr_-: H^*(\pa \bX;E_{m,n})\to H^*_-(\pa X;E_{m,n})$,
\begin{equation}
   \iota_{\pa \bX}^* H^*(\bX;E_{m,n})\cong H_-^*(\pa\bX;E_{m,n}).
\label{harder.3}\end{equation}
On the other hand, Proposition~\ref{cdat.2} and Lemma~\ref{lem-bound-cohom} show that the restriction map also induces an isomorphism
\begin{equation}
 H^*_{(2)}((T,\infty)_r\times Y_{\eta}; E_{m,n})\cong H^*_+(\pa\bX;E_{m,n}). 
\label{iso.34b}\end{equation}
By the decomposition \eqref{deco.4}, this means that the boundary homomorphism of the long exact sequence \eqref{se.3} is trivial, yielding the isomorphism \eqref{harder.1}. To prove
\eqref{vanish-cohom}, 
let $\tau\colon G_\infty\to \GL(V_\tau)$ be an irreducible finite dimensional
representation. Let $E_\tau\to X$ be the flat vector bundle which is associated
to $\tau|_\Gamma$. Equip $E_\tau$ with the Hermitian fibre metric defined in
\cite{MM1963}. Let
\[
\Delta_{q}(\tau)\colon \Lambda^q(X,E_\tau)\to \Lambda^q(X,E_\tau)
\]
be the Laplacian acting in the space of $E_\tau$-valued $q$-forms. Let
$L^2_{\di}\Lambda^q(X,E_\tau)$ be the subspace of $L^2\Lambda^q(X,E_\tau)$
which is spanned by the square integrable eigenforms of $\Delta_q(\tau)$.
Let $L^2_{\di}(\Gamma\bsl G_\infty)$ be the subspace of $L^2(\Gamma\bsl G_\infty)$,
which is spanned by the irreducible subrepresentations of the right regular
representation of $G_\infty$ on $L^2(\Gamma\bsl G_\infty)$. Then we have
\begin{equation}\label{disc-spec}
L^2_{\di}\Lambda^q(X,E_\tau)\cong
\left(L^2_{\di}(\Gamma\bsl G_\infty)\otimes V_\tau\right)^{K_\infty}.
\end{equation}
Let $\Delta_{q,\di}(\tau)$ be the restriction of $\Delta_q(\tau)$ to
$L^2_{\di}\Lambda^q(X,E_\tau)$. Now assume that $\tau\not\cong\tau\circ\vartheta$ where $\vartheta$ is the standard Cartan involution of $G_{\infty}$ with respect to $K_{\infty}$. Using
\eqref{disc-spec}
it follows as in the proof of \cite[Lemma 4.1]{BV} that there exists $c>0$
such that 
\begin{equation}\label{disc-spec1}
\Spec(\Delta_{q,\di}(\tau))\subset [c,\infty)
\end{equation}
for all $q=0,...,n$. In particular, it follows that
\begin{equation}\label{l2-cohom}
H^\ast_{(2)}(X,E_\tau)=
\mathcal{H}^\ast_{(2)}(X,E_\tau)=0,
\end{equation}
where the right hand side denotes the space of square integrable harmonic forms.
Now assume that $r_2>0$ and that there exists $j\in\{1,...,r_2\}$ with
$n_j\neq\bar n_j$. Then $\varrho_{m,n}$ is not self-conjugate, that is, $\varrho_{m,n}\ncong \varrho_{m,n}\circ\vartheta$, so
\eqref{vanish-cohom} follows from \eqref{l2-cohom}. 
\end{proof}
\begin{remark}
Comparing \eqref{se.3} with the long exact sequence in cohomology for the pair $(\bX,\pa\bX)$, we can also deduce from \eqref{harder.3} that
$$
      H^*_{(2)}(X;E_{m,n})\cong \Im \lrp{H^*_c(X;E_{m,n})\to H^*(X;E_{m,n})}.
$$
\label{harder.4}\end{remark}

Using the Mayer-Vietoris long exact sequence in cohomology
\begin{equation}
\xymatrix{
\cdots \ar[r]^-{\pa_{q-1}} & H^q(M;\hE_{m,n}) \ar[r]^-{i_q} & H^q(\bX;E_{m,n})\oplus H^q(\bX;E_{m,n}) \ar[r]^-{j_q} & H^q(\pa\bX;E_{m,n}) \ar[r]^-{\pa_q} & \cdots
}
\label{se.5}\end{equation}
and Proposition~\ref{dec.8d}, we deduce the following.
\begin{lemma}
If $r_1>0$ or $r_1=0$ with $\bn_1=\cdots=\bn_{r_2}=0$ and $n\ne 0$, then
$$
\begin{aligned}
   H^*(M;\hE_{m,n})&\cong  H^*_{(2)}(X;E_{m,n})\bigoplus H^*_{(2)}(X;E_{m,n}) \bigoplus \lrp{\bigoplus_{\eta\in\mathfrak{P}_{\Gamma}}\ker_{L^2_b}D_{b,\eta}}, \\
           &\cong \ker_{L^2} \eth_{\fc,0} \bigoplus \lrp{\bigoplus_{\eta\in\mathfrak{P}_{\Gamma}}\ker_{L^2_b}D_{b,\eta}},
\end{aligned}   
$$
where $\eth_{\fc,0}$ is the operator corresponding to $\eth_{\fc}$ on each copy of $X$ in $M\setminus \pa \bX$.
\label{se.6}\end{lemma}
We deduce from this lemma and Theorem~\ref{ucr.1} that the eigenspaces associated to small eigenvalues are cohomological, that is, all the small eigenvalues are zero.  As in \cite{MR1}, this allows us to apply \cite[Corollary~11.3]{ARS1}.  However, one important hypothesis in this corollary is that the base $S_{\eta}$ of the fiber bundle $\phi_{\eta}: Y_{\eta}\to S_{\eta}$ has to be even dimensional to ensure that some terms in the short time asymptotic expansion of the trace of the heat kernel vanish.  Our setting, in the other hand, is very special within the framework considered in \cite[Theorem~11.2]{ARS1}.  This will allow us to establish directly the vanishing of the these terms in the asymptotic expansion of the trace of the heat kernel without assuming that $S_{\eta}$ is even dimensional. 

The first useful special feature of our setting is that the family of metric $g_{\fc,\epsilon}$ is not just asymptotic to the model \eqref{model.1}. It is in fact exactly given by \eqref{model.1} for $\rho=\sqrt{x^2+\epsilon^2}$ sufficiently small.  Moreover, as already observed, the fiber bundle $\phi_{\eta}: Y_{\eta}\to S_{\eta}$ has trivial curvature in the sense of \cite{BGV}.  Consequently, there is no curvature term in the asymptotic model \eqref{dec.2} of the corresponding Hodge-deRham operator in the limit $\rho\to 0$ and this operator is in fact exactly given by \eqref{dec.2} in a small region near $\rho=0$.  

To study the asymptotic behavior of the heat kernel near $\rho=0$ in the limit $t\to 0$, this means that we can use the model operator \eqref{model.2}, namely
\begin{equation}
\hD_{\fc,\epsilon}= \lrp{\begin{array}{cc} \rho^{-1}\eth_{Y_{\eta}/S_{\eta}}+ \eth_{S_{\eta}} & \frac{x}{\rho}\left(W-\frac{d_{\bbK}}{2}\right)-\rho\frac{\pa}{\pa x} \\
     \rho\frac{\pa}{\pa x}+ \frac{x}{\rho}\left(W-\frac{d_{\bbK}}2 \right) & -\rho^{-1}\eth_{Y_{\eta}/S_{\eta}}- \eth_{S_{\eta}} \end{array}}
\label{model.3}\end{equation}
seen as defined on $Y_{\eta}\times \bbR_x$ for $\epsilon$ small.  Since $\eth_{S_{\eta}}$ commutes with $W$ and anti-commutes with $\eth_{Y_{\eta}/S_{\eta}}$, we see that 
\begin{equation}
    \hD_{\fc,\epsilon}^2= \hD_{c,\epsilon}^2 + \eth^2_{S_{\eta}},
\label{model.3b}\end{equation}
where 
$$
 \hD_{c,\epsilon}:=  \lrp{\begin{array}{cc} \rho^{-1}\eth_{Y_{\eta}/S_{\eta}} & \frac{x}{\rho}\left(W-\frac{d_{\bbK}}{2}\right)-\rho\frac{\pa}{\pa x} \\
     \rho\frac{\pa}{\pa x}+ \frac{x}{\rho}\left(W-\frac{d_{\bbK}}2 \right) & -\rho^{-1}\eth_{Y_{\eta}/S_{\eta}} \end{array}}
$$
is an operator anti-commuting with $\eth_{S_{\eta}}$.  The corresponding heat kernel is therefore given by
$$
   e^{-t\hD^2_{\fc,\epsilon}}= e^{-t\hD^2_{c,\epsilon}}e^{-t\eth^2_{S_{\eta}}}.
$$
For analytic torsion, we are in fact interested in the weighted version
$$
   (-1)^N N e^{-t\hD_{\fc,\epsilon}^2}= (-1)^N N \lrp{e^{-t \hD^2_{c,\epsilon}}e^{-t \eth_{S_{\eta}}^2}},
$$
where $N$ is the number operator multiplying a form (of pure degree) by its degree.  This can be rewritten
\begin{equation}
 (-1)^N N e^{-t\hD^2_{\fc,\epsilon}}= (-1)^{N_x+ N_{Y_{\eta}/S_{\eta}}+ N_{S_{\eta}}}(N_x+N_{Y_{\eta}/S_{\eta}}+N_{S_{\eta}})\lrp{e^{-t\hD_{c,\epsilon}}e^{-t\eth_{S_{\eta}}^2}},
\label{cdat.3}\end{equation}
where $N_x$ is the number operator in the factor $\bbR_x$, while as before $N_{Y_{\eta}/S_{\eta}}$ and $N_{S_{\eta}}$ are the number operators in the fibers and the base of the fiber bundle $\phi_{\eta}: Y_{\eta}\to S_{\eta}$.  To understand this operator, it suffices to understand it when acting on each of the eigenspaces of $(N_x+ N_{Y_{\eta}/S_{\eta}})$ and $W$.  

We can also decompose in terms of the eigenbundles of $\eth_{Y_{\eta}/S_\eta}$.  Let us first see what happens in a given fiber $\phi_{\eta}^{-1}(s)$ of the fiber bundle \eqref{fb.1}.  There, the restriction $(\eth_{Y_\eta/S_\eta})^2_s$ of $\eth_{Y_\eta/S_\eta}^2$ is the Hodge Laplacian associated to the trivial flat Hermitian vector bundle $\phi_{\eta}^{-1}(s)\times \bbC^{\rank E_{m,n}}$ on a flat torus.  In this case, 
$$
  \dim_\bbC \ker (\eth_{Y_\eta/S_\eta})^2_s=  \dim_\bbC \ker (\eth_{Y_\eta/S_\eta})_s= 2^{d_\bbK}\rank E_{m,n}
$$
and $(\eth_{Y_\eta/S_\eta})^2_s$ is canonically identified with $2^{d_\bbK}\rank E_{m,n}$ copies of the scalar Laplacian $\Delta_{\phi_{\eta}^{-1}(s)}$ on $\phi_{\eta}^{-1}(s)$.  Hence, the eigenspace of $(\eth_{Y_\eta/S_\eta})^2_s$ associated to the eigenvalue $\nu$ is canonically of the form
\begin{equation}
  \ker((\eth_{Y_\eta/S_\eta})^2_s-\nu)\cong \ker(\Delta_{\phi_\eta^{-1}(s)}-\nu)\otimes \ker(\eth_{Y_\eta/S_\eta})_s.
\label{eigenbundle.2}\end{equation}
The torus $\phi_{\eta}^{-1}(s)$ can be taken to be $(\Gamma_{\eta}\cap N_\infty)\setminus N_\infty$, but notice that the flat metric on it depends on $s$.  In other words, the fibers of $\phi_\eta: Y_\eta\to S_\eta$ are not isometric with their induced metrics.  Nevertheless, under the identification  $\phi_{\eta}^{-1}(s)\cong (\Gamma_{\eta}\cap N_\infty)\setminus N_\infty$, notice that for $t\ne s$, $\ker(\Delta_{\phi_\eta^{-1}(s)}-\nu)$, can also be seen as an eigenspace of $\Delta_{\phi_\eta^{-1}(t)}$ but possibly with a different eigenvalue.  These eigenspaces can be combined to form a flat vector bundle on $S_\eta$ as follows.

Denote by $1$ the element of $S_\eta= \Gamma_{\eta,T} \setminus T_{\infty}(1)/K_T$ corresponding to taking 
$$
   \lambda_1=\cdots= \lambda_{r_1}=\mu_1=\cdots=\mu_{r_2}=1.
$$
The fundamental group $\Gamma_{\eta,T}$ of $S_\eta$ acts on $\phi_\eta^{-1}(s)$ in such a way that
$$
    Y_\eta= \Gamma_{\eta,T}\setminus ( (T_{\infty}(1)/K_T)\times \phi_\eta^{-1}(1) )
$$
with respect to the diagonal action of $\Gamma_{\eta,T}$ on $ (T_{\infty}(1)/K_T)\times \phi_\eta^{-1}(1)$.  The action of $\Gamma_{\eta,T}$ on $\phi_{\eta}^{-1}(1)$ also induces an action on $\ker(\Delta_{\phi_\eta^{-1}(1)}-\nu)$, so for $\nu$ an eigenvalue of $\Delta_{\phi_\eta^{-1}(1)}$, we can consider the flat vector bundle 
\begin{equation}
   H_\nu:= \Gamma_{\eta,T} \setminus ((T_{\infty}(1)/K_T)\times \ker(\Delta_{\phi_{\eta}^{-1}(1)}-\nu))
\label{eigenbundle.3}\end{equation}
on $S_{\eta}$, where the quotient is taken with respect to the diagonal action of $\Gamma_{\eta,T}$ on $(T_{\infty}(1)/K_T)\times \ker(\Delta_{\phi_{\eta}^{-1}(1)}-\nu)$.  For $s\in S_\eta$, the fiber of $H_\nu$ above $s$ corresponds to an eigenspace of $\Delta_{\phi_\eta^{-1}(s)}$ with eigenvalue $\nu(s)$ possibly different from $\nu$.  
\begin{remark}
The flat vector bundle $H_\nu$ comes with a natural Hermitian structure $h_\nu$ induced by the metric $g_{\fc}$.   This is the Hermitian structure that we will consider on $H_\nu$.
\end{remark}

Together with the canonical identification \eqref{eigenbundle.2}, we see that the flat vector bundles $H_\nu$ as $\nu$ varies yield the decomposition
\begin{equation}
  L^2(Y_\eta; \Lambda^*(T^*(Y_\eta/S_\eta))\otimes E_{m,n})= \bigoplus_{\nu\in \Spec(\Delta_{\phi_\eta^{-1}(1)})} L^2(S_\eta; H_\nu\otimes \ker\eth_{Y_\eta/S_\eta}).
\label{decom.1}\end{equation}

 Correspondingly, the restriction of $\eth_{S_{\eta}}$ acting on sections of the vector bundle $H_\nu\otimes \ker(\eth_{Y_{\eta}/S_{\eta}})$ is given by
$$
   \eth_{S_{\eta},\nu}= \eth_{\cC}+ \widetilde{\eth}_{S_{\eta},\nu} \quad \mbox{with} \quad \eth_{S_{\eta},\nu}^2= \eth^2_{\cC}+ \widetilde{\eth}_{S_{\eta},\nu}^2,
$$
where $\eth_{\cC}$ acts trivially on $H_\nu$  and $\eth_{S_{\eta},\nu}$ is a version of $\eth_{S_{\eta}}$ twisted by the  flat vector bundle $H_{\nu}$ with Hermitian structure $h_\nu$.
Hence, to understand the pointwise trace of \eqref{cdat.3}, it suffices to understand the sum 
\begin{equation}
 \sum_{q=0}^{r_1+r_2-1} (-1)^{p+q}(p+q)\tr\lrp{ (e^{-t\widetilde{\eth}_{S_{\eta},\nu}^2})_{\Omega^q(S_{\eta};L\otimes H_{\nu})} },
\label{cdat.4}\end{equation}
where $L\to S_{\eta}$ is a flat line bundle on $S_{\eta}$, seen as a subbundle of $\cC^*\to S_{\eta}$, which corresponds to a common  eigenspace of $\eth_{\cC}$  and $W$ contained in $\{ \omega \; | \; (N_x+ N_{Y_{\eta}/S_{\eta}})\omega=p\omega\}$.  But by \eqref{dec.5f}, in a trivialization of such a flat line bundle $L$, 
\begin{equation}
     \td_{S_{\eta},\nu}= d_{\nu}+a_L \Id_{\nu}
\label{cdat.4a}\end{equation}
for some parallel 1-form $a_L\in \Omega^1(S_{\eta})$, where $\Id_{\nu}$ is the identity section of $H_{\nu}$ and $d_{\nu}$ is the differential on $\Omega^*(S_{\eta};H_{\nu})$, so
\begin{equation}
   \tr\lrp{ e^{-t\widetilde{\eth}_{S_{\eta},\nu}^2}|_{\Omega^q(S_{\eta};L\otimes H_{\nu})} }= \tr(e^{-t\Delta_{\nu}}) e^{-t |a_L|^2} \lrp{\begin{matrix} r_1+r_2-1 \\ q \end{matrix}},
\label{cdat.5}\end{equation}
where $\Delta_{\nu}$ is the scalar Laplacian of $(S_{\eta},g_{S_{\eta}}, H_{\nu},h_\nu)$ and $|a_L|$ is the pointwise norm of $a_L$ with respect to the metric $g_{S_{\eta}}$.  Hence the sum \eqref{cdat.4} will vanish provided we can show that 
\begin{equation}
   \sum_{q=0}^{r_1+r_2-1} (-1)^{p+q}(p+q) \lrp{\begin{matrix} r_1+r_2-1 \\ q \end{matrix}}
\label{cdat.6}\end{equation}
vanishes.  When $r_1+r_2-1>1$, the vanishing of \eqref{cdat.6} is a consequence of the binomial theorem, since
$$
\begin{aligned}
\sum_{q=0}^{r_1+r_2-1}(-1)^{p+q}(p+q)\lrp{\begin{matrix} r_1+r_2-1 \\ q \end{matrix}}&= 
  (-1)^pp\sum_{q=0}^{r_1+r_2-1}(-1)^q \lrp{\begin{matrix} r_1+r_2-1 \\ q \end{matrix}}  \\ & \quad+ (-1)^p \sum_{q=0}^{r_1+r_2-1} (-1)^qq \lrp{\begin{matrix} r_1+r_2-1 \\ q \end{matrix}} \\
  &= (-1)^pp(1-1)^{r_1+r_2-1}  \\ & \quad +(-1)^p \sum_{q=1}^{r_1+r_2-1} (-1)^qq\lrp{\begin{matrix} r_1+r_2-1 \\ q \end{matrix}} \\
  &= 0 +(-1)^p \sum_{q=1}^{r_1+r_2-1}(-1)^q(r_1+r_2-1)\lrp{\begin{matrix} r_1+r_2-2 \\ q-1 \end{matrix}} \\
  &= (-1)^p (r_1+r_2-1) \sum_{q=1}^{r_1+r_2-1} (-1)^q \lrp{\begin{matrix} r_1+r_2-2 \\ q-1 \end{matrix}} \\
  &= (-1)^p(r_1+r_2-1) \sum_{q=0}^{r_1+r_2-2}(-1)^{q+1} \lrp{\begin{matrix} r_1+r_2-2 \\ q \end{matrix}} \\
  &= (-1)^{p+1} (r_1+r_2-1)(1-1)^{r_1+r_2-2}=0.
\end{aligned}
$$

If $r_1+r_2-1=1$, this argument does not show that \eqref{cdat.6} vanishes.  However, we should notice that the flat line bundles $L\to S_{\eta}$ arising in \eqref{cdat.4} come in pairs using the Hodge star operator of $\cC^*$ and $\bbR_x$.  Namely, if $L\to S_{\eta}$ is such that 
\begin{equation}
  (W-\frac{d_{\bbK}}2)(\sigma)=\nu \sigma \quad (N_x+ N_{Y_{\eta}/S_{\eta}})\sigma= p \sigma
\label{cdat.7}\end{equation}
for each of its sections, then there is a dual line bundle $L^*$ also arising in the decomposition \eqref{cdat.4} such that
\begin{equation}
  (W-\frac{d_{\bbK}}2)(\sigma)=-\nu \sigma \quad (N_x+ N_{Y_{\eta}/S_{\eta}})\sigma= (d_{\bbK}+1-p) \sigma.
\label{cdat.8}\end{equation}
If the flat connection of $L\otimes H_{\nu}$ is locally given by \eqref{cdat.4a}, then the one of $L^*\otimes H_{\nu}$ is given by
\begin{equation}
  d_{\nu}-a_L\Id_{\nu}
\label{cdat.9}\end{equation}
since $L^*$ is the dual of $L$.  Alternatively this can be seen directly from the description of $\td_{S_{\eta}}$ in \eqref{dec.5f}, see also \eqref{dec.5g} and \eqref{dec.5j}.  In particular, from \eqref{cdat.5}, we see that
\begin{equation}
  \tr \lrp{e^{-t\widetilde{\eth}_{S_{\eta}}^2}|_{\Omega^q(S_{\eta};L^*\otimes H_{\nu})}}=  \tr \lrp{e^{-t\widetilde{\eth}_{S_{\eta}}^2}|_{\Omega^q(S_{\eta};L\otimes H_{\nu})}}.
\label{cdat.10}\end{equation}
However, since \eqref{cdat.7} is replaced by \eqref{cdat.8} for sections of $L^*$, we see from \eqref{dec.8} that the contribution \eqref{cdat.6} coming from $L$, namely
\begin{equation}
   \sum_{q=0}^{1} (-1)^{p+q}(p+q) \lrp{\begin{matrix} r_1+r_2-1 \\ q \end{matrix}}= (-1)^{p}p+ (-1)^{p+1}(p+1)=(-1)^{p+1}
\label{cdat.10b}\end{equation}
since $r_1+r_2-1=1$,
 becomes the contribution 
\begin{equation}
   \sum_{q=0}^{1} (-1)^{d_{\bbK}+1-p+q}(d_{\bbK}+1-p+q) \lrp{\begin{matrix} r_1+r_2-1 \\ q \end{matrix}}= (-1)^{d_{\bbK}+1-p+1}= (-1)^{d_{\bbK}+p}
\label{cdat.11}\end{equation}
for $L^*$.  Therefore, if $d_{\bbK}$ is even, we see that \eqref{cdat.11} exactly cancel \eqref{cdat.10b}.  

To summarize the discussion so far, provided $r_1+ r_2-1>1$ or that $r_1+r_2-1=1$ and $d_{\bbK}$ is even (\ie $r_1=2$ and $r_2=0$ or $r_1=0$ and $r_2=2$), we see that the pointwise trace of \eqref{cdat.3} identically vanishes,
\begin{equation}
  \tr \lrp{ (-1)^N N e^{-t\hD^2_{\fc,\epsilon}}}=0.
\label{cdat.12}\end{equation}
This implies the following result for the heat kernel of $D^2_{\fc,\epsilon}$.
\begin{theorem} Suppose that $r_1>0$ or that $r_1=0$ with $\bn_1=\cdots=\bn_{r_2}=0$ and $n\ne 0$.
  Suppose also that $r_1+r_2-1>1$ or that $(r_1,r_2)\in\{(2,0),(0,2)\}$. Then as $\epsilon \searrow 0$, the logarithm of the analytic torsion of $(M, \hE_{m,n}, g_{\fc,\epsilon}, h_{\epsilon})$ has a polyhomogeneous expansion and its finite part is given by
$$
   \FP_{\epsilon=0} \log T(M;\hE_{m,n},g_{\fc,\epsilon}, h_{\epsilon})= 2\log T(X;E_{m,n},g_{\fc}, h),
$$
where $T(X;E_{m,n},g_{\fc}, h)$ is the analytic torsion of $(X;E_{m,n},g_{\fc}, h)$.
\label{cdat.13}\end{theorem}
\begin{proof}
By the discussion above, in a region near $\rho=0$, the heat kernel of $D_{\fc, \epsilon}^2$ has the same short time asymptotic expansion as the heat kernel of the model \eqref{model.3b}.  In particular, we see from \eqref{cdat.12} that 
$$
\tr \lrp{ (-1)^N N e^{-tD^2_{\fc,\epsilon}}}
$$
has a trivial short time expansion near $\rho=0$, in particular at the front face $\mathfrak{B}_{tff}$ of the surgery heat space of \cite{ARS1}.  Thus, by \cite[Theorem~11.2]{ARS1},  $ T(M;\hE_{m,n},g_{\fc,\epsilon}, h_{\epsilon})$ has a polyhomogeneous expansion as $\epsilon\searrow 0$ with finite part given by
$$
   \FP_{\epsilon=0} \log T(M;\hE_{m,n},g_{\fc,\epsilon}, h_{\epsilon})= 2\log T(X;E_{m,n},g_{\fc}, h) +\sum_{\eta\in\mathfrak{P}_{\Gamma}} \log T(D_{b,\eta}^2),
$$
where  $T(D_{b,\eta}^2)$ is the analytic torsion defined by
$$
    \log T(D^2_{b,\eta})= \frac12 \zeta_{D^2_{b,\eta}}'(0)
$$
with 
$$
\zeta_{D^2_{b,\eta}}(s):= \frac{1}{\Gamma(s)}\int_{0}^{\infty} (t^s) {}^{R}\Tr\lrp{(-1)^N N \lrp{e^{-tD_{b,\eta}^2}-\Pi_{\ker_{L^2}D_{b,\eta}^2}}} \frac{dt}{t},  \quad \Re(s)>\frac{r_1+r_2}2,
$$
admitting a meromorphic extension to the complex plane which is holomorphic near $s=0$.  Here, $N$ is the number operator giving the form degree, but also including the vertical degree of the forms $e^{-\tu_i}\rho dx_i$,  $e^{-\tv_j}\rho dz_j$ and $e^{-\tv_j}\rho d\bz_j$ seen as sections of $\ker D_{v,\eta}$ when multiplied by $\hat{w}_{k,l}$ for some $k$ and $l$.  The vanishing of \eqref{cdat.12} however implies that
$$
\zeta_{D^2_{b,\eta}}(s)=0. \quad \forall s>\frac{r_1+r_2}2,
$$
so $\log T(D^2_{b,\eta})=0$ and the result follows.

\end{proof}

\section{Small eigenvalues} \label{sev.0}

When $\ker_{L^2} D_{b,\eta}\ne 0$, there are small eigenvalues.  However, assuming that $r_1>0$ or that $r_1=0$ with $\bn_1=\cdots=\bn_{r_2}=0$ and $n\ne 0$,  all these small eigenvalues are zero by Lemma~\ref{se.6}, which will allow us to obtain a formula relating analytic torsion and Reidemeister torsion.  We will in fact closely follow the argument in \cite{MR1}.

First, Lemma~\ref{se.6} can be used to compute the fibered cusp degeneration of Reidemeister torsion using the long exact sequence in cohomology \eqref{se.5}.  This requires choosing suitable bases for the spaces involved in this long exact sequence.  When the cohomology on the boundary is not trivial, we can pick an orthonormal basis 
\begin{equation}
\mu_{\pa \bX}^q=( \mu_{+}^q, \mu_{-}^q) 
\label{se.7}\end{equation}
for $H^q(\pa\bX;E_{m,n})$, compatible with the orthogonal decomposition \eqref{deco.4} with $\mu_{\pm}^q$  an orthonormal basis (with each element of pure degree) for $H^q_{\pm}(\pa \bX;E_{m,n})$.
Without loss of generality, we can assume that $\mu_{+}$ is Poincar\'e dual to $\mu_{-}$.
On $H^q(\bX;E_{m,n})$, we can then just take the basis $\mu_X^q=(\mu_{X,(2)}^q,\mu_{X,\operatorname{inf}}^q)$ compatible with the decomposition \eqref{harder.1} and such that 
\begin{equation}
    \pr_-\circ\iota^*_{\pa\bX} \mu_X^q=  \pr_-\circ\iota^*_{\pa\bX} \mu_{X,\operatorname{inf}}^q= \mu_{-}^q,
\label{se.8}\end{equation}
where $\mu_{X,(2)}^q$ is an orthonormal basis of $H^q_{(2)}(X;E_{m,n})\cong \ker_{L^2}\eth_{\fc}$.
On $H^q(M;\hE_{m,n})$, there is a decomposition 
\begin{equation}
   H^q(M;\hE_{m,n})=\Im \pa_{q-1}\oplus \Im i_q
\label{se.9b}\end{equation}
in terms of the long exact sequence \eqref{se.5}, 
so we can choose a basis $\mu_M^q$ compatible with this decomposition and the previous choices of bases, namely
\begin{equation}
\mu_M^q= (\pa_*(\mu_{+}^{q-1}), \mu_{M,i}^q).
\label{se.9}\end{equation}
More precisely, if $\mu_X^q=\{\nu_1,\ldots,\nu_K\}$ with $\mu^q_{X,(2)}=\{\nu_1,\ldots,\nu_{K'}\}$, then $\mu_{M,i}^q$ will be chosen such that 
$$
i_q(\mu_{M,i}^q)=\{(\nu_1,0), (0,\nu_1), \ldots (\nu_{K'},0), (0,\nu_{K'}), (\nu_{K'+1},\nu_{K'+1}),\ldots, (\nu_K,\nu_K) \}.
$$ 
Now, on $H^q(\bX;E_{m,n})\oplus H^q(\bX;E_{m,n})$, we can consider, instead of $\mu_{X}^q\oplus \mu_{X}^q,$ the equivalent basis
\begin{multline}
 \{\nu_1,0), (0,\nu_1), \ldots (\nu_{K'},0), (0,\nu_{K'}), \lrp{\frac{\nu_{K'+1}}{\sqrt{2}},\frac{\nu_{K'+1}}{\sqrt{2}}},\lrp{\frac{\nu_{K'+1}}{\sqrt{2}},-\frac{\nu_{K'+1}}{\sqrt{2}}},\ldots, \\ \lrp{\frac{\nu_K}{\sqrt{2}},\frac{\nu_K}{\sqrt{2}}},\lrp{\frac{\nu_K}{\sqrt{2}},-\frac{\nu_K}{\sqrt{2}}}  \}
\label{se.10}\end{multline}
obtained by an orthonormal transformation.  It is such that 
$$
   \ker j_q= \spane\{\nu_1,0), (0,\nu_1), \ldots (\nu_{K'},0), (0,\nu_{K'}), \lrp{\frac{\nu_{K'+1}}{\sqrt{2}},\frac{\nu_{K'+1}}{\sqrt{2}}},\ldots,  \lrp{\frac{\nu_K}{\sqrt{2}},\frac{\nu_K}{\sqrt{2}}}   \},  
$$
$$
 j_q \{ \lrp{\frac{\nu_{K'+1}}{\sqrt{2}},-\frac{\nu_{K'+1}}{\sqrt{2}}},\ldots,  \lrp{\frac{\nu_K}{\sqrt{2}},-\frac{\nu_K}{\sqrt{2}}}   \}= \sqrt{2}\iota_{\pa\bX}^*(\mu_X)
$$
and 
$$
i_q(\mu_M)= \sqrt{2} \{ \lrp{\frac{\nu_{K'+1}}{\sqrt{2}},\frac{\nu_{K'+1}}{\sqrt{2}}},\ldots,  \lrp{\frac{\nu_K}{\sqrt{2}},\frac{\nu_K}{\sqrt{2}}} \}.
$$
In particular, in terms of these bases, we see that
$$
   |\det(j_q)_{\perp}|= |\det(i_q)_{\perp}|.
$$
Using the formula of Milnor applied to the long exact sequence \eqref{se.5} therefore yields the following.
\begin{theorem}
If $r_1>0$ or $r_1=0$ with $\bn_1=\cdots=\bn_{r_2}=0$ and $n\ne 0$, then
$$
    \tau(M;\hE_{m,n},\mu_M)= \frac{\tau(X;E_{m,n},\mu_X)^2}{\tau(\pa\bX;E_{m,n},\mu_{\pa\bX})}= \tau(X;E_{m,n},\mu_X)^2.
$$
\label{se.11}\end{theorem}
\begin{proof}
The result would follow from the discussion above if instead of $\mu_{\pa\bX}$, we would have chosen the basis 
$$
    \mu_{\pa\bX}'=(\mu_{+},\iota_{\pa\bX}^*\mu_X).
$$
However, since the change of basis from $\mu_{\pa\bX}$ to $\mu_{\pa\bX}'$ is not orthonormal, but has determinant $1$, we can simply replace $\mu_{\pa\bX}'$ by $\mu_{\pa\bX}$ to obtain the result.  Since we assume that the basis $\mu_{\pa\bX}$ is self-dual, we also know from \cite[Corollary~1.9]{Muller1993} that
$$
\tau(\pa\bX;E_{m,n},\mu_{\pa\bX})=1.
$$
\end{proof}

To relate analytic torsion and Reidemeister torsion on $X$ when $\dim X$ is odd, that is, when $r_2$ is odd,  the idea is to start with the formula of \cite{Muller1993}
\begin{equation}
\log\tau(M;\hE_{m,n},\mu_M)=\log T(M,\hE_{m,n},g_{\fc,\epsilon},h_{\epsilon})-\log\lrp{\prod_q [\mu^q_M|\omega_{\epsilon}^q]^{(-1)^q}}
\label{se.14}\end{equation}
for $\epsilon>0$,
where $\omega_{\epsilon}^q$ is an orthonormal basis of harmonic forms in degree $q$ with respect to $g_{\fc,\epsilon}$ and $h_{\epsilon}$, while $[\mu^q_M|\omega_{\epsilon}^q]= |\det W^q|$ with $W^q$ the matrix describing the change of bases
$$
     (\mu^q_M)_i= \sum_jW^q_{ij}(\omega_{\epsilon}^q)_j.  
$$
Taking the finite part at $\epsilon=0$ of the formula \eqref{se.14} shall lead to a formula on $X$.  First, assuming that $r_1+r_2-1>1$, we know by Theorem~\ref{cdat.13} that
\begin{equation}
\FP_{\epsilon=0} \log T(M,\hE_{m,n},g_{\fc,\epsilon},h_{\epsilon})= 2\log T(X,E_{m,n},g_{\fc},h).
\label{se.15}\end{equation}
By Theorem~\ref{se.11}, we can express $\tau(M,E_{m,n},\mu_M)$ in terms of the torsion on $\bX$.  So, what remains to be done is to take the finite part of 
$$
  \log\lrp{\prod_q [\mu^q_M|\omega_{\epsilon}^q]^{(-1)^q}}
$$
as $\epsilon\searrow 0$.  

By the definition of $c_b$ in \cite[(4.16)]{MR1} and our definition of $\mu_{\pa\bX}$ in \eqref{se.7}, we see that an orthonormal basis of $\bigoplus_{\eta\in\mathfrak{P}_{\Gamma}}\ker_{L^2_b} D_{b,\eta}$ is given by
\begin{equation}
\lrp{  \frac{1}{\sqrt{c_{\frac{d_{\bbK}}{2}-W}}} \langle X\rangle^{W-\frac{d_{\bbK}}2} \rho^W \mu_{+} \frac{dX}{\langle X \rangle }, \frac{1}{\sqrt{c_{W-\frac{d_{\bbK}}2}}} \langle X\rangle^{\frac{d_{\bbK}}2 -W} \rho^W \mu_{-} }.
\label{se.16}\end{equation}
As in \cite{MR1}, this can be extended to a polyhomogeneous basis of harmonic forms on the single surgery space $(X_s,\hE_{m,n})$ of \eqref{sss.1} with respect to $D_{\fc,\epsilon}:= \rho^{\frac{d_{\bbK}}2}\eth_{\fc,\epsilon}\rho^{-\frac{d_{\bbK}}2}$.  Let $(\omega_{\epsilon,(2)},\omega_{\epsilon,+},\omega_{\epsilon,-})$ be a corresponding orthonormal basis of harmonic forms with respect to $\eth_{\fc,\epsilon}$, where $\omega_{\epsilon,(2)}$ corresponds in the limit $\epsilon\searrow 0$ to an orthonormal basis of $\ker_{L^2_b} D_{\fc,0}$ with $D_{\fc,0}= \rho^{\frac{d_{\bbK}}2}\eth_{\fc,0}\rho^{-\frac{d_{\bbK}}2}$,  $\omega_{\epsilon,+}$ corresponds to 
$$
\frac{1}{\sqrt{c_{\frac{d_{\bbK}}{2}-W}}} \langle X\rangle^{W-\frac{d_{\bbK}}2} \rho^W \mu_{+} \frac{dX}{\langle X \rangle }
$$
and $\omega_{\epsilon,-}$ corresponds to
$$
 \frac{1}{\sqrt{c_{W-\frac{d_{\bbK}}2}}} \langle X\rangle^{\frac{d_{\bbK}}2 -W} \rho^W \mu_{-}.
$$

  If $\epsilon^{\mu}w$ for $\mu>0$ is a higher order term in the expansion of $\omega_{\epsilon}\in \omega_{\epsilon,+}$ as $\epsilon\searrow 0$,  then since $D_{\fc,\epsilon}$ commutes with multiplication by $\epsilon$, we see that $\rho^{\frac{d_{\bbK}}2}w\in \ker D_{\fc,\epsilon}$, so in particular its restriction $\rho^{\frac{d_{\bbK}}2}w_{sb}$ to $\bhs{sb}$, the front face of the single surgery space $X_s$ considered in \cite{ARS1,MR1}, is in $\bigoplus_{\eta\in\mathfrak{P}_{\Gamma}}D_{b,\eta}$.  As in \cite{MR1}, we can check that cohomologically, these terms are negligible, so that  in the limit $\epsilon\searrow 0$, 
\begin{equation}
\begin{aligned}
\omega_{\epsilon,+} &\sim \frac{1}{\sqrt{c_{\frac{d_{\bbK}}2-W}}} (\langle X\rangle\rho)^{W-\frac{d_{\bbK}}2}\mu_{+}\frac{dx}{\rho} \\
 &= \epsilon^{W-\frac{d_{\bbK}}2}\langle X\rangle^{2W-d_{\bbK}} \frac{1}{\sqrt{c_{\frac{d_{\bbK}}2-W}}}\mu_{+} \frac{dX}{\langle X\rangle }.
\end{aligned}
\label{se.17}\end{equation}
Since
$$
     \int_{-\infty}^{\infty} \frac{\epsilon^{W-\frac{d_{\bbK}}2}}{\sqrt{c_{\frac{d_{\bbK}}{2}-W}}} \langle X\rangle^{2W-d_{\bbK}} \frac{dX}{\langle X\rangle}= \epsilon^{W-\frac{d_{\bbK}}2}\sqrt{c_{\frac{d_{\bbK}}{2}-W}}
$$
when acting on the negative eigenspaces of $W-\frac{d_{\bbK}}2$, 
we see that cohomologically,
\begin{equation}
[\omega_{\epsilon,+}]\sim \epsilon^{W-\frac{d_{\bbK}}2}\sqrt{c_{\frac{d_{\bbK}}{2}-W}} \pa_*[\mu_{+}]
\label{se.18}\end{equation}
as $\epsilon\searrow 0$.  On the other hand, if $\epsilon^{\mu}w$ for some $\mu>0$ is a higher order term in the expansion of $\omega_{\epsilon}\in \omega_{\epsilon,-}$, then again $\rho^{\frac{d_{\bbK}}2}w\in \ker D_{\fc,\epsilon}$, so its restriction $\rho^{\frac{d_{\bbK}}2}w_{sb}$ to $\mathfrak{B}_{sb}$ is in $\bigoplus_{\eta\in\mathfrak{P}_{\Gamma}}D_{b,\eta}$.  As in \cite{MR1}, though such a term is negligible in $L^2$-norm, it could still give a contribution in cohomology.  More precisely,  if 
$$
   \rho^{\frac{d_{\bbK}}2}w_{sb}= \langle X\rangle^{\frac{d_{\bbK}}2-W}\rho^W \psi
$$
for some $\psi \in \spane \langle \mu_{\pa\bX,m,n}\rangle$, then cohomologically
$$
\begin{aligned}
  [\epsilon^{\mu}\iota_Y^*w] &\sim \epsilon^{\mu}\iota_Y^*\left[ \lrp{\frac{\langle X\rangle}{\rho}}^{\frac{d_{\bbK}}2-W} \psi \right] \\
  &= \epsilon^{\mu+W-\frac{d_{\bbK}}2} [\psi].
\end{aligned}  
$$
As in \cite[(6.18) and (6.19)]{MR1} and keeping in mind \cite[Remark~6.1]{MR1}, we expect in general a non-trivial contribution in cohomology from some lower order terms and we can deduce that 
\begin{equation}
 [\iota_{\pa\bX}^*\omega_{\epsilon,-}]\sim \frac{\epsilon^{\hat{W}-\frac{d_{\bbK}}2}}{\sqrt{c_{\hat{W}-\frac{d_{\bbK}}2}}} [\iota_{\pa\bX}^*\mu_X] \quad \mbox{as} \; \epsilon\searrow 0,
\label{se.20b}\end{equation}
where the weight operator $\hat{W}$ is defined by
$$
    \hat{W}\Xi= \nu \Xi
$$
if $\Xi \in \iota_{\pa\bX}^*\mu_X$ is such that 
$$
\Xi= \hat{\Xi} \mod \spane \langle \mu_{+}\rangle
$$ 
with $\hat{\Xi}\in \mu_{\pa\bX,0,0}$ such that $W\hat{\Xi}=\nu \hat{\Xi}$.

Combining \eqref{se.18} and \eqref{se.20b} yields the following.
\begin{lemma}
When $r_1+r_2-1>0$,
\begin{equation}
\FP_{\epsilon=0} \log\lrp{\prod_q [\mu^q_M|\omega^q]^{(-1)^q}}=0
\label{se.20}\end{equation}
in the limit $\epsilon\searrow 0$.
\label{se.20c}\end{lemma}
\begin{proof}
From \eqref{se.18} and \eqref{se.20b}, we see that
$$
  \FP_{\epsilon=0} \log\lrp{\prod_q [\mu^q_M|\omega^q]^{(-1)^q}}= \sum_{\eta\in\mathfrak{P}_{\Gamma}}B_{\eta}\chi(S_{\eta}),
$$
where $\chi\lrp{S_{\eta}  }$ is the Euler characteristic of $S_{\eta}$ and $B_{\eta}$ is a certain sum of terms.  Our assumption that $\dim S_{\eta}=r_1+r_2-1>0$ ensures that $\chi\lrp{S_{\eta}  }=0$, from which the result follows. 
\end{proof}

This yields our main result.
\begin{theorem} Suppose that $r_1>0$ or that $r_1=0$ with $\bn_1=\cdots=\bn_{r_2}=0$ and $n\ne 0$.
If $r_2$ is odd (\ie $\dim X$ is odd) and $r_1+r_2-1>1$,  then 
$$
 \log T(X,E_{m,n},g_{\fc},h)= \log\tau(\bX,E_{m,n},\mu_X).
$$
\label{se.21}\end{theorem}
\begin{proof}
Theorem~\ref{cdat.13}, Theorem~\ref{se.11} and Lemma~\ref{se.20c} allow us to obtain the result by taking the finite part at $\epsilon=0$ of \eqref{se.14}.
\end{proof}

In many cases, the long exact sequence \eqref{se.5} is trivial, in which case this result simplifies as follows.

\begin{theorem}
Suppose that $m,n$ are chosen in Proposition~\ref{dec.8d} with $n_j\ne \bn_j$ for some $j$ and  in such a way that $D_{b,\eta}$ is Fredholm and has trivial $L^2$-kernel for each $\eta\in \mathfrak{P}_{\Gamma}$.  If $\dim X$ is odd and $r_1+r_2-1>1$, then
$$
     T(X,E_{m,n},g_{\fc},h)= \tau(\bX,E_{m,n}),
$$
where $T(X,E_{m,n}, g_{\fc},h)$ is the analytic torsion of $(X,E_{m,n},g_{\fc},h)$ and   $\tau(\bX,E_{m,n})$ is the Reidemeister torsion of $(\bX,E_{m,n})$.
If instead $\dim X$  is even with $r_1+r_2-1>1$ or $(r_1,r_2)\in\lrp{(2,0),(0,2)}$, then
$$
 T(X,E_{m,n},g_{\fc},h)= 1.
$$
\label{cm.1}\end{theorem}
\begin{proof}
Since we assume that $n_j\ne \bn_j$ for some $j$, we know by Theorem~\ref{cdat.1} that $H^*_{(2)}(X;E_{m,n})=\{0\}$.  Since $\ker_{L^2_b}D_{b,\eta}=\{0\}$ for each $\eta\in \mathfrak{P}_{\Gamma}$, this implies by Lemma~\ref{se.6} and \eqref{deco.3b} that $H^*(M;\hE_{m,n})=\{0\}$ and $H^*(\pa\bX;E_{m,n})=\{0\}$. We deduce from the long exact sequence \eqref{se.5} that $H^*(\bX;E_{m,n})=\{0\}$ as well, hence that this long exact sequence is trivial.  

Thus, if $\dim X$ is odd, Theorem~\ref{cm.1} is just Theorem~\ref{se.21} when the long exact sequence \eqref{se.5} is trivial.
If $\dim X$ is even, the result follows instead from Theorem~\ref{cdat.13} due to the fact that $\hE_{m,n}$ is self-dual as a flat vector bundle and as a Hermitian vector bundle, so by \cite[Proposition~2.9]{Muller1993},
$$
T(M,\hE_{m,n},g_{\fc,\epsilon},h_{\epsilon})= 1 \quad \forall \epsilon\ge 0.
$$

\end{proof}

\section{Construction of acyclic bundles}\label{sect-acyclic}

Following \cite[sect. 8.1]{BV}, we explain in some detail how 
acyclic $\Gamma$-modules are constructed in our case. 
To construct the corresponding representations we proceed as follows.
Let $G_0=\SL(2)/\bbK$ and let $T_0\subset G_0$ be the standard maximal
torus.  Let $T:=\Restr_{F/\bbQ}(T_0)$ be the corresponding maximal torus of $G$.
Let $\bbG_m$ be the multiplicative group. We select an isomorphism
$\bbG_m/F\cong T_0/F$, given by
\[
a\mapsto\begin{pmatrix}a&0\\0& a^{-1}\end{pmatrix}
\]
for $a\in\bbG_m(F)$. This gives rise to an identification
\[
\Restr_{F/\bbQ}(\bbG_m)=\Restr_{F/\bbQ}(T_0)=T.
\]
Let $E/\bbQ$ be a Galois extension splitting $T$ such that $\Hom(F,E)\neq 0$.
Such a Galois extension always exists \cite[Prop. 1.5]{Borel-Tits}.
We denote by $\{\sigma\colon F\to E\}$ the set of embeddings of $F$ into $E$
on which the Galois group $\Gal(E/\bbQ)$ acts transitively. Note that
\[
\#\{\sigma\colon F\to E\}=[F\colon\bbQ].
\]
Let $G\times_\bbQ E$ and $T\times_\bbQ E$ be the groups obtained from $G$ and
$T$, respectively, by extension of scalars \cite[I, 4c]{Milne2011}. We have
\begin{equation}\label{prod1}
G\times_\bbQ E=\prod_{\sigma\colon F\to E}G_0\times_\sigma E=
\prod_{\sigma\colon F\to E}\SL(2)/F\times_\sigma E.
\end{equation}
Hence an irreducible representation $\rho$ of $G\times _\bbQ E$ is a
tensor product
\[
\rho=\bigotimes_{\sigma\colon F\to E}\rho_\sigma,
\]
where $\rho_\sigma$ is the irreducible representation of $\SL(2)/E$ on the
$E$-vector space $W_\sigma:=\Sym^{d_\sigma}(E^2)$ of homogeneous polynomials of
degree $d_\sigma$ in two variables. Thus we have an $E$-rational representation
\begin{equation}\label{rep1}
\rho\colon G\times_\bbQ E\to\GL(W),
\end{equation}
where
\begin{equation}\label{rep-space}
W=\bigotimes_{\sigma\colon F\to E}W_\sigma=\bigotimes_{\sigma\colon F\to E}
\Sym^{d_\sigma}(E^2).
\end{equation}
The base change $T\times_\bbQ E$ is a split torus, i.e., we have
\begin{equation}\label{prod2}
T\times_\bbQ E=\prod_{\sigma\colon F\to E} T_0\times_{F,\sigma}E=
\prod_{\sigma\colon F\to E} T_0.
\end{equation}
Let $X^\ast(T\times_\bbQ E)=\Hom(T\times_\bbQ E,\bbG_m)$ be the group of
characters of $T\times_\bbQ E$ and let $X^\ast_+(T\times_\bbQ E)$ be the dominant
characters. By \eqref{prod2} we have
\begin{equation}\label{weights1}
X^\ast(T\times_\bbQ E)=\bigoplus_{\sigma\colon F\to E} X^\ast(T_0\times_{F,\sigma} E)
=\bigoplus_{\sigma\colon F\to E} X^\ast(T_0).
\end{equation}
{\bf Remark.}
If $E^\prime$ is another Galois extension of $F$ with injection
$\iota\colon E\to E^\prime$, then $\iota$ induces an isomorphism
$X^\ast(T\times E)\to X^\ast(T\times E^\prime)$ by $\lambda\mapsto\lambda_\iota:=
\iota\circ\lambda$.

\bigskip
By \eqref{weights1}  the highest weight $\lambda_\rho\in X^\ast_+(T\times_\bbQ E)$ of $\rho$ is given by 
\begin{equation}\label{weights2a}
\lambda_\rho=(d_{\sigma_1},...,d_{\sigma_n}).
\end{equation}
and to every highest weight as above there is associated a unique irreducible
finite dimensional rational representation of $G\times_{\bbQ} E$ defined over $E$.

Now observe that the functor ``restriction of scalars''  is right adjoint to the
functor ``extension of scalars'' \cite[I, 4d]{Milne2011}. Thus  we have
\begin{equation}\label{rad}
\Hom_{\bbQ}(G,\Restr_{E/\bbQ}(\GL(W)))\cong\Hom_E(G\times_\bbQ E,\GL(W)).
\end{equation}
Hence the representation \eqref{rep1} corresponds to a representation
\begin{equation}\label{rep}
\tilde\rho\colon G\to \Restr_{E/\bbQ}(\GL(W))
\end{equation}
which is defined over $\bbQ$. Since $\Restr_{E/\bbQ}$ is a functor,
there is a canonical homomorphism
\begin{equation}\label{restr2}
\Restr_{E/\bbQ}(\GL(W))\to\GL(\Restr_{E/\bbQ}(W)).
\end{equation}
Let $V=\Restr_{E/\bbQ}(W)$ which is just $W$ considered as $\bbQ$-vector space.
Combining \eqref{rep} and \eqref{restr2}, we obtain a representation
\begin{equation}\label{rep2}
\varrho\colon G\to \GL(V),
\end{equation}
which is defined over $\bbQ$. Now recall that by \eqref{rep-space}, $W$ is the
tensor product of the $E$-vector spaces $\Sym^{d_\sigma}(E^2)$, $\sigma\colon F\to
E$, and $\Sym^{d_\sigma}(E^2)$ is the space of homogeneous polynomials of degree
$d_\sigma$ in two variables with coefficients in $E$. Choose an integral basis
$e_1,...,e_q$ of $E$ over $\bbQ$, where $q=[E\colon \bbQ]$. Expressing the
coefficients of the polynomials in this basis, we obtain a vector space over
$\bbQ$. So
\begin{equation}\label{repr-Q}
W_\sigma=\Res_{E/\bbQ}\Sym^{d_\sigma}(E^2)=\bigoplus_{i=1}^q W_i,
\end{equation}
where $W_i$ consists of homogeneous polynomials
\[
\sum_{k=1}^{d_\sigma} a_{ki} e_iX^kY^{d_\sigma-k},\quad a_{ki}\in\bbQ.
\]
Correspondingly, the tensor product $W$ becomes a vector space over $\bbQ$,
which is denoted by $V$.

Next we determine how $V\otimes_\bbQ\bbC$ decomposes into irreducible
representations of $G(\bbC)$. Let $S_\infty$ denote the set of Archimedean
places of $F$. For $v\in S_\infty$ let $F_v$ be the completion of $F$ with
respect to $v$. Let
$G_v/\bbR:= \Res_{F_v/\bbR}(\SL(2)/F_v)$. 
We have
\begin{equation}
G\times_\bbQ\bbR=\prod_{v\in S_\infty} G_v.
\end{equation}
Fix an embedding $\iota\colon E\to\bbC$. Then
\[
W\otimes_{E,\iota}\bbC=\bigotimes_{v\in S_\infty} W_v.
\]
If $v$ is a real place, it corresponds to an embedding $\sigma\colon F\to \bbR$
and if $v$ is complex then it corresponds to a pair of conjugate embeddings
\[
\sigma,\bar\sigma\colon F\to \bbC,
\]
which are viewed as the two continuous isomorphisms $\sigma,\bar\sigma\colon
F_v\cong \bbC$. In the first case, we have $W_v=W_\sigma\otimes_E\bbC$, and in the
second case
\[
W_v\cong (W_\sigma\otimes_E\bbC)\otimes(W_{\bar\sigma}\otimes_E\bbC).
\]
Let $\sigma_1,...,\sigma_{r_1}$ denote the real embeddings and
$(\nu_1,\bar\nu_1),...,(\nu_{r_2},\bar\nu_{r_2})$ the complex embeddings of
$F$ in $\bbC$. Then the highest weight $\lambda_\rho$ of $\rho$ takes the form
\begin{equation}\label{weight4}
\lambda_\rho=(d_{\sigma_1},...,d_{\sigma_{r_1}},(d_{\nu_1},d_{\bar\nu_1}),...,
(d_{\nu_{r_2}},d_{\bar\nu_{r_2}})).
\end{equation}
An embedding $\tau\colon E\to \bbC$ induces an isomorphism $X^\ast(T\times E)
\cong X^\ast(T\times \bbC)$. 
From \eqref{repr-Q} follows that $V\otimes_\bbQ\bbC$ is the direct sum of
irreducible representations whose highest weights are obtained from
$\lambda_\rho$ by applying the various embeddings $E\hookrightarrow\bbC$.
Now assume that $r_2\ge 1$ and the highest weight
$\lambda_\rho=(d_{\sigma_1},...,d_{\sigma_n})$
of the representation $\rho$ of $G\times_\bbQ E$ satisfies
\begin{equation}\label{weight6}
d_{\sigma_i}\neq d_{\sigma_j},\quad i\neq j.
\end{equation}
Then it follows from the considerations above that $V\otimes_\bbQ\bbC$
decomposes in the direct sum of irreducible, non-selfconjugate representations with respect to the standard Cartan involution $\vartheta$ of $G_{\infty}$ with respect to $K_{\infty}$.

We summarize the properties of the representation $\varrho\colon G\to\GL(V)$
which is defined as \eqref{rep2}.

\begin{proposition}\label{prop-repr1}
Assume that the highest weight of $\rho$ satisfies \eqref{weight6}.
Then  $\varrho$ has the following properties:
\begin{enumerate}
\item Let $\varrho_{m,n}$ be defined by \eqref{stand-repr}. The irreducible
constituents of
$\varrho_\infty\colon G_\infty\to\GL(V\otimes_\bbQ \bbC)$ are of the form
$\varrho_{m,n}$, where $m_i\neq m_j$ for $i\neq j$ and $n_j\neq \bar n_j$,
$j=1,...,r_2$.
\item Let $\Gamma\subset G(\bbQ)$ be a congruence subgroup. There exists a
$\Gamma$-stable lattice $\Lambda\subset V$.  
\item Let $\Gamma\subset G(\bbQ)$ be a torsion free congruence subgroup.
Let $X=\Gamma\bsl \widetilde X$ and let $\cE_\varrho\to X$ be the flat vector
bundle associated to $\varrho_\infty|_\Gamma$. Assume that $r_2>0$. Then
$H^\ast_{(2)}(X,\cE_\varrho)=0$.  Moreover, if $r_1>0$ as well, then in fact $H^{\ast}(X,\cE_{\varrho})=0$.
\end{enumerate}
\end{proposition}
\begin{proof}
The irreducible representations of $G_\infty$ are of the form $\varrho_{m,n}$.
It follows from \eqref{repr-Q} that the highest weights of the irreducible
constituents are permutations of \eqref{weight4}. Then (1) follows from
\eqref{weight6}.
By \cite[Chapter VII, Prop 5.1, p.400]{Milne2011}  there exists a lattice
$\Lambda\subset V$, which is stable under $\Gamma$. This proves (2).
 
Now assume that $r_2>0$. By (1) every irreducible constituent of $\varrho_\infty$
is of the form $\varrho_{m,n}$ with $n_j\neq\bar n_j$, $j=1,...,r_2$ and $m_i\ne m_k$ for $i,k\in\{1,\dots,r_1\}$ distinct. Hence
$\varrho_{m,n}$ is not self-conjugate. Then by \eqref{vanish-cohom} it follows
that $H^\ast_{(2)}(X;E_{m,n})=0$, while $H^{\ast}(X;E_{m,n})=0$ if $r_1>0$ by \eqref{harder.2} and Lemma~\ref{lem-bound-cohom}. This implies (3).
\end{proof}

\section{Exponential growth of torsion in cohomology}\label{eg.0}

To prove Theorem \ref{main-thm} we follow the approach of \cite{BV}.
Recall that in \cite{BV} only the case of co-compact arithmetic groups is
considered. The method is
based on the approximation of the $L^2$-torsion of $X=\Gamma\bsl\widetilde X$
by the renormalized logarithm
of the analytic torsion of the compact manifolds $X_i$ as $i\to\infty$.
Since in our case the manifolds are not compact, we use a regularized 
version of the analytic torsion, which for a general reductive group $G$
has been defined in \cite{Matz-Muller}. The approximation of the
$L^2$-analytic torsion has been studied in \cite{Matz-Muller}. We recall the
definition in our case. To define the regularized analytic torsion, one has
to define the regularized trace of the heat operator. The definition of the
regularized trace of the heat operator, given in \cite{Matz-Muller}, uses the
adelic framework. Let $\bbA$ be the ring of adeles, $\bbA_f$ the ring of finite
adeles, $G(\bbA)$ the group of adelic points, and $K_f\subset G(\bbA_f)$ an
open compact subgroup. Then the adelic space is defined as
$X(K_f)=G(\bbQ)\bsl G(\bbA)/K_\infty^0\cdot K_f$. Now observe that $G$ is
simply connected and $\bbQ$-simple. Therefore $G$ satisfies the strong
approximation property with respect to the infinite place and we have
\[
X(K_f)\cong \Gamma_{K_f}\bsl \widetilde X,
\]
where $\Gamma_{K_f}:=G(\bbR)\cap K_f$. In \cite{Matz-Muller} the truncation of
$X(K_f)$ is used to define a regularized trace of the heat
operators. For the truncation we need a height function. Recall that $X(K_f)$ is
the union of a compact manifold with a finite number of fibered cusps ends
\eqref{mc.2}. On the fibered cusp ends, one can use the radial variable $r$ of
the coordinates \eqref{new-coord} to define a height function. A height
function is not unique. For our purpose we chose the height function $h$ as in
\cite[p. 46]{Har}.  Sufficiently far in a given fibered cusp end, this corresponds to the square of the function $\nu$ of \eqref{bdf.1}.  On the other hand, the function $r$ in \eqref{new-coord} corresponds to $\nu^{\frac{2}{d_{\bbK}}}$, so for $r$ sufficiently large in a given fibered cusp end, we have the identification $r=h^{\frac{1}{d_{\bbK}}}$, that is, $h=x^{-d_{\bbK}}$ in terms of the boundary defining function $x=\frac{1}{r}$ introduced at the beginning of \S~\ref{hdR.0}. For $T\gg0$ let $X(T)=\{x\in X\colon h(x)\le T\}$. Let
$\tau$ be an irreducible finite dimensional representation of $G_\infty$. Let
$K_{p,\tau}(t,x,y)$ be the kernel of the heat operator $e^{-t\Delta_p(\tau)}$
of the Laplacian on $p$-forms with values in the flat vector bundle $E_\tau$
associated to $\tau$. Then one can show that there exist functions $a(t), b(t)$ of $t>0$ such that
\begin{equation}\label{asympt-exp}
\int_{X(T)}\tr K_{p,\tau}(t,x,x) dx=a(t)\log(T) + b(t)+O(T^{-1})
\end{equation}
as $T\to\infty$.  In general this follows from the work of Arthur related to the
trace formula. In the present case, however, this can be worked out explicitly,
using the spectral expansion of the kernel $K_{p,\tau}(t,x,y)$ as in the case of
hyperbolic manifolds of finite volume \cite[$\S$ 5]{MP2012}. This regularization of
the heat kernel is the Hadamard regularization, which was introduced by Melrose
in the case of manifolds with cylindrical ends. 
In the present paper we use the regularization considered in
\cite[Sect. 9]{ARS1}.
This is the Riesz regularization of the heat kernel. The relation between the
two is discussed in \cite{Alb}. By \cite[(7.1)]{ARS1}, seen as a $b$-density, the expansion of the pointwise trace of 
heat kernel at infinity for fixed $t$ only has nonnegative integer powers of the boundary
defining function $x$, so in particular the expansion contains no logarithmic terms. By \cite[p. 146]{Alb}, for a fixed choice of boundary defining function, Hadamard and Riesz regularizations agree.  Similarly, replacing $x$ by a positive power $u=x^{\lambda}$ yields the same Riesz or Hadamard regularization, e.g.
\begin{equation}
 \sideset{^H}{}{\int^{\delta}_0}\frac{dx}{x}=  \sideset{^H}{} {\int_0^{\delta^{\lambda}}}\frac{du}{\lambda u}= \log \delta =  \sideset{^R}{}{\int_0^{\delta^{\lambda}}}\frac{du}{\lambda u}=  \sideset{^R}{}{\int_0^{\delta}} \frac{dx}{x}.
\label{hr.1}\end{equation}
 This implies     that the the regularization used here and the one used in \cite{Matz-Muller} give exactly the same
regularized trace.
\begin{remark}
We note that in \cite{MR2} we also used two different regularizations of the
trace of the heat operators. On the one hand, we use the main result of
\cite{MR1}, relating analytic torsion and Reidemeister torsion. In this case,
we use the Riesz regularization of the trace of the heat operator as defined in
\cite{ARS1}. On the other hand, we use \cite[Theorem 1.1]{MP2014}, which is
equivalent to \eqref{appr-l2-tor}. This result is
based on the Hadamard regularization of the trace of the heat operator.
As above, it can be shown that the two methods lead to the same regularized
trace. 
\end{remark}
Let $X_i=\Gamma(\fn_i)\bsl\widetilde X$, $i\in\bbN$, be the sequence of
manifolds associated to a sequence $\{\fn_i\}_{i\in\bbN}$ of ideals in
${\mathcal O}_F$ satisfying \eqref{ideals}.  Then we have
\begin{proposition}\label{prop-appr-l2-tor}
Let $\tau\in\Rep(G_\infty)$ be irreducible and assume that
$\tau\not\cong\tau\circ\vartheta$. Let $E_\tau\to X_i$ be the flat vector
bundle over $X_i$ which is associated to $\tau|_{\Gamma(\fn_i)}$. Then
\begin{equation}\label{appr-l2-tor}
\lim_{i\to\infty}\frac{\log T(X_i,E_\tau,g_{\fc},h)}
{\vol(X_i)}= t^{(2)}_{\widetilde X}(\tau),
\end{equation}
where 
$$
 t^{(2)}_{\widetilde X}(\tau)= \frac12\sum_{q\ge 0}(-1)^q q\lrp{ \left.\frac{d}{ds}\right|_{s=0} \frac{1}{\Gamma(s)}\int_0^{\infty}t^{s-1}\Tr (e^{-t\Delta^{(2)}_q}(x,x))dt}
$$
is the $L^2$-analytic torsion of $(\widetilde{X},\widetilde{g},\widetilde{E}_\tau,\widetilde{h})$ with $\widetilde{g}$, $\widetilde{E}_\tau$ and $\widetilde{h}$ the lifts of $g_{\fc}$, $E_\tau$ and $h$ to $\widetilde{X}$ and with $e^{-t\Delta^{(2)}_q}(x,x)$ the heat kernel of the Hodge Laplacian in degree $q$ of $(\widetilde{X},\widetilde{g},\widetilde{E}_\tau,\widetilde{h})$ evaluated at some arbitrary point $(x,x)$ on the diagonal of $\widetilde{X}\times \widetilde{X}$.
\end{proposition}
\begin{proof}
Assume that there exists a finite set $S$ of primes such that $p\nmid N(\fn_i)$
for all $p\not\in S$ and $i\in\bbN$. Then \eqref{appr-l2-tor} follows from
\cite[Theorem 1.5]{Matz-Muller}. However, in our case one can eliminate the
additional assumption on the $\fn_i$'s as follows.
The analogous result for $G=\SL(n)/\bbQ$ and principal congruence subgroups of
$\SL(n,\bbZ)$ was proved in \cite[Theorem 1.1]{MzM2}. The proof can be
extended to $G=\SL(n)/F$ for any number field $F$. The reason is the following.
The proof of \eqref{appr-l2-tor} involves the use of the Arthur trace formula,
in particular the unipotent part of the geometric side of the trace formula.
Arthur's fine expansion of the unipotent contribution \cite[(8.1)]{Matz-Muller}
involves global coefficients, for which appropriate bounds are needed.
The existence of such bounds is not known in general. However, for $\GL(n)/F$
J. Matz \cite{Mat} has obtained suitable bounds for these coefficients.
Using these bounds one can proceed as in \cite{MzM2} and prove
\eqref{appr-l2-tor} for $G=\SL(n)/F$, which implies the corresponding result
for $G=\Res_{F/\bbQ}(\SL(n)/F)$.
\end{proof}

Following \cite{BV}, we can use the identification of analytic torsion with Reidemeister torsion to prove Theorem~\ref{int.4} about the exponential growth of torsion in cohomology for the sequence $\{\Gamma(\fn_i)\}_{i\in\bbN}$ of
principal congruence subgroups. Let 
\begin{equation}
\varrho\colon G\to \GL(V_\varrho)
\label{qrr.1}\end{equation}
be a $\bbQ$-rational representation over a finite dimensional $\bbQ$-vector space $V_{\varrho}$ and let $\cE_\varrho\to X$ be the
flat vector bundle associated to
$\varrho_\infty\colon G_\infty\to\GL(V_\varrho\otimes_\bbQ\bbC)$.  We will suppose that $\varrho_{\infty}$ decomposes into a sum of irreducible representations which are not self-conjugate, so that Proposition~\ref{prop-appr-l2-tor} implies that 
\begin{equation}\label{appr-l2-tor-1}
\lim_{i\to\infty}\frac{\log T(X_i,\cE_\varrho,g_{\fc},h)}{\vol(X_i)}=
t^{(2)}_{\widetilde X}(\varrho_\infty).
\end{equation}

 By  \eqref{disc-spec1}, $\cE_\varrho$ has trivial $L^2$-cohomology. Let us first focus on the case where the flat bundle
$\cE_\varrho$ is acyclic.  Assuming $r_1>1$ and $r_2>0$, Proposition~\ref{prop-repr1} provides many instances where this is the
case. 
Let $\Lambda_\varrho\subset V_\varrho$ be a $\SL(2,\cO_{\bbK})$-invariant lattice, which by
\cite[Chapter VII, Prop 5.1, p.400]{Milne2011} exists. Let $L_\varrho$ be the associated local
system of free $\bbZ$-modules over $X$. Then $H^*(X;L_\varrho)$ entirely
consists of torsion elements, namely
\[
      H^*(X;L_\varrho)= H^*_{\tor}(X;L_\varrho). 
\]
In this case, Theorem~\ref{int.4} asserts the following.
\begin{theorem}
Let $F$ be a number field with $r_2=1$ and $r_1>1$. Let \eqref{qrr.1} be a $\bbQ$-rational representation and assume that $\varrho_{\infty}$ decomposes into a sum of irreducible representations that are not self-conjugate.  Let $\Lambda_\varrho\subset V_\varrho$ be a $\SL(2,\cO_{\bbK})$-invariant lattice.  If for a sequence of congruence subgroups $\Gamma(\fn_i)$
with ideals $\fn_i$ satisfying \eqref{ideals}, $\Lambda$ is an acyclic $\Gamma(\fn_i)$-module for each $i$, then
$$
\liminf_{i\to \infty} \sum_{q+r_1 \;\; \operatorname{even}}
\frac{\log |H^q(\Gamma(\mathfrak{n}_i);\Lambda_{\varrho})|}{[\Gamma(\mathfrak{n}_1):\Gamma(\mathfrak{n}_i)]}\ge 2(-1)^{r_1+1}t^{(2)}_{\widetilde{X}}(\varrho_{\infty})\vol(X_1)>0,
$$
where $t^{(2)}_{\widetilde{X}}(\varrho_{\infty})$ is the $L^2$-torsion associated to
$\widetilde{X}$ and $\varrho_\infty$.
\label{eg.9}\end{theorem}

\begin{proof}
Let $\tau(\bX_j,\cE_\varrho)$ be the Reidemeister torsion of $X_j$ and
$\cE_\varrho$. By assumption $H^\ast(X_j;\cE_\varrho)=0$
for all $j\in\bbN$, so no choice of a basis of the cohomology is
needed to define the Reidemeister torsion. By Proposition~\ref{prop-appr-l2-tor}
and Theorem \ref{cm.1} we have
\begin{equation}\label{reidem-tors}
\lim_{j\to\infty}\frac{\log\tau(\bX_j,\cE_\varrho)}
{[\Gamma(\mathfrak{n}_1)\colon\Gamma(\mathfrak{n}_j)]}
=t^{(2)}_{\widetilde X}(\varrho_\infty)\vol(X_1).
\end{equation}
By \cite[Proposition~5.2]{BV}, $t^{(2)}_{\widetilde{X}}(\varrho_\infty)\ne 0$ if
and only if $\delta(\widetilde{X})=1$, where $\delta(\widetilde X)=\delta(G_{\infty})$ is the fundamental rank of $G_{\infty}$.  Since we assume that $r_2=1$, this is indeed the case by \cite[$\S$ 1.2]{BV}.  Moreover, still by \cite[Proposition~5.2]{BV}, we know that 
$$
    (-1)^{\frac{\dim\widetilde{X}-1}{2}}t^{(2)}_{\widetilde{X}}(\varrho_\infty)>0.
$$ 
Since $\dim\widetilde{X}=2r_1+3r_2=2r_1+3$, this means that 
\begin{equation}
  (-1)^{r_1+1}t^{(2)}_{\widetilde{X}}(\varrho_\infty)>0.
\label{eg.11}\end{equation}
On the other hand, by a result of Cheeger \cite[(1.4)]{Cheeger1979}, we know
that 
\begin{equation}
  \tau(\bX_i,\cE_\varrho)^2= \prod_q |H^q(\bX_i;L_\varrho)|^{(-1)^{q+1}}.
\label{eg.12}\end{equation}
By \eqref{reidem-tors} we thus conclude that 
$$
  \lim_{i\to \infty} \frac{ \sum_q (-1)^{q+1}\log |H^q(\bX_i;L_\varrho)|}{[\Gamma(\mathfrak{n}_1): \Gamma(\mathfrak{n}_i)]}= 2t^{(2)}_{\widetilde{X}}(\varrho_\infty)\vol(X_1).
$$
Using \eqref{eg.11}, this implies the result.

\end{proof}

If the flat vector bundle $\cE_\varrho\to X_i$ is not acyclic, but $L^2$-acyclic,
we can still obtain exponential growth of torsion following the approach of \cite{MR2}.  In this case, there are also non-trivial cohomology groups on the
boundary.  The decomposition \eqref{deco.4} induces in fact a decomposition
\begin{equation}
H^*(\pa\bX_i;\cE_\varrho)= H^*_+(\pa \bX_i;\cE_\varrho)\oplus
H^*_-(\pa \bX_i;\cE_\varrho).
\label{eg.12a}\end{equation}

\begin{proposition}
Assume that the number field $\bbK$ is such that $r_2=1$ and $r_1>0$.  Then 
$$
H^*_{\free,\pm}(\pa \bX_i;L_\varrho):=H^*_{\pm}(\pa \bX_i;\cE_\varrho)\cap
H^*_{\free}(\pa \bX_i;L_\varrho)
$$
induces a lattice in $H^*_{\pm}(\pa \bX_i;\cE_\varrho)$.
\label{eg.13}\end{proposition} 
\begin{proof}
Since we assume that $r_2=1$ and $r_1>0$, we see from Lemma~\ref{se.4} that the space $H^*_-(\pa \bX_i;\cE_\varrho)$ corresponds to the subspace of forms in $H^*(\pa \bX_i;\cE_\varrho)$ of vertical degree greater or equal to $r_1+r_2$ with respect to the fiber bundle \eqref{fb.1} induced on each boundary component of $\pa \bX_i$, while $H^*_+(\pa \bX_i;\cE_\varrho)$ corresponds to forms in $H^*(\pa \bX_i;\cE_\varrho)$ of vertical degree less than or equal to $r_2$.  By Leray's theorem \cite[Theorem~15.11]{Bott-Tu}, the vertical degree still makes sense on $H^*_{\free}(\pa \bX_i;L_\varrho)$, hence $H^*_{\free,-}(\pa \bX_i;L_\varrho)$ corresponds to cohomology classes of vertical degree greater or equal to $r_2+r_1$, while $H^*_{\free,+}(\pa \bX_i;L_\varrho)$ corresponds to cohomology classes of vertical degree at most $r_2$.  This gives the decomposition
$$
  H^*_{\free}(\pa \bX_i;L_\varrho)= H^*_{\free,+}(\pa \bX_i;L_\varrho)\oplus H^*_{\free,-}(\pa \bX_i;L_\varrho).
$$
Since $H^*_{\free}
(\pa \bX_i;L_\varrho)$ is a lattice in $H^*(\pa \bX_i;\cE_\varrho)$, this means that $H^*_{\free,\pm}
(\pa \bX_i;L_\varrho)$ is a lattice in $H^*_{\pm}(\pa \bX_i;\cE_\varrho)$.

\end{proof}

For $\eta_i\in \mathfrak{P}_{\Gamma(\mathfrak{n}_i)}$, let $[\eta_i]$ be the induced element in $\mathfrak{P}_{\Gamma(\mathfrak{n}_1)}$ and let $\pi_i: Y_{\eta_i}\to Y_{[\eta_i]}$ be the induced covering map, where 
$$
       Y_{\eta_i}= \Gamma(\mathfrak{n}_i)_{\eta_i}\setminus B_{\infty}(1)/B_{\infty}\cap K \quad \mbox{and} \quad  Y_{[\eta_i]}= \Gamma(\mathfrak{n}_1)_{[\eta_i]}\setminus B_{\infty}(1)/B_{\infty}\cap K.
$$
By the previous proposition, 
$$
    H^*_{\free,\pm}(Y_{\eta_i};L_\varrho):= H^*_{\free}(Y_{\eta_i};L_\varrho)\cap H^*_{\pm}(Y_{\eta_i};\cE_\varrho)
$$
is a lattice in $H^*_{\pm}(Y_{\eta_i};\cE_\varrho)$ and 
$$
    H^*_{\free,\pm}(Y_{[\eta_i]};L_\varrho):= H^*_{\free}(Y_{[\eta_i]};L_\varrho)\cap H^*_{\pm}(Y_{[\eta_i]};\cE_\varrho)
$$
is a lattice in $H^*_{\pm}(Y_{[\eta_i]};\cE_\varrho)$.  Using the metric on $Y_{\eta}$ induced by restriction of the metric \eqref{mc.4} at $r=1$, namely
$$
g_{S_{\eta}}+ \frac{1}{d_{\bbK}r^2}\left( \sum_{i=1}^{r_1}e^{-2\tu_i}dx_i^2+ 2 \sum_{j=1}^{r_2} e^{-2\tv_j}|dz_j|^2 \right),
$$
as well the restriction at $r=1$ of the bundle metric $h$, Hodge theory gives an identification of $H^*_{\pm}(Y_{\eta};\cE_\varrho)$ with the corresponding space of harmonic forms, inducing in particular on the vector space $H^*_{\pm}(Y_{\eta};\cE_\varrho)$ a natural inner product and a volume form.  With respect to this volume form, the covolume $\vol(H^q_{\free,\pm}(Y_{\eta_i};L_\varrho))$ of $H^q_{\free,\pm}(Y_{\eta_i};L_\varrho)$ in $H^q_{\pm}(Y_{\eta_i};\cE_\varrho)$ can be estimated in terms of the covolume $\vol(H^q_{\free,\pm}(Y_{[\eta_i]};L_\varrho))$ of $H^q_{\free,\pm}(Y_{[\eta_i]};L_\varrho)$ in  $H^{q}_{\pm}(Y_{[\eta_i]};\cE_\varrho)$.  More precisely, the following estimate will be useful.

\begin{proposition}\label{eg.14}
Suppose that the natural isomorphism $V^*\cong V$ induces an isomorphism $\Lambda^*\cong \Lambda$.  Then for each $i\in\bbN$, 
\begin{equation}\label{estim7}
\begin{split}
\frac{[\Gamma(\mathfrak{n}_1)_{[\eta_i]}:\Gamma(\mathfrak{n}_i)_{\eta_i}]^{-  \frac{b_{q,\pm}}2   }}{\vol(H^{2r_1+2-q}_{\free,\mp}(Y_{[\eta_i]};L_\varrho)) }
&\le \vol(H^q_{\free,\pm}(Y_{\eta_i};L_\varrho))\\
&\le \vol(H^q_{\free,\pm}(Y_{[\eta_i]};L_\varrho))[\Gamma(\mathfrak{n}_1)_{[\eta_i]}:\Gamma(\mathfrak{n}_i)_{\eta_i}]^{ \frac{b_{q,\pm}}2  },
\end{split}
\end{equation} 
where $b_{q,\pm}:= \dim_{\bbR} H^q_{\pm}(Y_{\eta_i};\cE_\varrho)$.
\end{proposition}
\begin{proof}
Since $\pi_i^*H^q_{\free,\pm}(Y_{[\eta_i]};L_\varrho)$ is a sublattice of $H^q_{\free,\pm}(Y_{\eta_i};L_\varrho)$, we see that
\begin{equation}
\begin{aligned}
\vol(H^q_{\free,\pm}(Y_{\eta_i};L_\varrho) \le \vol(\pi_i^*H^q_{\free,\pm}(Y_{[\eta_i]};L_\varrho))&= \lrp{\frac{\vol(Y_{\eta_i})}{\vol(Y_{[\eta_i]})}}^{\frac{b_{q,\pm}}2}\vol(H^q_{\free,\pm}(Y_{[\eta_i]};L_\varrho)) \\
  & = [\Gamma(\mathfrak{n}_1)_{[\eta_i]}:\Gamma(\mathfrak{n}_i)_{\eta_i}]^{ \frac{b_{q,\pm}}2} \vol(H^q_{\free,\pm}(Y_{[\eta_i]};L_\varrho)),
\end{aligned}
\label{eg.15}\end{equation}
giving the inequality on the right.  For the inequality on the left, recall that the representations $\varrho_{m,n}$ are self-dual, hence the flat vector bundle $\cE_\varrho$ is naturally self-dual both as a flat vector bundle and as a Hermitian vector bundle.  Using our assumption that $L_i^*\cong L_i$, Poincar\'e duality therefore induces an isomorphism 
$$
      H^{2r_1+2-q}_{\free,\mp}(Y_{\eta_i};L_\varrho)^*=H^q_{\free,\pm}(Y_{\eta_i};L_\varrho)
$$
as well as an isometry 
$$
    H^{2r_1+2-q}_{\mp}(Y_{\eta_i};\cE_\varrho)^*=H^q_{\pm}(Y_{\eta_i};\cE_\varrho),
$$
where we recall that $\dim Y_{\eta_i}= 2r_1+3r_2-1=2r_1+2$ since we assume that $r_2=1$.  Hence, we see that
$$
    \vol(H^{2r_1+2-q}_{\free,\mp}(Y_{\eta_i};L_\varrho)^* )= \vol (H^q_{\pm}(Y_{\eta_i};L_\varrho)).
$$
Since essentially by definition,
$$
 \vol(H^{2r_1+2-q}_{\free,\mp}(Y_{\eta_i};L_\varrho)^* )=  \vol(H^{2r_1+2-q}_{\free,\mp}(Y_{\eta_i};L_\varrho) )^{-1},
$$
we see from \eqref{eg.15} that
$$
\begin{aligned}
\vol(H^q_{\free,\pm}(Y_{\eta_i};L_\varrho)) &= \frac{1}{\vol(H^{2r_1+2-q}_{\free,\mp}(Y_{\eta_i};L_\varrho) )}  \\
     &\ge  \frac{1}{\vol(H^{2r_1+2-q}_{\free,\mp}(Y_{[\eta_i]};L_\varrho) ) [\Gamma(\mathfrak{n}_1)_{[\eta_i]}:\Gamma(\mathfrak{n}_i)_{\eta_i}]^{ \frac{b_{2r_1+2-q,\mp}}2   }  },
\end{aligned}
$$
so the result follows by noticing that $b_{2r_1+2-q,\mp}=b_{q,\pm}$ by Poincar\'e duality.

\end{proof}

We also need to control the number of cusp ends.  Proceeding as in \cite[Proposition~8.6]{MP2014}, this can be achieved as follows.
\begin{proposition}
The sequence \eqref{ideals} is such that 
\begin{equation}
    \lim_{i\to \infty} \frac{ \#\mathfrak{P}_{\Gamma(\mathfrak{n}_i)}+ \sum_{\eta_i\in \mathfrak{P}_{\Gamma(\mathfrak{n}_i)} }  \log[\Gamma(\mathfrak{n}_1)_{[\eta_i]}:\Gamma(\mathfrak{n}_i)_{\eta_i}] }{[\Gamma(\mathfrak{n}_1):\Gamma(\mathfrak{n}_i)]}=0.
\label{eg.24b}\end{equation}
\label{eg.26}\end{proposition}
\begin{proof}
Since $\Gamma(\mathfrak{n_i})$ is a normal subgroup of $\Gamma(\mathfrak{n}_1)$, notice that for $\eta\in \mathfrak{P}_{\Gamma(\mathfrak{n}_1)}$,
\begin{equation}
  \#\{ \eta_i\in \mathfrak{P}_{\Gamma(\mathfrak{n}_i)} \; | \;  \exists \gamma_i\in \Gamma(\mathfrak{n}_1) \; \mbox{such that} \; \eta_i= \gamma_i\eta\}= \# \{ \Gamma(\mathfrak{n}_i)\setminus\Gamma(\mathfrak{n}_1)/(\Gamma(\mathfrak{n}_1)\cap P_{\eta}) \}
\label{eg.27}\end{equation}
and 
\begin{equation}
  \# \{ \Gamma(\mathfrak{n}_i)\setminus\Gamma(\mathfrak{n}_1)/(\Gamma(\mathfrak{n}_1)\cap P_{\eta}) \}=  \frac{[\Gamma(\mathfrak{n}_1): \Gamma(\mathfrak{n}_i)]}{[\Gamma(\mathfrak{n}_1)\cap P_{\eta}: \Gamma(\mathfrak{n}_i)\cap P_{\eta}]},
\label{eg.27b}\end{equation}
so 
\begin{equation}
\begin{aligned}
\frac{\#\mathfrak{P}_{\Gamma(\mathfrak{n}_i)}}{[\Gamma(\mathfrak{n}_1): \Gamma(\mathfrak{n}_i)]} &= \sum_{\eta\in \mathfrak{P}_{\Gamma(\mathfrak{n}_1)}}  \frac{  \# \{ \Gamma(\mathfrak{n}_i)\setminus\Gamma(\mathfrak{n}_1)/(\Gamma(\mathfrak{n}_1)\cap P_{\eta}) \}  }{ [\Gamma(\mathfrak{n}_1): \Gamma(\mathfrak{n}_i)] } \\
&= \sum_{\eta\in \mathfrak{P}_{\Gamma(\mathfrak{n}_1)}}  \frac{1}{ [\Gamma(\mathfrak{n}_1)\cap P_{\eta}: \Gamma(\mathfrak{n}_i)\cap P_{\eta}] }  \longrightarrow 0 \quad \mbox{as} \; i\to\infty
\end{aligned}
\label{eg.28}\end{equation}
and
\begin{equation}
\begin{aligned}
\sum_{\eta_i\in \mathfrak{P}_{\Gamma(\mathfrak{n}_i)}} &  \frac{\log [\Gamma(\mathfrak{n}_1)_{[\eta_i]}: \Gamma(\mathfrak{n}_i)_{\eta_i}]}{[\Gamma(\mathfrak{n}_1): \Gamma(\mathfrak{n}_i)]}  \\ 
  &= \sum_{\eta\in \mathfrak{P}_{\Gamma(\mathfrak{n}_1)}}  \frac{ \#\{ \Gamma(\mathfrak{n}_i)\setminus \Gamma(\mathfrak{n}_1)/(\Gamma(\mathfrak{n}_1)\cap P_{\eta})  \}     \log [\Gamma(\mathfrak{n}_1)\cap P_{\eta}: \Gamma(\mathfrak{n}_i)\cap P_{\eta}]}{[\Gamma(\mathfrak{n}_1): \Gamma(\mathfrak{n}_i)]}  \\
 &= \sum_{\eta\in \mathfrak{P}_{\Gamma(\mathfrak{n}_1)}}  \frac{    \log [\Gamma(\mathfrak{n}_1)\cap P_{\eta}: \Gamma(\mathfrak{n}_i)\cap P_{\eta}]}{[\Gamma(\mathfrak{n}_1)\cap P_{\eta}: \Gamma(\mathfrak{n}_i)\cap P_{\eta}]} \; \longrightarrow 0 \quad \mbox{as} \; i\to\infty.
\end{aligned}
\label{eg.29}\end{equation}

\end{proof}

Now, consider the long exact sequence in cohomology
\begin{equation}
\xymatrix{
\cdots \ar[r]^-{\pa}& H^q(\bX_i,\pa \bX_i;L_\varrho) \ar[r] & H^q(\bX_i;L_\varrho) \ar[r] & H^q(\pa \bX_i; L_\varrho) \ar[r]^-{\pa} & \cdots
}
\label{eg.16}\end{equation}
associated to the pair $(\bX_i,\pa\bX_i)$, as well as its version tensored over the reals 
\begin{equation}
\xymatrix{
\cdots \ar[r]^-{\pa}& H^q(\bX_i,\pa\bX_i;\cE_\varrho) \ar[r] & H^q(\bX_i;\cE_\varrho) \ar[r] & H^q(\pa \bX_i; \cE_\varrho) \ar[r]^-{\pa} & \cdots .
}
\label{eg.17}\end{equation}

The choices of orthonormal bases in \S~\ref{sev.0} for the bundles $E_{m,n}$ induce corresponding bases for the cohomology groups of the bundle $\cE_\varrho$.  Trusting this will lead to no confusion, we will still denote by $\mu_{X_i}$, $\mu_{\pa X_i}=(\mu_{i,+},\mu_{i,-})$, $\mu_{i,+}$ and $\mu_{i,-}$ the bases we obtain in this way for $H^*(\bX_i;\cE_\varrho)$, $H^*(\pa \bX_i;\cE_\varrho)$, $H^*_+(\pa \bX_i;\cE_\varrho)$ and  $H^*_-(\pa \bX_i;\cE_\varrho)$.
By Poincar\'e-Lefshetz duality \cite{Milnor1962}, we know that
\begin{equation}
   \tau(\bX_i,\cE_\varrho,\mu_{X_i})=\tau(\bX_i,\pa\bX_i,\cE_\varrho,\mu_{\bX_i,\pa\bX_i}),
\label{eg.17b}\end{equation}
where $\mu_{X_i,\pa X_i}$ is the basis of $H^*(\bX_i,\pa \bX_i;\cE_\varrho)$ dual to $\mu_{X_i}$.  As in \cite[Lemma~4.1]{MR2}, the basis $\mu_{X_i,\pa X_i}$ admits a simpler description compared to $\mu_{X_i}$.  
\begin{lemma}
In terms of the boundary homomorphism of the long exact sequence \eqref{eg.17}, 
$$
    \mu_{X_i,\pa X_i}= \pa (\mu_{i,+}).
$$
\label{eg.18}\end{lemma}
\begin{proof}
Since Poincar\'e duality induces the identification
$$
      (H^*_{\pm}(\pa \bX_i;\cE_\varrho))^*\cong H^*_{\mp}(\pa \bX_i;\cE_\varrho),
$$
the result follows by observing that the map $\pr_{i,-}\circ \iota_{\pa \bX_i}: H^*(\bX_i;\cE_\varrho)\to H^*_-(\pa \bX_i;\cE_\varrho)$ is an isomorphism (this is in agreement with \cite[Theorem~4.6.3]{Harder}), where
$$
    \pr_{i,-}: H^*(\pa \bX_i;\cE_\varrho)\to H^*_{-}(\pa \bX_i;\cE_\varrho)
$$
is the orthogonal projection induced by the orthogonal decomposition \eqref{eg.12a}.
\end{proof}

By \cite[(1.4)]{Cheeger1979}, we know that
\begin{equation}
\tau(\bX_i,\pa \bX_i,\cE_\varrho,\mu_{X_i,\pa X_i})^2= \prod_q \lrp{ \frac{|H^q_{\tor}(\bX_i,\pa \bX_i;L_\varrho)|}{\vol_{\mu^{\bbR}_{X_i,\pa X_i}}(H^q_{\free}(\bX_i,\pa\bX_i;L_\varrho))}    }^{(-1)^{q+1}},
\label{eg.19}\end{equation}
where $\mu^{\bbR}_{X_i,\pa X_i}= \{ \mu_{X_i,\pa X_i}, \sqrt{-1} \mu_{X_i,\pa X_i}\}$ is the real basis associated to the complex basis $\mu_{X_i,\pa X_i}$ and $\vol_{\mu^{\bbR}_{X_i,\pa X_i}}(H^q_{\free}(\bX_i,\pa\bX_i;L_\varrho))$ is the covolume of the lattice $H^q_{\free}(\bX_i,\pa\bX_i;L_\varrho)$ in $H^q(\bX_i,\pa\bX_i;\cE_\varrho)$ with respect to the basis $\mu^{\bbR}_{X_i,\pa X_i}$.  Since $\mu_{X_i,\pa X_i}= \pa (\mu_{i,+})$, this covolume can be rewritten
\begin{equation}
\vol_{\mu^{\bbR}_{X_i,\pa X_i}}(H^q_{\free}(\bX_i,\pa\bX_i;L_\varrho))= \frac{  \vol_{\mu^{\bbR}_{i,+}}( H^{q-1}_{\free,+}(\pa \bX_i;L_\varrho)  )      }{ [H^q_{\free}(\bX_i,\pa \bX_i;L_\varrho):\pa H^{q-1}_{\free,+}(\pa \bX_i;L_\varrho)] } 
\label{eg.20}\end{equation}
with $\mu^{\bbR}_{i,+}=\{\mu_{i,+}, \sqrt{-1}\mu_{i,+}\}$ the real basis associated to the complex basis $\mu_{i,+}$.  On the other hand, from the long exact sequence \eqref{eg.16}, we see that
$$
\begin{aligned}
1\le [H^q_{\free}(\bX_i,\pa \bX_i;L_\varrho):\pa H^{q-1}_{\free,+}(\pa \bX_i;L_\varrho)] &\le |H^q_{\tor}(\bX_i;L_\varrho)|  \\
    &= |H^{2r_1+4-q}_{\tor}(\bX_i,\pa \bX_i;L_\varrho)|,
\end{aligned}
$$
where in in the second line we have used the fact that
\begin{equation}
  H^q_{\tor}(\bX_i;L_\varrho)\cong H_{q-1}(\bX_i;L_\varrho)_{\tor}\cong H^{2r_1+4-q}_{\tor}(\bX_i,\pa \bX_i, L_\varrho),
\label{eg.21}\end{equation}
which is itself a consequence of the universal coefficient theorem and Poincar\'e-Lefshetz duality.  Consequently, we deduce that 
\begin{equation}
  \frac{  \vol_{\mu^{\bbR}_{i,+}}( H^{q-1}_{\free,+}(\pa \bX_i;L_\varrho)) }{ | H^{2r_1+4-q}_{\tor}(\bX_i,\pa \bX_i, L_\varrho) |  } \le \vol_{\mu^{\bbR}_{X_i,\pa X_i}}(H^q_{\free}(\bX_i,\pa\bX_i;L_\varrho)).
\label{eg.22}\end{equation}
Therefore, using \eqref{eg.17b} and inserting the above inequalities in \eqref{eg.19} using \eqref{eg.21} and the fact that $| H^{2r_1+4-q}_{\tor}(\bX_i,\pa \bX_i, L_\varrho) | \ge 1$, we see that
\begin{equation}
\begin{aligned}
 \tau(\bX_i,\cE_\varrho,\mu_{X_i})^{2(-1)^{r_1+1}}  &=\tau(\bX,\pa\bX,\cE_\varrho,\mu_{\bX_i,\pa\bX_i})^{2(-1)^{r_1+1}} \\
    &\le  \frac{  \lrp{ \prod_{q+r_1 \even} | H^q_{\tor}(\bX_i,\pa \bX_i;L_\varrho)  |^2   }    }{   \lrp{  \prod_q \vol_{\mu_{i,+}^{\bbR}}( H^{q-1}_{\free,+}(\pa \bX_i;L_\varrho)  )^{(-1)^{q+r_1}}     }  }.
\end{aligned}
\label{eg.23}\end{equation}

Combined with Theorem~\ref{se.21} and Proposition~\ref{prop-appr-l2-tor}, this inequality yields the following result, namely Theorem~\ref{int.4} when the flat vector bundle $\cE_{\varrho}$ is not necessarily acyclic.
\begin{theorem}
Let $\bbK$ be a number field with $r_2=1$ and $r_1>1$.  Let \eqref{qrr.1} be a $\bbQ$-rational representation and assume that $\varrho_{\infty}$ decompose into a sum of irreducible representations that are not self-conjugate.  Let $\Lambda_\varrho\subset V_\varrho$ be a $\SL(2,\cO_{\bbK})$-invariant lattice.  If the natural isomorphism $V_{\varrho}^*\cong V_{\varrho}$ induces an isomorphism $\Lambda_{\varrho}^*\cong \Lambda_{\varrho}$, then 
\begin{equation}
  \liminf_{i\to\infty} \frac{ \sum_{q+r_1 \even} \log |H^q_{\tor}(\Gamma(\mathfrak{n}_i);\Lambda_{\varrho})|   }{ [\Gamma(\mathfrak{n}_1):\Gamma(\mathfrak{n}_i)] } \ge (-1)^{r_1+1}t^{(2)}_{\tX}(\varrho_\infty)\vol(X_1)>0.
\label{eg.24c}\end{equation}
\label{eg.24}\end{theorem}
\begin{proof}
As in the proof of Theorem~\ref{eg.9}, we know that \eqref{appr-l2-tor-1} holds, while by \cite[Proposition~5.2]{BV}
$$
(-1)^{r_1+1}t^{(2)}_{\tX}(\varrho_\infty)\vol(X_1)>0.
$$
On the other hand, combining Proposition~\ref{eg.14} with Proposition~\ref{eg.26}, we see that
\begin{equation}
  \lim_{i\to \infty}  \frac{  \sum_{q} (-1)^{q+r_1} \log(\vol_{\mu^{\bbR}_{i,+}}( H^q_{\free,+}(\pa \bX_i;L_\varrho)))  }{ [\Gamma(\mathfrak{n}_1):\Gamma(\mathfrak{n}_i)] }=0.
  \label{eg.25}\end{equation}
We note that by \eqref{cdat.2b} and Lemmas \ref{dec.5e} and \ref{dec.5h} the
exponent $b_{q,\pm}$ occuring in \eqref{estim7} does not depend on $i$.
Hence, \eqref{eg.24c} follows by combining \eqref{eg.23} with Theorem~\ref{se.21} together with \eqref{appr-l2-tor-1} and \eqref{eg.25}.

\end{proof}

\bibliographystyle{amsalpha}
\bibliography{torsion_fc}

\end{document}